\documentclass[a4paper, 12pt, oneside]{article}
\usepackage[table]{xcolor} % To colour alternate rows in tables; must be loaded before tikz
\pdfoutput=1 % For arXiv to use pdflatex, for package microtype
\usepackage[T1]{fontenc}
\usepackage[british]{babel}
\usepackage{lmodern}
\usepackage{etoolbox}
\usepackage{amsmath, amsthm, amssymb, bm, mleftright}
\usepackage{booktabs, longtable, caption, subcaption, multirow}
\usepackage[space]{grffile}
\usepackage[margin=2.5cm, footskip=1cm]{geometry}
\usepackage{enumitem}
\setenumerate[1]{label=(\arabic*), ref=\arabic*}
\usepackage[protrusion=allmath]{microtype}
\usepackage{tikz}
\usetikzlibrary{cd, calc}
\usepackage{listings}
\usepackage[colorlinks, citecolor=blue, linkcolor=blue, linktoc=section]{hyperref}

\allowdisplaybreaks
\usepackage[capitalise, noabbrev]{cleveref}
\creflabelformat{enumi}{(#2#1#3)}
\creflabelformat{enumii}{(#2#1#3)}
\creflabelformat{enumiii}{(#2#1#3)}

\usepackage{titlesec}
\titleformat*{\section}{\large\bfseries}
\titleformat*{\subsection}{\normalsize\bfseries}
\newlength{\VerticalSpaceAfterParagraph}
\titlespacing*{\paragraph}{0pt}{\VerticalSpaceAfterParagraph}{1em}

\usepackage{titling}
\pretitle{\vspace{-\baselineskip}\begin{center}\Large\bfseries}
\posttitle{\end{center}\vspace{-0.25\baselineskip}}
\preauthor{\begin{center}}
\postauthor{\end{center}}
\predate{\begin{center}}
\postdate{\end{center}}

\setlist
  {
    topsep = 5.0pt plus 2.0pt minus 3.0pt,
    partopsep = 1.5pt plus 1.0pt minus 1.0pt,
    parsep = 2.5pt plus 1.25pt minus 0.5pt,
    itemsep = 0pt plus 1.25pt minus 0.5pt
  }

\theoremstyle{plain}
\newtheorem{theorem}{Theorem}
\newtheorem*{theorem*}{Theorem}
\newtheorem{proposition}[theorem]{Proposition}
\newtheorem{lemma}[theorem]{Lemma}
\newtheorem{corollary}[theorem]{Corollary}

\theoremstyle{definition}
\newtheorem{definition}[theorem]{Definition}
\newtheorem{example}[theorem]{Example}
\newtheorem{notation}[theorem]{Notation}
\newtheorem{construction}[theorem]{Construction}

\theoremstyle{remark}
\newtheorem{remark}[theorem]{Remark}

\usepgfmodule{oo}

\DeclareRobustCommand\ShowAuthors[2]{%
  \ShowAuthorsSignal.emit({#1},{#2})%
}

\pgfoonew \ShowAuthorsSignal=new signal()

\DeclareRobustCommand\ShowAffiliations[1]{%
  \ShowAffiliationsSignal.emit(#1)%
}

\pgfoonew \ShowAffiliationsSignal=new signal()

\newcommand\Author[1]{
  \pgfoonew \CurrentPerson=new person()
  \CurrentPerson.set author(#1)
}

\newcommand\Email[1]{
  \CurrentPerson.set email(#1)
}

\newcommand\Address[1]{
  \CurrentPerson.set address(#1)
}

\newcommand\FirstPerson{0}
\newcommand\LastPerson{}

\pgfooclass{person}{
  \attribute author;
  \attribute email;
  \attribute address;

  \method person() {
    \pgfoothis.get id(\theid)

    \ifnum \FirstPerson=0
      \edef\FirstPerson{\theid}
    \fi

    \edef\LastPerson{\theid}

    \pgfoothis.get handle(\me)
    \ShowAuthorsSignal.connect(\me,show author)
    \ShowAffiliationsSignal.connect(\me,show affiliation)
  }

  \method get author(#1) {
    \pgfooget{author}{#1}
  }

  \method set author(#1) {
    \pgfooset{author}{#1}
  }

  \method get email(#1) {
    \pgfooget{email}{#1}
  }

  \method set email(#1) {
    \pgfooset{email}{#1}
  }

  \method get address(#1) {
    \pgfooget{address}{#1}
  }

  \method set address(#1) {
    \pgfooset{address}{#1}
  }

  \method show author(#1,#2) {%
    \pgfoothis.get id(\theid)%
    \pgfooget{author}{\theauthor}%
    \ifnum\theid=\FirstPerson%
    \else%
      \ifnum\theid=\LastPerson%
        #2%
      \else%
        #1%
      \fi%
    \fi%
    \mbox{\theauthor}%
  }

  \method show affiliation(#1) {%
    \pgfoothis.get id(\theid)%
    \pgfooget{author}{\theauthor}%
    \pgfooget{email}{\theemail}%
    \pgfooget{address}{\theaddress}%
    \ifnum\theid=\FirstPerson%
    \else%
      #1%
    \fi%
    \noindent\begin{minipage}{\linewidth}
      \noindent\begin{tabular}[t]{@{}l}
        \theauthor \quad \texttt{\theemail}\\
      \end{tabular}\\
      \theaddress%
    \end{minipage}%
  }
}

\newcommand\blfootnote[1]
  {%
    \begingroup%
      \renewcommand\thefootnote{}%
      \footnote{#1}%
      \addtocounter{footnote}{-1}%
    \endgroup%
  }

\newcommand\BigParenthesis[1]{\multirow{2}{*}{$\raisebox{0.56em}{#1}$}}
\newcommand*\Ta[1]{\rule{0pt}{2.5ex} \BigParenthesis{$#1 \Big($} \!\!\!\!\!}
\newcommand*\Tb[1]{\!\!\!\! \BigParenthesis{$\Big) #1$}}

\newcommand*\abs[1]{\left\lvert #1 \right\rvert}
\newcommand*\smat[1]{\left(\begin{smallmatrix}#1\end{smallmatrix}\right)}
\newcommand*\pmat[1]{\begin{pmatrix}#1\end{pmatrix}}

\crefname{maintheoremAux}{Theorem}{Theorems}

\theoremstyle{plain}
\newtheorem{maintheoremAux}{Theorem}

\newenvironment{maintheorem}
  {\maintheoremAux\addcontentsline{toc}{subsubsection}{Theorem~\themaintheoremAux{}}}
  {\endmaintheoremAux}
\renewcommand\themaintheoremAux{\Alph{maintheoremAux}}

\crefname{(part)}{part}{parts}
\Crefname{(part)}{Part}{Parts}
\creflabelformat{(part)}{(#2#1#3)}
\crefname{eqs}{Equations}{Equations}
\crefname{cond}{condition}{conditions}
\crefname{(cond)}{condition}{conditions}
\Crefname{(cond)}{Condition}{Conditions}
\creflabelformat{(cond)}{(#2#1#3)}
\crefname{(part)}{part}{parts}
\Crefname{(part)}{Part}{Parts}
\creflabelformat{(part)}{(#2#1#3)}

\crefname{divi}{the divisibility constraint}{the divisibility constraints}
\creflabelformat{divi}{(#2#1#3)}

\numberwithin{theorem}{section}
\numberwithin{equation}{section}

\Author{Erik Paemurru}
\Address{Mathematik und Informatik, Universität des Saarlandes, 66123 Saarbrücken, Germany.}
\Email{paemurru@math.uni-sb.de}

\newcommand\Thanks{The author was supported by the Engineering and Physical Sciences Research Council (grant 1820497) while in Loughborough University, and the London Mathematical Society Early Career Fellowship (grant ECF-1920-24) while in Imperial College London.}

\title{Birational geometry of sextic double solids with a compound $A_n$ singularity}
\author{\ShowAuthors{, }{, }}
\date{1st~July 2024}
\newcommand\keywords{Fano 3-folds, Sarkisov program, birational rigidity}
\newcommand\subjclass{14J45 (Primary) 14J30, 14J17, 14E30, 14E05 (Secondary)}

\begin{document}

\maketitle

\begin{abstract}
Sextic double solids, double covers of $\mathbb P^3$ branched along a sextic surface, are the lowest degree Gorenstein terminal Fano 3-folds, hence are expected to behave very rigidly in terms of birational geometry. Smooth sextic double solids, and those which are $\mathbb Q$-factorial with ordinary double points, are known to be birationally rigid. In this article, we study sextic double solids with an isolated compound $A_n$ singularity. We prove a sharp bound $n \leq 8$, describe models for each $n$ explicitly and prove that sextic double solids with $n > 3$ are birationally non-rigid.
\blfootnote{\textup{2020} \textit{Mathematics Subject Classification}. \subjclass{}.}%
\blfootnote{\Thanks{}}
\end{abstract}

\tableofcontents

\section{Introduction}

We work with projective varieties over~$\mathbb C$. Classification of algebraic varieties is one of the fundamental goals in algebraic geometry. The Minimal Model Program says that every variety is birational to either a minimal model or a Mori fibre space. Two Mori fibre spaces are birational if they are connected by a sequence of Sarkisov links (see \cite{Sar89}, \cite{Rei91}, \cite{Cor95}, \cite{HM13}). In the extreme case, the Mori fibre space is \emph{birationally rigid}, meaning that it is essentially the unique Mori fibre space in its birational class.

Examples of Mori fibre spaces include Fano varieties. The first birational rigidity result was in the seminal paper by Iskovskikh and Manin \cite{IM71} for smooth quartic 3-folds in~$\mathbb P^4$. A wealth of examples of birationally rigid varieties was given in \cite{CPR00} and \cite{CP17}, by showing that every quasismooth member of the 95 families of Fano 3-folds that are hypersurfaces in weighted projective spaces is birationally rigid. One major consequence of birational rigidity is nonrationality. Birational rigidity remains an active area of research (see \cite{Pro18}, \cite{Kry18}, \cite{AO18}, \cite{dF17}, \cite{CG17}, \cite{CS19}, \cite{EP18}).

Among smooth Fano 3-folds, the projective space has the highest degree~($64$), and sextic double solids, double covers of $\mathbb P^3$ branched along a sextic surface, have the least degree~($2$). In \cite{Isk80}, it is proved that smooth sextic double solids are birationally rigid. It is interesting to see how this changes as we impose singularities on the variety. The paper \cite{Puk97} proved that sextic double solids stay birationally rigid if we impose an ordinary double point, meaning the 3-fold $A_1$ singularity $x_1^2 + x_2^2 + x_3^2 + x_4^2$. A sextic double solid can have up to 65 singular points (see \cite{Bas06}, \cite{JR97}, \cite{Wah98}), and for each $n \leq 65$, there exists a sextic double solid with exactly $n$ ordinary double points and smooth otherwise (see \cite{Bar96}, \cite{CC82}). A sextic double solid with only ordinary double points is birationally rigid if and only if it is factorial, which is true for example if it has at most 14 ordinary double points (see \cite[Theorem~B]{CP10}).

The next natural question is to consider more complicated singularities in the Mori category. We study sextic double solids with an isolated \emph{compound $A_n$} singularity, also called a $cA_n$ singularity, meaning that the general section through the point is the Du Val $A_n$ singularity $x_1 x_2 + x_3^{n+1}$. A $cA_n$ singularity is locally analytically given by $x_1 x_2 + h(x_3, x_4)$ where the least degree among monomials in $h$ is~$n+1$. The first main result of the paper is describing sextic double solids with an isolated $cA_n$ singularity.
\begin{theorem*}[see \cref{mai:con}]
If a sextic double solid has an isolated $cA_n$ point, then $n \leq 8$.
\end{theorem*}
Moreover, in \cref{mai:con} we explicitly parametrize all sextic double solids with an isolated $cA_n$ singularity for every $n \leq 8$. These form $11$ families, as there are 4 families for $cA_7$. Every family except for family~7.4 contains members that are Mori fibre spaces over a point.

We say a few words on bounding the number of $cA_n$ singularities. It is clear that an isolated $cA_n$ singularity has Milnor number at least~$n^2$. Since the third Betti number of a smooth sextic double solid is $104$ (see \cite[Table~12.2]{IP99}), an argument similar to \cite[Section~3.2]{AK16} shows that the total Milnor number of a sextic double solid which is a Mori fibre space is at most $104$. This gives the bounds that a Mori fibre space sextic double solid can have up to 1 $cA_8$ singularity, or up to 2 $cA_7$ singularities, or up to 2 $cA_6$ singularities, \ldots, or up to 26 $cA_2$ singularities. We do not expect these bounds to be sharp, as already for ordinary double points it gives an upper bound of~104, far from the actual~65. Using \cref{mai:con}, it is possible to construct sextic double solids with a $cA_8$ point, a $cA_3$ point and two ordinary double points with both total Milnor and total Tjurina number at least~66.

The second main result is the following theorem:
\begin{theorem*}[see \cref{mai:mod} and \cref{thm:mod cA5}]
A general sextic double solid which is a Mori fibre space with an isolated $cA_n$ singularity where $n \geq 4$ is not birationally rigid.
\end{theorem*}
Birational non-rigidity for a sextic double solid $X$ is proved by describing a birational model, meaning a Mori fibre space $T \to S$ such that $X$ and $T$ are birational. We find the birational models by explicitly constructing a Sarkisov link for each family of sextic double solids, under the generality conditions described in \cref{def:mod generality conds}. \Cref{tab:SDS bir mod overview} gives an overview of the Sarkisov links $X \gets Y_0 \dashrightarrow Y_2 \to Z$ and the birational models, which are either fibrations $Y_2 \to Z$ or Fano varieties $Z$. In the latter case, $Y_2 \to Z$ is a divisorial contraction to the given singular point. The morphism $Y_0 \to X$ is a divisorial contraction with centre the $cA_n$ point. The birational maps $Y_0 \dashrightarrow Y_2$ are isomorphisms in codimension~1.

\newcommand\spacing{\vspace{4pt}}
\newcommand\ali[1]{\begin{aligned}#1\end{aligned}}
\newcommand\stext[1]{\begin{minipage}{.12\textwidth} \spacing #1 \spacing \end{minipage}}
\newcommand\smtext[1]{\begin{minipage}{.18\textwidth} \spacing #1 \spacing \end{minipage}}
\newcommand\mtext[1]{\begin{minipage}{.25\textwidth} \spacing #1 \spacing \end{minipage}}
\begin{table}[h]
\centering
\caption{Birational models for general sextic double solids that are Mori fibre spaces with an isolated $cA_n$ singularity\label{tab:SDS bir mod overview}}
\newcommand\negspace{\hspace{-0.1\textwidth}}

\begin{equation*}
\begin{tikzcd}
  & \arrow[ld, "\varphi"'] Y_0 \arrow[r, dashed] & Y_2 \arrow[rd, "\psi"] & \\
cA_n \in X_6 \subseteq \mathbb P(1^4, 3) \negspace & & & Z
\end{tikzcd}
\end{equation*}

\medskip

\rowcolors{2}{black!5}{white}
\begin{tabular}{ccccc}
\toprule
$cA_n$ & \stext{weighted\\blowup $\varphi$} & \begin{tikzcd}{} \arrow[r, dashed] & {}\end{tikzcd} & \stext{weighted\\blowup or\\fibration $\psi$\par} & $Z$\\
\midrule
\hyperref[subsec:mod cA4]{$cA_4$} & $(3, 2, 1, 1)$ & \mtext{$10$ Atiyah flops} & $(\frac14, \frac14, \frac34)$ & \smtext{$\frac14(1, 1, 3) \in Z_{5, 6}$\par\noindent$\subseteq \mathbb P(1^3, 2, 3, 4)$}\\
\hyperref[subsec:mod cA5]{$cA_5$} & $(3, 3, 1, 1)$ & \mtext{$4$ Atiyah flops} & $(3, 3, 1, 1)$ & \mtext{$cA_5 \in Z_6 \subseteq \mathbb P(1^4, 3)$, \par\noindent $X \ncong Z$ if general}\\
\hyperref[subsec:mod cA6]{$cA_6$} & $(4, 3, 1, 1)$ & \mtext{$2$ Atiyah flops, then\par\noindent$(4, 1, 1, -2, -1; 2)$-flip\par} & $(3, 1, 1, 1)$ & $cA_3 \in Z_5 \subseteq \mathbb P(1^4, 2)$\\
\hyperref[subsec:mod cA7-1]{$cA_7$, 1} & $(4, 4, 1, 1)$ & \mtext{two $(4, 1, 1, -2, -1; 2)$-\par\noindent flips} & $(1, 1, 1, 1)$ & $\operatorname{ODP} \in Z_{3, 4} \subseteq \mathbb P(1^4, 2^2)$\\
\hyperref[subsec:mod cA7-2]{$cA_7$, 2} & $(4, 4, 1, 1)$ & \mtext{Atiyah flop, then\par\noindent two $(4, 1, -1, -3)$-flips\par} & $(3, 3, 2, 1)$ & $cA_2 \in Z_{2, 4} \subseteq \mathbb P(1^5, 2)$\\
\hyperref[subsec:mod cA7-3]{$cA_7$, 3} & $(4, 4, 1, 1)$ & \mtext{$2$ Atiyah flops} & $\operatorname{dP_2}$-fibration & $\mathbb P^1$\\
\hyperref[subsec:mod cA8]{$cA_8$} & $(5, 4, 1, 1)$ & \mtext{$(4, 1, 1, -2, -1; 2)$-flip\par} & $(3, 2, 2, 1, 5)$ & $cD_4 \in Z_{3, 3} \subseteq \mathbb P(1^5, 2)$\\
\bottomrule
\end{tabular}
\end{table}

Note that when we say that a birational map $Y_0 \dashrightarrow Y_1$ is $k$ Atiyah flops, then we mean that algebraically it is one flop, contracting $k$ curves to $k$ points and extracting $k$ curves, and locally analytically around each of those points, it is an Atiyah flop. Similarly for flips. Also note that the Sarkisov link to a sextic double solid with a $cA_4$ singularity was already described in \cite[Section~9, No.~9]{Oka14}, starting from a general quasismooth complete intersection $X_{5, 6} \subseteq \mathbb P(1, 1, 1, 2, 3, 4)$.

We briefly describe the proof. The first step in the Sarkisov link starting from a Fano variety $X$ is a divisorial contraction $Y \to X$. Kawakita described divisorial contractions to $cA_n$ points locally analytically, showing that they are certain weighted blowups. To construct Sarkisov links, we need a global description. In \cref{thm:wei cAn wt correct implies Kawakita blup,thm:wei cAn p v}, we show how to construct divisorial contractions to $cA_n$ points algebraically on affine hypersurfaces, and use this in \cref{sec:mod} to construct divisorial contractions $Y \to X$ for (projective) sextic double solids~$X$. Using unprojection techniques (see \cite{PR04} for a general theory of unprojection), we find an embedding of $Y$ inside a toric variety~$T$, such that the 2-ray link of $T$ restricts to a Sarkisov link for~$X$ (following \cite{BZ10} and \cite{AZ16}).

If we try the same methods as in the proof of \cref{mai:mod} on sextic double solids with a $cA_n$ singularity where $n \leq 3$, then we do not find any new birational models. More precisely: a $(3, 1, 1, 1)$-Kawakita blowup of a $cA_3$ singularity on a general Mori fibre space sextic double solid initiates a Sarkisov link to itself $X \dashrightarrow X$. A $(2, 2, 1, 1)$-Kawakita blowup for a~$cA_3$ singularity, a $(2, 1, 1, 1)$-Kawakita blowup for a $x_1 x_2 + x_3^3 + x_4^3$ singularity and the (usual) blowup for an ordinary double point on a general Mori fibre space sextic double solid initiate `bad links', which end in either a non-terminal 3-fold or a K3-fibration. These are 2-ray links which are not Sarkisov links, where in the last step of the 2-ray game only $K$-trivial curves are contracted, leaving the Mori category. We expect that general Mori fibre space sextic double solids with a $cA_3$ singularity are birationally rigid, and with certain $cA_2$ or $cA_1$ singularities are birationally superrigid.

\subsection*{Organization of the paper}
In \cref{sec:pre singularities,,sec:pre div contr,,sec:pre sark links}, we give known results that we use respectively in \cref{sec:constructing sds with cAn,,sec:wei,,sec:mod}.
In \cref{sec:constructing sds with cAn}, we construct a parameter space of sextic double solids in \cref{mai:con} with an isolated $cA_n$ singularity.
In \cref{sec:wei}, we explain the relationship between algebraic and local analytic weighted blowups, and in \cref{thm:wei cAn wt correct implies Kawakita blup} and the technical \cref{thm:wei cAn p v}, we show how to construct divisorial contractions to $cA_n$ points algebraically on affine hypersurfaces.
In \cref{sec:mod}, we construct birational models for general sextic double solids which are Mori fibre spaces with an isolated $cA_n$ singularity where $n \geq 4$, thereby showing that they are not birationally rigid. We treat the 7 families separately.

\section{Preliminaries} \label{sec:Preliminaries}

An algebraic variety is an integral separated scheme of finite type over the complex numbers~$\mathbb C$. When we say `morphism', we mean a morphism over~$\mathbb C$.

We study sextic double solids, which are double covers of the projective $3$-space branched along a sextic surface. We use the following equivalent characterization:

\begin{definition}
A \textbf{sextic double solid} is the variety given by the zero locus of an irreducible polynomial $w^2 + g$ in the weighted projective space $\mathbb P(1, 1, 1, 1, 3)$ with variables $x, y, z, t, w$, where $g \in \mathbb C[x, y, z, t]$ is homogeneous of degree six.
\end{definition}

\subsection{Singularity theory} \label{sec:pre singularities}

We recall some results from the singularity theory of complex analytic spaces and terminal singularities.

We denote the variables on $\mathbb C^n$ by $\bm x = (x_1, \ldots, x_n)$, where $n$ is a positive integer. Let $\mathbb C\{\bm x\}$ denote the convergent power series ring. The zero set of an ideal $I \subseteq \mathbb C\{\bm x\}$ is denoted by $\mathbb V(I)$, where $I$ is either an ideal of regular functions or holomorphic functions, depending on the context. Given a point $P \in \mathbb V(I)$, the pair $(\mathbb V(I), P)$ denotes the (possibly reducible or non-reduced) complex space subgerm of $(\mathbb C^n, P)$ given by~$I$. Given a regular or holomorphic function $f$ on a variety or a complex space~$X$, denote the non-zero locus of $f$ by~$X_f$. Given positive integer weights $\bm w = (w_1, \ldots, w_n)$ for~$\bm x$, we can write a non-zero polynomial or power series $f$ as a sum of its weighted homogeneous parts~$f_i$. Then, the \textbf{weight} of $f$, denoted $\operatorname{wt}(f)$, is the least non-negative integer $d$ such that $f_d \neq 0$.
We define $\operatorname{wt}(0) = \infty$.
If $\bm w = (1, \ldots, 1)$, then $d$ is called the \textbf{multiplicity}, denoted $\operatorname{mult}(f)$. A \textbf{hypersurface singularity} is a complex analytic space germ (not necessarily irreducible or reduced) that is isomorphic to $(\mathbb V(f), \bm 0)$ for some $f \in \mathbb C\{\bm x\}$. A singularity is \textbf{isolated} if it has a smooth analytic punctured neighbourhood.

\begin{definition}[{\cite[Definition~2.9]{GLS07}}]
Let $f, g \in \mathbb C\{\bm x\}$.
\begin{enumerate}[label=(\alph*)]
\item We say $f$ and $g$ are \textbf{right equivalent} if there exists a biholomorphic map germ $\varphi\colon (\mathbb C^n, \bm 0) \to (\mathbb C^n, \bm 0)$ such that $g = f \circ \varphi$.
\item We say $f$ and $g$ are \textbf{contact equivalent} if there exists a biholomorphic map germ $\varphi\colon (\mathbb C^n, \bm 0) \to (\mathbb C^n, \bm 0)$ and a unit $u \in \mathbb C\{x_1, \ldots, x_n\}$ such that $g = u(f \circ \varphi)$.
\end{enumerate}
\end{definition}

\begin{remark}[{\cite[Remark~2.9.1(3)]{GLS07}}]
Two convergent power series $f, g \in \mathbb C\{\bm x\}$ are contact equivalent if and only if the complex analytic space germs $(\mathbb V(f), \bm 0)$ and $(\mathbb V(g), \bm 0)$ are isomorphic.
\end{remark}

We use the following proposition in \cref{sec:constructing sds with cAn} to parametrize sextic double solids with a $cA_1$ singularity:

\begin{proposition}[{\cite[Remark~2.50.1]{GLS07}}] \label{thm:ana lowest degree parts}
Let $f, g \in \mathbb C\{x_1, \ldots, x_n\}$ be two contact equivalent power series with zero constant term. Then their multiplicity $m$ (as defined above) is the same and furthermore, $f_m$ and $g_m$ are the same up to an invertible linear change of coordinates.
\end{proposition}

We use the following proposition in \cref{sec:constructing sds with cAn} to construct sextic double solids with a $cA_n$ singularity where $n \geq 2$, as well as in \cref{sec:wei} to describe weighted blowups of $cA_n$ points:

\begin{proposition} \label{thm:ana stable equivalence}
Let $F = x_1^2 + \ldots + x_k^2 + f$ and $G = x_1^2 + \ldots + x_k^2 + g$, where $f$ and $g$ are convergent power series in $\mathbb C\{x_{k+1}, \ldots, x_n\}$ with zero constant term. Then $F$ and $G$ are contact (respectively, right) equivalent if and only if $f$ and $g$ are contact (respectively, right) equivalent.
\end{proposition}

\begin{proof}
By a result of Mather and Yau \cite{MY82} (see also \cite[Theorem~2.26]{GLS07}), $f$ and $g$ are contact equivalent if and only the Tjurina algebras $T_f$ and $T_g$ are isomorphic. A simple computation shows that $T_f \cong T_F$ and $T_g \cong T_G$, which proves the proposition for contact equivalence.

The proof for right equivalence is similar. Namely, we use a statement analogous to \cite{MY82}: two elements $h, k \in \mathbb C\{\bm x\}$ with zero constant term are right equivalent if and only if the Milnor algebras $M_h$ and $M_k$ are isomorphic as algebras over the ring $\mathbb C\{t\}$, where $t$ acts on $M_h$, respectively $M_k$, by multiplying by $h$, respectively $k$ (see \cite[Theorem~2.28]{GLS07}).
\end{proof}

Reid defined in \cite[Definition~2.1]{Rei80} that a \textbf{compound Du Val singularity} is a $3$-dimensional singularity where a hypersurface section is a Du Val singularity, also called a surface ADE singularity. The singularity is denoted $cA_n$, $cD_n$ or $cE_n$, respectively, if the general hyperplane section is an $A_n$, $D_n$ or $E_n$ singularity, respectively. Reid showed in \cite[Theorem~0.6]{Rei83} that a $3$-dimensional hypersurface singularity is terminal if and only if it is an isolated compound Du Val singularity.

In this paper, we focus on the most general class of compound Du Val singularities, namely $cA_n$ singularities. Since a surface $A_n$ singularity is given by $x^2 + y^2 + z^{n+1}$, we have the following corollary:

\begin{corollary} \label{thm:pre cAn equation}
Let $n$ be a positive integer. A singularity is of type $cA_n$ if and only if it is isomorphic to the complex analytic space subgerm $(\mathbb V(x_1^2 + x_2^2 + g), \bm 0)$ of $(\mathbb C^4, \bm 0)$ with variables $x_1, x_2, x_3, x_4$ for some convergent power series $g \in \mathbb C\{x_3, x_4\}$ of multiplicity~$n+1$.
\end{corollary}

For a proof of \cref{thm:pre cAn equation}, see \cite[Theorem~2.8]{Kol98}.

The simplest example of a $cA_1$ singularity is the \emph{ordinary double point}, given by $x^2 + y^2 + z^2 + t^2$.

\begin{remark}
Terminal sextic double solids have only isolated hypersurface singularities, therefore only $cA_n$, $cD_n$ and $cE_n$ singularities. Sextic double solids are Gorenstein, since by \cite[Corollary~21.19]{Eis95} every variety with local complete intersection singularities is Gorenstein.
\end{remark}

\subsection{\texorpdfstring
{$\mathbb Q$-factoriality}
{Q-factoriality}}

\begin{definition}
A Weil divisor $D$ on a normal algebraic variety is \textbf{$\mathbb Q$-Cartier} if a positive integer multiple of $D$ is Cartier. A normal algebraic variety $X$ is \textbf{factorial}, respectively \textbf{$\mathbb Q$-factorial}, if every Weil divisor on $X$ is Cartier, respectively $\mathbb Q$-Cartier.
\end{definition}

\begin{definition}
A \textbf{Fano variety} is a normal projective algebraic variety with an ample $\mathbb Q$-Cartier anti-canonical divisor.
\end{definition}

To prove factoriality of certain singular sextic double solids, we use the following proposition by Namikawa.

\begin{proposition}[{\cite[Proposition~2]{Nam97}}] \label{thm:int Nam97 Pic}
Let $X$ be a Fano 3-fold with Gorenstein terminal singularities and $D$ its general effective anti-canonical divisor. Then, the natural homomorphism $\operatorname{Pic}(X) \to \operatorname{Pic}(D)$ is an injection.
\end{proposition}

\begin{remark} \label{rem:int Nam97 typos}
The proof of \cite[Proposition~2]{Nam97} contains a few typos that do not affect the result:
\begin{enumerate}
\item \label{itm:con injection from Pic X to Pic U} The sentence ``Since $\operatorname{Pic}(X) \cong \operatorname{Pic}(U)$, we have shown that\ldots'' should be replaced with ``Since $\operatorname{Pic}(X)$ injects into $\operatorname{Pic}(U)$, we have shown that\ldots''. The isomorphism of $\operatorname{Pic}(X)$ and $\operatorname{Pic}(U)$ would imply that every $X$ that is smooth along $D$ is factorial, which is not true. To see that $\operatorname{Pic}(X)$ injects into $\operatorname{Pic}(U)$ for every Zariski open set $U$ containing $D$, note that since the complement of $U$ in $X$ is of codimension at least~$2$, the class groups $\operatorname{Cl}(X)$ and $\operatorname{Cl}(U)$ are isomorphic. We have a map $\operatorname{Pic}(X) \to \operatorname{Pic}(U)$, since every Weil divisor which is Cartier on $X$ is Cartier on~$U$. The map $\operatorname{Pic}(X) \to \operatorname{Pic}(U)$ is injective.
\item The sentence ``Thus, the complement $X - U$ is of codimension $2$ in $X$'' should be replaced with ``Thus, the complement $X - U$ is of codimension at least $2$ in $X$''.
\item The sentence ``There is a Zariski open subset $U$ of $W$\ldots'' should be replaced with ``There is a Zariski open subset $U$ of $X$\ldots''.
\end{enumerate}
\end{remark}

We remind that a terminal variety is log terminal, see \cite[Definition~2.34]{KM98}. The Picard number of log terminal sextic double solids is $1$ by the following \lcnamecref{pro:int restriction map on Picard groups}:

\begin{proposition} \label{pro:int restriction map on Picard groups}
Let $X$ be a log terminal complete intersection Fano variety of dimension $n \geq 3$ in a weighted projective space~$\mathbb P$. Then, the Picard number of $X$ is~$1$.
\end{proposition}

\begin{proof}
By \cite[Prop~2.1.2]{IP99}, we have isomorphisms $\operatorname{Pic}(\mathbb P) \cong H^2(\mathbb P_{\mathrm{top}}^{\mathrm{an}}, \mathbb Z)$ and $\operatorname{Pic}(X) \cong H^2(X_{\mathrm{top}}^{\mathrm{an}}, \mathbb Z)$, where $\mathbb P_{\mathrm{top}}^{\mathrm{an}}$, respectively $X_{\mathrm{top}}^{\mathrm{an}}$, denotes the underlying topological space of the analytification of~$\mathbb P$, respectively~$X$. By \cite[Proposition~1.4]{Mav99}, the restriction map $H^i(\mathbb P_{\mathrm{top}}^{\mathrm{an}}, \mathbb C) \to H^i(X_{\mathrm{top}}^{\mathrm{an}}, \mathbb C)$ is an isomorphism for $i < n$. By \cite[Corollary 1]{Zha06}, $X$ and $\mathbb P$ are simply connected. The \lcnamecref{pro:int restriction map on Picard groups} now follows from universal coefficient theorems.
\end{proof}

To show that some sextic double solids are not $\mathbb Q$-factorial, we use the \lcnamecref{thm:int non-Cartier} below.

\begin{lemma} \label{thm:int non-Cartier}
Let $X$ be a normal projective variety of Picard number one.
Let $D$ be a non-zero effective $\mathbb Q$-Cartier divisor and $C$ a closed curve in~$X$.
Then $D \cdot C > 0$.
\end{lemma}

\begin{proof}
Replacing $D$ by a suitable multiple, it suffices to consider the case where $D$ is Cartier.
There are no non-zero effective principal divisors on a normal projective variety.
Therefore, since $X$ has Picard number one, either $D$ or $-D$ is ample.
Since $D$ intersects some closed integral curve positively, $D$ is ample by Kleiman's criterion.
Again by Kleiman's criterion, $D$ intersects $C$ positively.
\end{proof}

\subsection{Weighted blowups} \label{sec:pre div contr}

We remind the definition of weighted blowups, \cref{def:pre wei blup}.

\begin{definition} \label{def:pre equivalent morphisms}
Let $\varphi\colon Y \to X$ and $\varphi'\colon Y' \to X'$ be birational morphisms of varieties (or bimeromorphic holomorphisms of complex analytic spaces). We say that an isomorphism $X \to X'$ \textbf{lifts} if there exists an isomorphism $Y \cong Y'$ such that the diagram
\begin{equation*}
\begin{tikzcd}
Y \arrow[r, ""]{r} \arrow[d, "\varphi"]{r} & Y' \arrow[d, "\varphi'"]{r}\\
X \arrow[r, ""]{r} & X'
\end{tikzcd}
\end{equation*}
commutes. We say that $\varphi$ and $\varphi'$ are \textbf{equivalent} if there exists an isomorphism $X \cong X'$ that lifts. We say $\varphi$ and $\varphi'$ are \textbf{locally equivalent} if there exist isomorphic open subsets $U \subseteq X$ and $U' \subseteq X'$ containing the centres of the morphisms $\varphi$ and $\varphi'$ such that the restrictions $\mleft.\varphi\mright|_{\varphi^{-1}U}\colon \varphi^{-1}U \to U$ and $\mleft.\varphi'\mright|_{\varphi'^{-1}U'}\colon \varphi'^{-1}U' \to U'$ are equivalent.
\end{definition}

If we consider the complex analytic space corresponding to a variety or when we wish to emphasize that we are working in the category of complex analytic spaces, we sometimes say \textbf{analytically equivalent} or \textbf{locally analytically equivalent}.

\begin{definition} \label{def:pre wei blup}
Let $n$ be a positive integer and let $\bm w = (w_1, \ldots, w_n)$ be positive integers, called the weights of the blowup. Define a $\mathbb C^*$-action on $\mathbb C^{n+1}$ by $\lambda \cdot (u, x_1, \ldots, x_n) = (\lambda^{-1} u, \lambda^{w_1} x_1, \ldots, \lambda^{w_n} x_n)$ and define $T$ by the geometric quotient $(\mathbb C^{n+1} \setminus \mathbb V(x_1, \ldots, x_n)) / \mathbb C^*$ (or its analytification). Then the map $\varphi\colon T \to \mathbb C^n$, $[u, x_1, \ldots, x_n] \mapsto (u^{w_1} x_1, \ldots, u^{w_n} x_n)$ is called the \emph{$\bm w$-blowup of~$\mathbb C^n$}. If $Z \subseteq \mathbb C^n$ is a closed subvariety (or a closed complex subspace $Z \subseteq D$ where $D \subseteq \mathbb C^n$ is open) and $\tilde Z$ is the closure of $\varphi^{-1}(Z \setminus \{\bm 0\})$ in $T$ (in~$\varphi^{-1} D$), then the restriction $\mleft.\varphi\mright|_{\tilde Z}\colon \tilde Z \to Z$ is called the \emph{$\bm w$-blowup of~$Z$}. Let $\rho\colon Y \to X$ be a surjective birational morphism of varieties (or a surjective bimeromorphic holomorphism of complex spaces). Given an open subset $U \subseteq X$ containing the centre of $\rho$ and an isomorphism $U \cong Z \subseteq \mathbb C^n$ taking a point $P \in X$ to the origin~$\bm 0$, the map $\rho$ is called the \textbf{$\bm w$-blowup of $X$ at~$P$} if the isomorphism $U \cong Z$ lifts to $\rho^{-1}U \to \tilde Z$.
\end{definition}

\begin{remark}
\begin{enumerate}[label=(\alph*), ref=\alph*]
\item A weighted blowup crucially depends both on the isomorphism $U \cong X'$ and a choice of coordinates $x_1, \ldots, x_n$, even though it is not explicit in the notation.
\item Replacing $\bm w$ by $(w_1/g, \ldots, w_n/g)$ in \cref{def:pre wei blup}, where $g$ is the greatest common divisor of $w_1, \ldots, w_n$, gives an isomorphic blowup over~$X$.
\item By \cite[Theorem~5.1.11]{CLS11}, the weighted blowup of an affine space in \cref{def:pre wei blup} coincides with the toric description of subdividing a cone in \cite[Proposition-Definition~10.3]{KM92}.
\end{enumerate}
\end{remark}

We give alternative definitions of weighted blowup in \cref{def:pre wei blup proj} and \cref{def:pre wei blup projan} that we use in \cref{thm:wei wt-resp implies equiv blup}:

\begin{definition} \label{def:pre wei blup proj}
Let $n \in \mathbb Z_{\geq1}$ and $\bm w \in \mathbb Z_{\geq1}^n$. Let $X = \operatorname{Spec} \mathbb C[\bm x]/I$ be an affine variety. Define the $\mathbb Z_{\geq0}$-graded $\mathbb C$-algebra
\[
R_X = \mathbb C\mleft[ \mleft\{ t^d \bar x_i \;\middle|\; i \in \{1, \ldots, n\},\, d \in \{0, \ldots, w_i\} \mright\} \mright],
\]
where $t$ denotes the grading and $\bar x_i \in \mathbb C[\bm x]/I$ denotes the image of $x_i \in \mathbb C[\bm x]$. Define the morphism $\operatorname{Proj} R_X \to X$.
\end{definition}

\begin{definition} \label{def:pre wei blup projan}
Let $n \in \mathbb Z_{\geq1}$ and $\bm w \in \mathbb Z_{\geq1}^n$. Let $D \subseteq \mathbb C^n$ be an open subset. Let $X \subseteq D$ be a closed complex analytic space. For every open $V \subseteq D$, we denote the image of $f \in \mathcal O_{\mathbb C^n}(V)$ in $\mathcal O_X(X \cap V)$ by~$\bar f$.
Define the finitely presented $\mathbb Z_{\geq0}$-graded $\mathcal O_X$-algebra $\mathcal B_X$ to be the sheafification of the presheaf $\mathcal A_X$ given by
\[
\mathcal A_X(U) = \mathcal O_X(U)\mleft[ \mleft\{ t^d \bar x_i \;\middle|\; i \in \{1, \ldots, n\},\, d \in \{0, \ldots, w_i\} \mright\} \mright],
\]
where $U \subseteq X$ is open and $t$ denotes the grading. By \cite[Proposition~II.3.19]{CDG+94}, we have a morphism $\operatorname{Projan} \mathcal B_X \to X$, where $\operatorname{Projan}$ is the analytic homogeneous spectrum.
\end{definition}

\begin{lemma}
The morphisms in \cref{def:pre wei blup proj} and \cref{def:pre wei blup projan} are $\bm w$-blowups.
\end{lemma}

\begin{proof}
First, we show that \cref{def:pre wei blup proj} is the $\bm w$-blowup when $X$ is the affine space $\mathbb A^n = \operatorname{Spec} \mathbb C[x_1, \ldots, x_n]$. Let $S = \mathbb C[u, \bm x]$ be the $\mathbb Z$-graded $\mathbb C$-algebra with grading $(-1, \operatorname{wt} x_1, \ldots, \operatorname{wt} x_n)$ for $u, \bm x$. Let $S_{\geq0}$ be the non-negatively graded part of~$S$. By definition of the geometric quotient, the weighted blowup of $\mathbb A^n$ is given by $\operatorname{Proj} S_{\geq0} \to \mathbb A^n$. The $\mathbb Z_{\geq0}$-graded $\mathbb C$-algebra isomorphism
\[
\begin{gathered}
  S_{\geq0} \to R_{\mathbb A^n}\\
  u \mapsto t^{-1}, \quad x_i \mapsto t^{\operatorname{wt} x_i} x_i
\end{gathered}
\]
induces an isomorphism $\operatorname{Proj} R \to \operatorname{Proj} S_{\geq0}$ over~$\mathbb A^n$.

We show that \cref{def:pre wei blup proj} is the $\bm w$-blowup for any $X$. Define $N = n \cdot \operatorname{lcm}(w_1, \ldots, w_n)$. If $M = x_1^{a_1} \cdot \ldots \cdot x_n^{a_n}$ is any monomial such that $\sum a_i w_i > N$, then $M$ is divisible by $x_k^{N / (n w_k)}$ for some~$k$. It follows that the $N$-th Veronese subring $R_X^{(N)}$ of $R_X$ is generated by its degree $1$ part $(R_X)_N$. Therefore, $\operatorname{Proj} R_X$ is isomorphic over $X$ to $\operatorname{Bl}_{(R_X)_N} X$, where $\operatorname{Bl}_{(R_X)_N} X \to X$ is blowup of $X$ along~$(R_X)_N$.
Since the intersection of $\operatorname{Spec} (R_{\mathbb A^n})_N$ and $X$ is $\operatorname{Spec} (R_X)_N$,
we find that $\operatorname{Bl}_{(R_X)_N} X$ is the strict transform of~$X$ under the blowup of $\mathbb A^n$ along~$(R_{\mathbb A^n})_N$, which coincides with the closure of the inverse image of $X \setminus \mathbb V(x_1, \ldots, x_n)$ in $\operatorname{Bl}_{(R_{\mathbb A^n})_N} \mathbb A^n$.

We show that \cref{def:pre wei blup projan} is the $\bm w$-blowup. We similarly prove that $\operatorname{Projan} \mathcal B_X$ is the closure of the inverse image of $X \setminus \{\bm 0\}$ in $\operatorname{Projan} \mathcal B_D$. Now, it suffices to note that the analytification of $\operatorname{Proj} R_{\mathbb A^n} \to \mathbb A^n$ is $\operatorname{Projan} \mathcal B_{\mathbb C^n} \to \mathbb C^n$.
\end{proof}

In \Cref{thm:wei wt-resp implies equiv blup}, we give a simple criterion for a local biholomorphism to lift to weighted blowups.

\subsection{Divisorial contractions}

The first step in a Sarkisov link from a Fano variety is a divisorial contraction.

\begin{definition}
A \textbf{divisorial contraction} is a proper birational morphism $\varphi\colon Y \to X$ between normal varieties with terminal singularities such that
\begin{enumerate}
\item the exceptional locus of $\varphi$ is a prime divisor and
\item $-K_Y$ is $\varphi$-ample.
\end{enumerate}
\end{definition}

Kawakita \cite{Kaw03} described divisorial contractions with centre a $cA_n$ point by weighted blowups. Notational differences from \cite[Theorem~1.13]{Kaw03} are that below we have left out the description for $cA_1$ singularities and an exceptional case for~$cA_2$. Also, we have written out the converse statement more explicitly (that being a Kawakita blowup implies that it is a divisorial contraction).

\begin{theorem}[{\cite[Theorem~1.13]{Kaw03}}] \label{thm:pre Kawakita blowup}
Let $P$ be a $cA_n$ point where $n \geq 3$ of a variety $X$ with terminal singularities. Let $\varphi\colon Y \to X$ be a morphism of varieties such that the restriction $\mleft.\varphi\mright|_{Y \setminus E}\colon Y \setminus E \to X \setminus \{P\}$ is an isomorphism, where the closed subvariety $E$ is given by~$\varphi^{-1}\{P\}$. If $\varphi$ is a divisorial contraction, then $\varphi$ is locally analytically equivalent to the $(r_1, r_2, a, 1)$-blowup of $\mathbb V(x_1 x_2 + g(x_3, x_4)) \subseteq \mathbb C^4$ at $\bm 0$ with variables $x_1, x_2, x_3, x_4$ where
\begin{enumerate}
\item $a$ divides $r_1 + r_2$ and is coprime to both $r_1$ and $r_2$,
\item $g$ has weight $r_1 + r_2$, and
\item the monomial $x_3^{(r_1 + r_2)/a}$ appears in $g$ with non-zero coefficient.
\end{enumerate}
Moreover, any $\varphi$ which is locally analytically equivalent to a weighted blowup as above is a divisorial contraction, even for $n = 2$.
\end{theorem}

Any weighted blowup that is locally analytically equivalent to $\varphi$ in \cref{thm:pre Kawakita blowup} for $n \geq 2$ is called a \textbf{$(r_1, r_2, a, 1)$-Kawakita blowup}, or simply a Kawakita blowup.

\subsection{Sarkisov links} \label{sec:pre sark links}

One of the possible outcomes of the minimal model program is a Mori fibre space:

\begin{definition}
A \textbf{Mori fibre space} is a morphism of normal projective varieties $\varphi\colon X \to S$ with connected fibres such that
\begin{enumerate}
\item $X$ is $\mathbb Q$-factorial and has terminal singularities,
\item the anti-canonical class $-K_X$ is $\varphi$-ample,
\item $X/S$ has relative Picard number~1, and
\item $\dim S < \dim X$.
\end{enumerate}
If $\dim S > 0$, then we say $\varphi$ is a \emph{strict} Mori fibre space.
\end{definition}

The main examples of Mori fibre spaces we see in this paper are Fano 3-folds that are projective, $\mathbb Q$-factorial, with terminal singularities and Picard number~1, considered as a morphism over a point.

Any birational map between two Mori fibre spaces is a composition of Sarkisov links (see \cite{Cor95} or \cite{HM13}). Below, we describe the two possible types of Sarkisov links starting from a Fano variety.

\begin{definition} \label{def:Sarkisov link}
A \textbf{Sarkisov link} of type I (respectively II) between a Fano variety $X$ and a strict Mori fibre space $Y_k \to Z$ (respectively Fano variety~$Z$) is a diagram of the form
\begin{equation*}
\begin{tikzcd}
& Y_0 \arrow[ld, "\varphi"'] \arrow[r, dashed] & \ldots \arrow[r, dashed] & Y_k \arrow[rd, "\psi"]\\
X & & & & Z
\end{tikzcd}
\end{equation*}
where $X$, $Y_0$, \ldots, $Y_k$, $Z$ are normal, projective and $\mathbb Q$-factorial, the varieties $X$, $Y_0$, \ldots, $Y_k$ have terminal singularities, $Z$ has terminal singularities if it 3-dimensional, $X$ has Picard number~1, $\varphi\colon Y_0 \to X$ is a divisorial contraction, $Y_0 \dashrightarrow \ldots \dashrightarrow Y_k$ is a sequence of anti-flips, flops and flips, and $\psi\colon Y_k \to Z$ is a strict Mori fibre space (respectively divisorial contraction). If we do not require the varieties $X, Y_0, \ldots Y_k$ (respectively $X, Y_0, \ldots Y_k, Z$) to be terminal and we do not require $-K_{Y_0}$ to be $\varphi$-ample and we do not require $-K_{Y_k}$ to be $\psi$-ample but all the other properties hold, then the diagram above is called a \textbf{2-ray link} (\cite[Definition~2.1]{BZ10}).
\end{definition}

\begin{definition}
A Fano 3-fold $X$ that is a Mori fibre space is \textbf{birationally rigid} if for any Mori fibre space $Y \to S$ such that $X$ and $Y$ are birational, we have that $S$ is a point and $X$ and $Y$ are isomorphic.
\end{definition}

In \cref{sec:mod}, we show that a general sextic double solid $X$ with a $cA_n$ singularity with $n \geq 4$ which is a Mori fibre space is not birationally rigid. We show this by explicitly constructing a Sarkisov link between $X$ and another Mori fibre space. We find the Sarkisov link by restricting from a toric 2-ray link, as described in \cref{cons:pre}.

See \cite{Cox95} for the definition of Cox rings for toric varieties (where it is called the \emph{homogeneous coordinate ring}), and \cite[Definition~2.6]{HK00} for the definition of Cox rings for Mori dream spaces. Note that isomorphic varieties can have different Cox rings. By \cite[Theorem~3.7]{Cox95}, closed subschemes of a toric variety $T$ with only cyclic quotient singularities are given by homogeneous ideals in the Cox ring $\operatorname{Cox} T$, which is a polynomial ring.

\begin{construction} \label{cons:pre}
Let $X$ be a Fano variety embedded in a weighted projective space~$\mathbb P$, where $X$ is a Mori fibre space, and let $Y_0 \to X$ be a divisorial contraction from a projective $\mathbb Q$-factorial variety $Y$. By \cite[Lemma~2.9]{AK16}, the divisorial contraction $Y_0 \to X$ can be part of a Sarkisov link only if $Y_0$ is a Mori dream space.

By \cite[Proposition~2.11]{HK00}, we can embed a Mori dream space $Y_0$ into a projective toric variety $T_0$ with cyclic quotient singularities such that the Mori chambers of $Y_0$ are unions of finitely many Mori chambers of~$T_0$. Moreover, we can embed $Y_0$ in such a way that $Y_0$ is given by a homogeneous ideal $I_Y$ in $\operatorname{Cox} T_0$, and the toric 2-ray link
\begin{equation*}
\begin{tikzcd}[column sep = small]
& \arrow[ld, ""'] T_0 \arrow[rd, ""'] \arrow[rr, dashed, ""] & & T_1  \arrow[r, dashed, ""] \arrow[ld, ""']  \arrow[rd, ""'] & \cdots \arrow[r, dashed, ""] & T_r \arrow[ld, ""'] \arrow[rd, ""] \\
\mathbb P & & \mathcal W_0 & & \cdots & & S_T
\end{tikzcd}
\end{equation*}
restricts to a 2-ray link
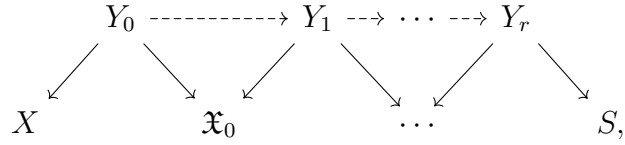
\begin{equation*}
\begin{tikzcd}[column sep = small]
& \arrow[ld, ""'] Y_0 \arrow[rd, ""'] \arrow[rr, dashed, ""] & & Y_1  \arrow[r, dashed, ""] \arrow[ld, ""']  \arrow[rd, ""'] & \cdots \arrow[r, dashed, ""] & Y_r \arrow[ld, ""'] \arrow[rd, ""] \\
X & & W_0 & & \cdots & & S,
\end{tikzcd}
\end{equation*}
where each $Y_i \subseteq T_i$ is given by the same ideal $I_Y \subseteq \operatorname{Cox} T_0 = \ldots = \operatorname{Cox} T_r$, and $W_i \subseteq \mathcal W_i$ is given by the ideal $I_Y \cap \mathbb C[\nu_0, \ldots, \nu_s]$, where $\mathcal W_i$ is given by $\operatorname{Proj} \mathbb C[\nu_0, \ldots, \nu_s]$ for some polynomials $\nu_j \in \operatorname{Cox} T_0$ that depend on~$i$ (see \cite[Remark~4]{AZ16}). In this case, $\operatorname{Cox}(T_0) / I_Y$ is a Cox ring for~$Y_0$ and we say that $I_Y$ \textbf{2-ray follows} $T_0$. In contrast to \cite[Definition~3.5]{AZ16}, we emphasise the ideal~$I_Y$, since there could be other ideals $I$ satisfying $\mathbb V(I_Y) = \mathbb V(I)$ such that the toric 2-ray link restricts to a 2-ray link for $I_Y$ but not for~$I$.

Note that some of the small birational maps $T_i \dashrightarrow T_{i+1}$ may restrict to isomorphisms $Y_i \to Y_{i+1}$. If all the varieties $Y_i$ are terminal and the anti-canonical divisor $-K_{Y_0}$ of $Y_0$ is inside the interior $\operatorname{int} (\operatorname{Mov} Y_0)$ of the movable cone, then the 2-ray link for $Y_0$ is a Sarkisov link (see \cite[Lemma~2.9]{AK16}), otherwise it is called a \textbf{bad link}.
\end{construction}

In \cref{sec:mod}, where $X$ is a sextic double solid and the centre of $Y_0 \to X$ is a $cA_n$ point, we use a projective version of \cref{thm:wei SDS cAn p} to construct the divisorial contraction $Y_0 \to X$, which is the restriction of a toric weighted blowup $\bar T_0 \to \mathbb P$. This gives us an embedding $Y_0 \to \mathbb V(I_{\bar Y}) \subseteq \bar T_0$ where $I_{\bar Y}$ might not 2-ray follow~$\bar T_0$. We use unprojection to modify $\bar T_0$ to find an embedding $Y_0 \to \mathbb V(I_Y) \subseteq T_0$ such that $I_Y$ 2-ray follows~$T_0$. See \cite[Section~2.1]{Rei00} for a simple example of unprojection, and \cref{subsec:mod cA4,subsec:mod cA7-1,subsec:mod cA7-2,subsec:mod cA8} for applications of unprojection.

To explain the notation we use for 2-ray links, we do an example in detail, namely the 2-ray link for the ambient space of the sextic double solid with a $cA_4$ singularity in \cref{subsec:mod cA4}.

\begin{example}[2-ray link for $\mathbb P(1, 1, 1, 1, 3, 5)$] \label{exa:pre cA4 toric link}
Denote the variables on $\mathbb P(1, 1, 1, 1, 3, 5)$ by $x, y, z, t, \alpha, \xi$. We perform the weighted blowup $T_0 \to \mathbb P(1, 1, 1, 1, 3, 5)$ with weights $(1, 1, 2, 3, 6)$ for variables $y, z, t, \alpha, \xi$, where the centre is the point $P_x = [1, 0, 0, 0, 0, 0]$.

We define $T_0$ as a geometric quotient. By a slight abuse of notation we denote the variables on $\mathbb C^7$ by $u, x, y, z, \alpha, \xi, t$, repeating the symbols for $\mathbb P(1, 1, 1, 1, 3, 5)$. Define a $(\mathbb C^*)^2$-action on $\mathbb C^7$ for all $(\lambda, \mu) \in (\mathbb C^*)^2$ by
\[
(\lambda, \mu) \cdot (u, x, y, z, \alpha, \xi, t) = (\mu^{-1} u, \lambda x, \lambda \mu y, \lambda \mu z, \lambda^3 \mu^3 \alpha, \lambda^5 \mu^6 \xi, \lambda \mu^2 t).
\]
Define the irrelevant ideal $I_0 = (u, x) \cap (y, z, \alpha, \xi, t)$, and define $T_0$ by the geometric quotient $\mathbb C^7 \setminus \mathbb V(I_0) / (\mathbb C^*)^2$. We use the notation
\[
\begin{array}{ccc|ccccccccc}
            &  u & x & y & z & \alpha & \xi & t &\\
\Ta{T_0\colon} &  0 & 1 & 1 & 1 & 3 & 5 & 1 & \Tb{.}\\
            & -1 & 0 & 1 & 1 & 3 & 6 & 2 &
\end{array}
\]
to describe this construction of~$T_0$. Note that we order the variables $u, x, \ldots, t$ such that the corresponding rays $\smat{0\\-1}$, $\smat{1\\0}$, \ldots, $\smat{1\\2}$ are ordered anti-clockwise around the origin. The vertical bar indicates that the irrelevant ideal is $(u, x) \cap (y, z, \alpha, \xi, t)$. The Cox ring of $T_0$ is given by $\operatorname{Cox} T_0 = \mathbb C[u, x, y, z, \alpha, \xi, t]$. The weighted blowup $T_0 \to \mathbb P(1, 1, 1, 1, 3, 5)$ is given by
\begin{equation} \label{eqn:pre 2-ray link example T0 to P}
[u, x, y, z, \alpha, \xi, t] \mapsto [x, uy, uz, u^2t, u^3\alpha, u^6\xi].
\end{equation}

We describe the cones of the toric variety~$T_0$. By \cite{HK00}, $T_0$ is a Mori dream space. The Picard group of $T_0$ is generated by $\mathbb V(u)$, the reduced exceptional divisor, and~$\mathbb V(x)$, the strict transform of a plane not passing through~$P_x$, which have bidegree $\smat{0\\-1}$ and~$\smat{1\\0}$, respectively. The variety $T_0$ is $\mathbb Q$-factorial, and any two divisors with the same bidegree are linearly equivalent. As in \cite[Section~4.1.3]{BZ10}, the effective cone $\operatorname{Eff}(T_0)$ is given by $\langle\mathbb V(u), \mathbb V(x)\rangle$, a cone in the group $N^1(T_0)$ of divisors of $T_0$ up to numerical equivalence with coefficients in~$\mathbb R$. As in \cite[Section~3.2]{AZ16}, the movable cone $\operatorname{Mov(T_0)}$ is $\langle\mathbb V(x), \mathbb V(\xi)\rangle$, and it is divided into the nef cone $\operatorname{Nef}(T_0) = \langle\mathbb V(x), \mathbb V(y)\rangle$ of $T_0$ and $\langle\mathbb V(y), \mathbb V(\xi)\rangle$, which is the pull-back of the nef cone of the small $\mathbb Q$-factorial modification $T_1$ of~$T_0$. The cones $\langle\mathbb V(x), \mathbb V(y)\rangle$ and $\langle\mathbb V(y), \mathbb V(\xi)\rangle$ are called \emph{Mori chambers}. The variety $T_1$ is defined by
\[
\begin{array}{cccccc|cccccc}
            &  u & x & y & z & \alpha & \xi & t &\\
\Ta{T_1\colon} &  0 & 1 & 1 & 1 & 3 & 5 & 1 & \Tb{.}\\
            & -1 & 0 & 1 & 1 & 3 & 6 & 2 &
\end{array}
\]
Here, $T_1$ is the geometric quotient $(\mathbb C^7 \setminus I_1) / (\mathbb C^*)^2$, where the irrelevant ideal $I_1$ is given by $(u, x, y, z, \alpha) \cap (\xi, t)$, which is indicated by the position of the vertical bar in the action-matrix. The Cox ring of $T_1$ is equal to the Cox ring of~$T_0$, namely $\operatorname{Cox} T_1 = \mathbb C[u, x, y, z, \alpha, \xi, t]$.

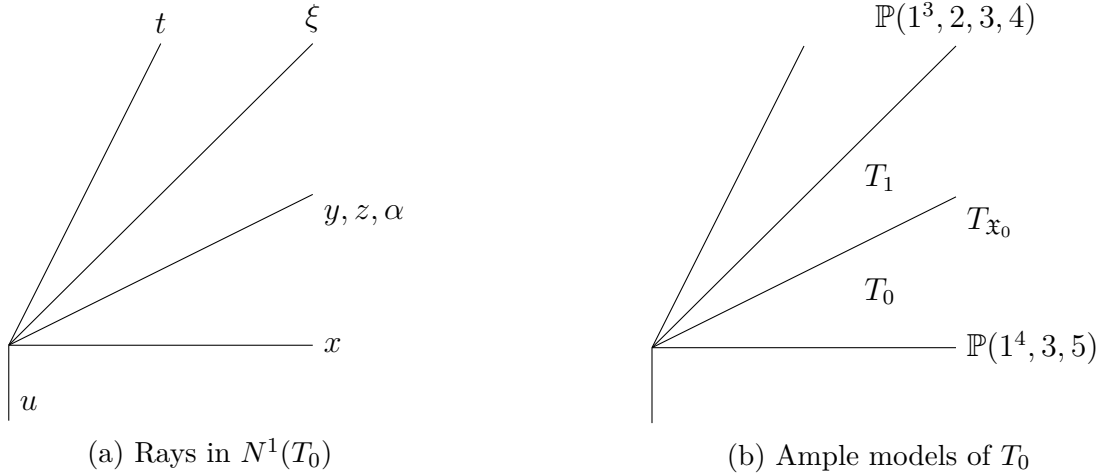
\begin{figure}
  \begin{centering}
    \begin{subfigure}[h]{.4\textwidth}
      \centering
      \begin{tikzpicture}
        \draw (0, 0) -- (0, -1) node [above right] {$u$};
        \draw (0, 0) -- (4, 0) node [right] {$x$};
        \draw (0, 0) -- (4, 1) node [right] {$y, z, \alpha$};
        \draw (0, 0) -- (4, 2) node [above right] {$\xi$};
        \draw (0, 0) -- (2, 2) node [above] {$t$};
        \end{tikzpicture}
      \caption{Rays in~$N^1(T_0)$}
    \end{subfigure}%
    \hspace{.05\textwidth}%
    \begin{subfigure}[h]{.4\textwidth}
      \centering
      \begin{tikzpicture}
        \draw (0, 0) -- (0, -1) node [above right] {};
        \draw (0, 0) -- (4, 0) node [right] {$\mathbb P(1^4, 3, 5)$};
        \draw (3, 0.375) node {$T_0$};
        \draw (0, 0) -- (4, 1) node [right] {$\mathcal W_0$};
        \draw (3, 1.125) node {$T_1$};
        \draw (0, 0) -- (4, 2) node [above right] {$\mathbb P(1^3, 2, 3, 4)$};
        \draw (0, 0) -- (2, 2) node [above] {};
      \end{tikzpicture}
      \caption{Ample models of~$T_0$}
    \end{subfigure}%
    \caption{Cones of $T_0$}
    \label{fig:pre example T0}
  \end{centering}
\end{figure}

The weighted blowup morphism $T_0 \to \mathbb P(1, 1, 1, 1, 3, 5)$ can be read off from the action-matrix of~$T_0$. Consider the ray given by $\mathbb V(x)$ in~$N^1(T_0)$. The union of the linear systems $\abs{\smat{n\\0}}$ where $n \geq 0$ has a $\mathbb C$-algebra basis $x, uy, uz, u^2t, u^3\alpha, u^6\xi$. So, the ample model (see \cite[Definition~3.6.5]{BCHM10}) of the divisor class $\mathbb V(x)$ is the morphism
\[
T_0 \to \operatorname{Proj} \bigoplus_{n \geq 0} H^0(T_0, \mathcal O_{T_0}(n \smat{1\\0})) = \operatorname{Proj} \mathbb C[x, uy, uz, u^2t, u^3\alpha, u^6\xi] = \mathbb P(1, 1, 1, 1, 3, 5)
\]
given by
\[
[u, x, y, z, \alpha, \xi, t] \mapsto [x, uy, uz, u^2t, u^3\alpha, u^6\xi],
\]
which is precisely the weighted blowup $T_0 \to \mathbb P(1, 1, 1, 1, 3, 5)$ given in \cref{eqn:pre 2-ray link example T0 to P}.

As in \cite[Section~2.1]{BZ10}, there are two projective morphisms of relative Picard number 1 from $T_0$ up to isomorphisms, corresponding to the ample models of divisors in the two edges of the nef cone of~$T_0$. The ample model of any divisor in the interior of the nef cone of $T_0$ gives an embedding of $T_0$ into a weighted projective space. The ample model of $\mathbb V(y) \in N^1(T_0)$ is given by
\[
\begin{aligned}
T_0 & \to \operatorname{Proj} \mathbb C[y, z, \alpha, u\xi, ut, x\xi, xt] \subseteq \mathbb P(1, 1, 3, 5, 1, 6, 2)\\
[u, x, y, z, \alpha, \xi, t] & \mapsto [y, z, \alpha, u\xi, ut, x\xi, xt].
\end{aligned}
\]
Denoting $\mathcal W_0 = \operatorname{Proj} \mathbb C[y, z, \alpha, u\xi, ut, x\xi, xt]$, we see that the morphism $T_0 \to \mathcal W_0$ contracts $\mathbb V(\xi, t)$ to the surface $\mathbb P(1, 1, 3) \subseteq \mathcal W_0$ and is an isomorphism elsewhere. The ample model of $\mathbb V(y) \in N^1(T_1)$ is given similarly by
\[
T_1 \to \operatorname{Proj} \mathbb C[y, z, \alpha, u\xi, ut, x\xi, xt] = \mathcal W_0,
\]
contracting $\mathbb V(u, x)$ to $\mathbb P(1, 1, 3)$. This induces a birational map $T_0 \dashrightarrow T_1$, a small $\mathbb Q$-factorial modification, given by
\[
[u, x, y, z, \alpha, \xi, t] \mapsto [u, x, y, z, \alpha, \xi, t].
\]
Note that this is the identity map on the affine space $\mathbb A^7$ but it is not an isomorphism between $T_0$ and $T_1$ since the irrelevant ideals are different. The diagram $T_0 \to \mathcal W_0 \gets T_1$ is a flop.

Note that multiplying the action-matrix of $T_0$ or $T_1$ with a matrix in $\operatorname{GL}(2, \mathbb Q)$ is equivalent to choosing a different basis for the group $(\mathbb C^*)^2$, so the geometric quotients $T_0$ and $T_1$ stay the same (see \cite[Lemma~2.4]{Ahm17}). If we multiply with a matrix with negative determinant, then we change the order of the rays in $N^1(T_0)$ from anti-clockwise to clockwise.

Similarly, there are only two projective morphisms of relative Picard number 1 from~$T_1$: the contraction $T_1 \to \mathcal W_0$ and the ample model of~$\mathbb V(\xi)$. We multiply the action-matrix of $T_1$ by the matrix $\smat{6 & -5\\2 & -1}$ with determinant~$4$ to find
\[
\begin{array}{cccccc|ccc}
            & u & x & y & z & \alpha & \xi &  t & \\
\Ta{T_1\colon} & 5 & 6 & 1 & 1 & 3 & 0 & -4 & \Tb{.}\\
            & 1 & 2 & 1 & 1 & 3 & 4 &  0 &
\end{array}
\]
The ample model of $\mathbb V(\xi)$ is given by
\[
\begin{aligned}
T_1 & \to \mathbb P(1, 1, 1, 2, 3, 4)\\
[u, x, y, z, \alpha, \xi, t] & \mapsto \mleft[ t^{\frac54} u, t^{\frac14} y, t^{\frac14} z, t^{\frac32} x, t^{\frac34} \alpha, \xi \mright].
\end{aligned}
\]
Note that this is a morphism of varieties despite having fractional powers (see~\cite{BB13}).

The 2-ray link that we have found for $\mathbb P(1, 1, 1, 1, 3, 5)$ is summarized by the diagram below.
\begin{equation*}
\begin{tikzcd}[column sep = small]
& \arrow[ld, ""'] T_0 \arrow[rd, ""'] \arrow[rr, dashed, ""] & & T_1 \arrow[ld, ""'] \arrow[rd, ""] \\
\mathbb P(1^4, 3, 5) & & \mathcal W_0 & & \mathbb P(1^3, 2, 3, 4)
\end{tikzcd}
\end{equation*}
\end{example}

For more examples on toric 2-ray links, see \cite[Section~4]{BZ10}.

\section{\texorpdfstring
{Constructing sextic double solids with a $cA_n$ singularity}
{Constructing sextic double solids with a cAn singularity}} \label{sec:constructing sds with cAn}

In this section, we give a bound $n \leq 8$ for an isolated $cA_n$ singularity on a sextic double solid, and we explicitly describe all sextic double solids that contain an isolated $cA_n$ singularity where $n \leq 8$. The main tool we use for this is the splitting lemma from singularity theory, first introduced in~\cite{Tho72}, which is used for separating the quadratic terms and the higher order terms of a power series.

\subsection{\texorpdfstring
{Splitting lemma from singularity theory}
{Splitting lemma from singularity theory}}

The splitting lemma below is taken from~\cite[Theorem~2.47]{GLS07}, with a slight modification in notation. Specifically, we write $v(x+p)$ instead of $x + g$, where $v$ is a unit in the power series ring and $p$ does not depend on $x$, as we use this form in \cref{sec:mod} for constructing birational models.

\begin{theorem}[Splitting lemma] \label{thm:con splitting}
Let $m$ be a positive integer and let $\bm y$ denote variables $(y_1, \ldots, y_m)$. Let $f \in \mathbb C\{x, \bm y\}$ be a convergent power series of multiplicity two, with degree two part of the form $x^2 + \emph{(terms in $\bm y$)}$. Then, there exist unique $v \in \mathbb C[[x, \bm y]]$ and $p, h \in \mathbb C[[\bm y]]$, where $v$ is a unit and the multiplicity of $p$ is at least two, such that
\[
f = (v(x + p))^2 + h.
\]
Moreover, the power series $h, p$ and $v$ are absolutely convergent around the origin, and the multiplicity of $h$ is at least two. It follows immediately that $f$ is right equivalent to $x^2 + h$.
\end{theorem}

\begin{proof}
It is proved in \cite[Theorem~2.47]{GLS07} that there exist unique $g \in \mathbb C[[x, \bm y]]$ and $h \in \mathbb C[[\bm y]]$, where the multiplicity of $g$ is at least two, such that $f = (x + g)^2 + h$. Moreover, it is proved that the power series $g$ and $h$ are absolutely convergent around the origin, and the multiplicity of $h$ is at least two.

By the Weierstrass preparation theorem (see \cite[Theorem~1.6]{GLS07}), there exists a unique unit $v \in \mathbb C\{x, \bm y\}$ and a unique $p \in \mathbb C\{\bm y\}$ such that $x + g = v(x + p)$.
\end{proof}

Below we give explicit recurrent formulas for $g, h, p, v$ of the splitting lemma in terms of the coefficients of~$f$.

\begin{proposition}[Explicit splitting lemma] \label{thm:con splitting explicit}
Below, we use the same notation as in the splitting lemma \cref{thm:con splitting} and its proof. Denote
\[
f = \sum_{i, d \geq 0} x^i f_{i, d}, \qquad
g = \sum_{i, d \geq 0} x^i g_{i, d}, \qquad
h = \sum_{d \geq 0} h_{d}, \qquad
p = \sum_{d \geq 0} p_{d}, \qquad
v = \sum_{i, d \geq 0} x^i v_{i, d}
\]
where $f_{i, d}, g_{i, d}, h_d, p_d, v_{i, d} \in \mathbb C[\bm y]$ are homogeneous of degree~$d$. Then,
\begin{align}
g_{1, 0} & = 0, \nonumber\\
g_{i, d} & = \frac{1}{2} \mleft(f_{i+1, d} -
  \sum_{k=0}^{d} \sum_{j = \max(0, 2-k)}^{\min(i+1, i+d-k-1)} g_{j, k} g_{i+1-j, d-k}
  \mright), \quad \text{if $(i, d) \neq (1, 0)$}, \label{eqn:con splitting g}\\
h_{d} & = f_{0, d} - \sum_{j=2}^{d-2} g_{0, j} g_{0, d-j}, \label{eqn:con splitting h}\\
p_d & = g_{0, d} - \sum_{j=2}^{d-1} v_{0, d-j} p_j, \label{eqn:con splitting p}\\
v_{0, 0} & = 1, \nonumber\\
v_{i, d} & = g_{i+1, d} - \sum_{j=2}^{d} \mleft( v_{i+1, d-j} p_j \mright), \quad \text{if $(i, d) \neq (0, 0)$.} \label{eqn:con splitting v}
\end{align}
\end{proposition}

\begin{proof}
Taking the degree $d$ part of the coefficient of $x^{i+1}$ in $f = (x + g)^2 + h$ where $i \geq 0$, we find \cref{eqn:con splitting g}. Taking all degree $d$ terms of $f = (x + g)^2 + h$ that are not divisible by~$x$, we find \cref{eqn:con splitting h}. Taking the degree $d$ part of the coefficient of $x^{i+1}$ in $x + g = v(x + p)$ where $i \geq 0$, we find \cref{eqn:con splitting v}, and taking all degree $d$ terms not divisible by~$x$, we find \cref{eqn:con splitting p}.
\end{proof}

\begin{example}
Using the notation of \cref{thm:con splitting explicit}, the first few homogeneous parts of $h$ are given in terms of coefficients of $f$ by
\begin{align*}
h_2 & = f_{0, 2}\\
h_3 & = f_{0, 3}\\
h_4 & = f_{0, 4} - \frac{f_{1, 2}^2}{4}\\
h_5 & = f_{0, 5} - \frac{f_{1, 2}^2 f_{2, 1}}{4} - \frac{f_{1, 2} f_{1, 3}}{2}\\
h_6 & = f_{0, 6} - \frac{f_{1, 2}^3 f_{3, 0}}{8} + \frac{f_{1, 2}^2 f_{2, 2}}{4} - \frac{f_{1, 2}^2 f_{2, 1}^2}{4} + \frac{f_{1, 2} f_{1, 3} f_{2, 1}}{2} - \frac{f_{1, 2} f_{1, 4}}{2} - \frac{f_{1, 3}^2}{4}.
\end{align*}
\end{example}

\subsection{Parameter spaces of sextic double solids}

We apply the explicit splitting lemma (\cref{thm:con splitting explicit}) to describe the equation of a sextic double solid~$X \subseteq \mathbb P(1, 1, 1, 1, 3)$ that has a singular point at~$P_x = [1, 0, 0, 0, 0]$.

\begin{notation} \label{nota:con}
Let $X$ be the subscheme of $\mathbb P(1, 1, 1, 1, 3)$, with variables $x, y, z, t, w$, defined by~$f$, where
\begin{equation} \label{eqn:con notation}
\begin{aligned}
f & = -w^2 + x^4 (t^2 + Q_2)\\
  & + x^3 (4 t^3 a_0 + 4 t^2 a_1 + 2 t a_2 + a_3)\\
  & + x^2 (2 t^4 b_0 + 2 t^3 b_1 + 2 t^2 b_2 + 2 t b_3 + b_4)\\
  & + x (2 t^5 c_0 + 2 t^4 c_1 + 2 t^3 c_2 + 2 t^2 c_3 + 2 t c_4 + c_5)\\
  & + t^6 d_0 + 2 t^5 d_1 + t^4 d_2 + 2 t^3 d_3 + t^2 d_4 + 2 t d_5 + d_6,
\end{aligned}
\end{equation}
where the polynomials $a_j, b_j, c_j, d_j \in \mathbb C[y, z]$ and $Q_j \in \mathbb C[y, z, t]$ are homogeneous of degree~$j$.

We define the following $11$ technical conditions, where $i \in \{1, 2, 3, 4\}$:
\begin{enumerate}
\item (This condition is always true).
\item $Q_2 = 0$.
\item Condition (2) holds and $a_3 = 0$.
\item Condition (3) holds and $b_4 = a_2^2$.
\item Condition (4) holds and $c_5 = 2 a_2 b_3 - 4 a_1 a_2^2$.
\item Condition (5) holds and $d_6 = 2 a_2 c_4 + b_3^2 - 8 a_1 a_2 b_3 - 2 a_2^2 b_2 + 4 a_0 a_2^3 + 16 a_1^2 a_2^2$.
\item[\phantomsection($7.i$)] \stepcounter{enumi}
Condition (6) holds and there exist
polynomials $q, r, s, e \in \mathbb C[y, z]$ that are respectively homogeneous of degrees $i-1$, $3-i$, $4-i$, $i+1$, where $0$ is considered to be the only polynomial homogeneous of degree $-1$, such that
\[
\begin{aligned}
a_2 & = q r\\
b_3 & = q s + 4 a_1 q r\\
c_4 & = 2 a_1 q s - 6 a_0 q^2 r^2 + 8 a_1^2 q r + e r\\
d_5 & = 2 b_2 q s - 8 a_1^2 q s - e s - b_1 q^2 r^2 + c_3 q r
\end{aligned}
\]
\item Condition~(7.1) holds and there exists a
constant $A_0 \in \mathbb C$ and a
polynomial $B_1 \in \mathbb C[y, z]$ homogeneous of degree~$1$ such that
\[
\begin{aligned}
e_2 & = 4 A_0 r_2 + b_2 - 6 a_1^2\\
c_3 & = 6 a_0 s_3 - 4 A_0 s_3 + 4 a_0 a_1 r_2 - 8 A_0 a_1 r_2 + B_1 r_2 + 2 a_1 b_2 - 4 a_1^3\\
d_4 & = -2 s_3 B_1 + 16 r_2^2 A_0^2 - 8 b_2 r_2 A_0 + 16 a_1^2 r_2 A_0 + 4 b_1 s_3\\
  & - 8 a_0 a_1 s_3 - 2 b_0 r_2^2 + 2 c_2 r_2 + b_2^2 - 4 a_1^2 b_2 + 4 a_1^4.
\end{aligned}
\]
\end{enumerate}

Note that zero is homogeneous of every non-negative degree, so for example in Condition~(7.1), the term $e$ can be zero.

Next, define the set of $11$ rational indices
\[
\mathrm{Inds} := \{1,\, 2,\, 3,\, 4,\, 5,\, 6,\, 7.1,\, 7.2,\, 7.3,\, 7.4,\, 8\}.
\]
Let $\lfloor k\rfloor$ denote the greatest integer not greater than~$k$.
For every $k \in \mathrm{Inds}$, let $R_k$ denote the \mbox{$\mathbb C$-a}lgebra freely generated by the coefficients of the polynomials
\begin{itemize}
\item $Q_2, a_i, b_i, c_i, d_i$ if $k \leq 6$,
\item $a_i, b_i, c_i, d_i, q, r, s, e$ if $k \in \{7.1,\, 7.2,\, 7.3,\, 7.4\}$, and
\item $a_i, b_i, c_i, d_i, q, r, s, e, A_0, B_1$ if $k = 8$,
\end{itemize}
where we consider the coefficients to be variables satisfying Condition~$(k)$. Define
\[
F_k = \operatorname{Spec} \left( R_k[x, y, z, t, w] / (f) \right),
\]
where $f  \in R_k[x, y, z, t, w]$ is the polynomial in \cref{eqn:con notation}. Let \emph{family $k$} denote the set of fibres of $F_k \to \operatorname{Spec} R_k$ over closed points.
We say that a \textbf{general} sextic double solid in family $k$ satisfies a property if the property is satisfied by all the fibres of $F_k \to \operatorname{Spec} R_k$ over the closed points of some Zariski open dense set in $\operatorname{Spec} R_k$.
We say that an \textbf{analytically very general} sextic double solid in family $k$ satisfies a property if there is a Zariski open dense subset $U$ of $\operatorname{Spec} R_k$ such that the property is satisfied by all the fibres of $F_k \to \operatorname{Spec} R_k$ over the closed points of $U$ that are in the complement of some countable union of closed analytic proper subsets.
\end{notation}

\begin{remark} \label{rem:con notation}
\begin{enumerate}[label=(\alph*), ref=\alph*]
\item The following are equivalent in \cref{nota:con}:
\begin{enumerate}[label=(\roman*), ref=\roman*]
\item $X$ is a sextic double solid,
\item $X$ is a variety,
\item $f$ is irreducible, and
\item $f + w^2$ is not the square of a polynomial in $\mathbb C[x, y, z, t]$.
\end{enumerate}
Note that if $(\mathbb V(f), \bm 0)$ is a $cA_n$ singularity for some~$n$, then $f$ is irreducible.

\item \label{ite:con notation closed point defines f} Every closed point of $\operatorname{Spec} R_k$ bijectively corresponds to a choice of complex coefficients of
\begin{itemize}
\item $Q_2, a_i, b_i, c_i, d_i$ if $k \leq 6$,
\item $a_i, b_i, c_i, d_i, q, r, s, e$ if $k \in \{7.1,\, 7.2,\, 7.3,\, 7.4\}$, and
\item $a_i, b_i, c_i, d_i, q, r, s, e, A_0, B_1$ if $k = 8$,
\end{itemize}
so determines a unique polynomial $f \in \mathbb C[x, y, z, t, w]$. For every closed point $P \in \operatorname{Spec} R_k$ such that $f$ is irreducible, the fibre of $F_k \to \operatorname{Spec} R_k$ over $P$ is a sextic double solid.

\item The varieties $\operatorname{Spec} R_k$ are affine spaces and their dimensions are given in \cref{tab:con param space dims}. The affine spaces $\operatorname{Spec} R_{7.1}$, $\operatorname{Spec} R_{7.2}$, $\operatorname{Spec} R_{7.3}$ and $\operatorname{Spec} R_{7.4}$ all have the same dimension.

\item \label{itm:con param space automorphisms} Let $k \in \mathrm{Inds}$ and let $f \in \mathbb C[x, y, z, t, w]$ in \cref{nota:con} satisfy Condition~($k$). The graded \mbox{$\mathbb C$-}algebra automorphisms $\sigma$ of $\mathbb C[x, y, z, t, w]$, which fix the point $P_x = [1, 0, 0, 0, 0]$ and take $f$ to another polynomial $\sigma(f)$ of the form in \cref{nota:con} satisfying Condition~($k$), are given by
\[
\pmat{x\\y\\z\\t\\w} \mapsto \pmat{
  \alpha & R_3 &\\
    & M_3 &\\
    &     & \pm 1
} \pmat{x\\y\\z\\t\\w}
\]
when $k = 1$, and given by
\[
\pmat{x\\y\\z\\t\\w} \mapsto \pmat{
  \alpha & R_2 & \beta\\
    & M_2 & C_2 &\\
    &     & \alpha^{-2} &\\
    &     &        & \pm 1
} \pmat{x\\y\\z\\t\\w}
\]
when $k \geq 2$, where $M_i \in \operatorname{GL}(i, \mathbb C)$ are matrices, $R_i \in \mathbb C^i$ are row-vectors, $C_2 \in \mathbb C^2$ is a column-vector and $\alpha \in \mathbb C^*, \beta \in \mathbb C$ are scalars. These automorphisms form an algebraic group which is of dimension $13$ if $k = 1$ and of dimension $10$ if $k \geq 2$.

If $k > 7$, then we also have the $\mathbb C^*$-action
\[
\lambda \cdot q := \lambda q, \quad
\lambda \cdot r := \lambda^{-1}r, \quad
\lambda \cdot s := \lambda^{-1}s, \quad
\lambda \cdot e := \lambda e
\]
which leaves $f$ invariant.

If a coarse moduli space of sextic double solids with an isolated $cA_{\lfloor k\rfloor}$ singularity exists, then we expect its dimension to differ from $\dim \operatorname{Spec} R_k$ by $13$ if $k = 1$, by $10$ if $2 \leq k \leq 6$ and by $11$ if $k > 7$. The moduli space of smooth sextic double solids has dimension~68.
\Cref{tab:con param space dims} shows the expected moduli space dimensions.

\item If $X$ has an isolated singularity at~$P_x$, then by using the $\mathbb C^*$-action described in \labelcref{itm:con param space automorphisms} for $k > 7$ and \cref{thm:con cA7 no common prime divisor}, we can set $q = 1$, $r = 1$ and $s = 1$ respectively for families 7.1, 7.3 and~7.4.
\end{enumerate}
\end{remark}

\begin{table}[ht]
\centering
\caption{Dimension of the space of sextic double solids with an isolated $cA_{\lfloor k\rfloor}$\label{tab:con param space dims}}
\begin{tabular}{ccccccccc}
\toprule
$k$                          &  1 &  2 &  3 &  4 &  5 &  6 & 7.1, 7.2, 7.3, 7.4 &  8\\
\midrule
$\dim \operatorname{Spec} R_k$       & 80 & 74 & 70 & 65 & 59 & 52 & 45 & 36\\
expected moduli space $\dim$ & 67 & 64 & 60 & 55 & 49 & 42 & 34 & 25\\
\bottomrule
\end{tabular}
\end{table}

We state the main theorem of this section, describing sextic double solids with an isolated $cA_n$ singularity.

\begin{maintheorem} \label{mai:con}
For every positive integer~$n$, both of the following hold:
\begin{enumerate}[label=(\alph*), ref=\alph*]
\item \label[(part)]{itm:con theo bound m leq 8} If a sextic double solid has an isolated $cA_n$ singularity, then $n \leq 8$.
\item \label[(part)]{itm:con theo every SDS is some X} Every sextic double solid with an isolated $cA_n$ singularity $P$ is isomorphic to a variety $X$ in \cref{nota:con} satisfying Condition~$(l)$ for some $l \in \mathrm{Inds}$ such that $\lfloor l\rfloor = n$, with the isomorphism sending $P$ to $P_x = [1, 0, 0, 0, 0]$.
\end{enumerate}
Furthermore, for every $k \in \mathrm{Inds}$, all of the following hold:
\begin{enumerate}[resume, label=(\alph*), ref=\alph*]
\item \label[(part)]{itm:con theo cAm} If $k \geq 2$, then every scheme $X$ in \cref{nota:con} satisfying Condition~$(k)$ has either a (possibly non-isolated) $cA_m$ singularity or the singularity $(\mathbb V(x_1^2 + x_2^2), \bm 0) \subseteq (\mathbb C^4, \bm 0)$ at~$P_x$, where $m \geq \lfloor k\rfloor$ and $\mathbb C^4$ has variables $x_1, x_2, x_3, x_4$.
\item \label[(part)]{itm:con theo general smooth} A general sextic double solid in family~$k$ is smooth outside a $cA_{\lfloor k\rfloor}$ singularity at~$P_x$.
\item\label[(part)]{itm:con very gen mfs} An analytically very general sextic double solid in family~$k$ is factorial, except for $k = 7.4$. No terminal variety in family~$7.4$ is $\mathbb Q$-factorial.
\end{enumerate}
\end{maintheorem}

\begin{remark}
\begin{enumerate}[label=(\alph*), ref=\alph*]
\item By \cref{pro:int restriction map on Picard groups}, all log terminal sextic double solids have Picard number~$1$. Therefore, by \cref{mai:con} \cref{itm:con theo general smooth,itm:con very gen mfs}, an analytically very general sextic double solid in each family $k \in \mathrm{Inds} \setminus \{7.4\}$ is a Mori fibre space over a point.
\item Let $k \in \mathrm{Inds} \setminus \{1, 8\}$. Let $X$ satisfy Condition~$(k)$ but not Condition~$(l)$ for any $l \in \mathrm{Inds}$ satisfying $\lfloor l\rfloor = \lfloor k\rfloor + 1$. The proof of \cref{mai:con} \cref{itm:con theo every SDS is some X} implies that if one of the following holds:
\begin{itemize}
\item $k < 6$,
\item $P_x$ is an isolated singularity, or
\item $\lfloor k\rfloor = 7$, $r$ and $s$ are coprime and $q$ and $e$ are coprime,
\end{itemize}
then $X$ has a $cA_{\lfloor k\rfloor}$ singularity at~$P_x$.
\end{enumerate}
\end{remark}

\subsection{\texorpdfstring
{Bound $n \leq 8$ for an isolated $cA_n$ singularity}
{Bound n <= 8 for an isolated cAn singularity}} \label{sec:bounding n}

In this section, we prove \cref{itm:con theo bound m leq 8,itm:con theo cAm} of \cref{mai:con}, showing that the parameter spaces in \cref{nota:con} describe sextic double solids with a $cA_n$ singularity. In addition, we prove \cref{itm:con theo every SDS is some X} of \cref{mai:con}, namely the bound $n \leq 8$ for an isolated $cA_n$ singularity. The bound $n \leq 8$ for an isolated $cA_n$ singularity is proved by explicitly describing a curve of singularities for $n > 9$.

First we state a few lemmas needed for the proof.

\begin{lemma} \label{thm:con family (7.5) not terminal}
If $X$ in \cref{nota:con} satisfies Condition~$(6)$ and $P_x$ is an isolated singularity of~$X$, then $a_2 \neq 0$ or $b_3 \neq 0$.
\end{lemma}

\begin{proof}
If Condition~$(6)$ holds and $a_2 = b_3 = 0$, then $a_3 = b_4 = c_5 = d_6 = 0$. Let $C$ be the curve defined by the ideal $(t, w, 2xc_4 + 2d_5)$. Note that $C$ contains~$P_x$. Taking partial derivatives, we see that every point of $C$ is a singular point of~$X$.
\end{proof}

The following \lcnamecref{thm:con cA7 no common prime divisor} is useful when using \cref{nota:con}:

\begin{proposition} \label{thm:con cA7 no common prime divisor}
If $X$ in \cref{nota:con} satisfies Condition~$(k)$ and $P_x$ is an isolated singularity of~$X$ where $k > 7$, then $q$ and $e$ are coprime and $r$ and $s$ are coprime as polynomials in $\mathbb C[y, z]$.
\end{proposition}

\begin{proof}
Let $D \in \mathbb C[y, z]$ be a common prime divisor of $r$ and~$s$ or a common prime divisor of $q$ and~$e$. Then $D$ divides $a_2, b_3, c_4, d_5$, and $D^2$ divides $a_3, b_4, c_5, d_6$. Let $C$ be the curve defined by the ideal $(D, t, w)$. Note that $C$ contains~$P_x$. Taking partial derivatives, we see that $X$ is singular at every point of~$C$
\end{proof}

\begin{lemma} \label{thm:con solve equations key}
Let $r, s \in \mathbb C[y, z]$ have no common prime divisors, and let $q \in \mathbb C[y, z]$ be non-zero. Let $h_n \in \mathbb C[y, z]$ be of the form $h_n = q^\alpha (r^\beta C_r - s^\gamma C_s)$ where $C_r, C_s \in \mathbb C[y, z]$ and $\alpha, \beta, \gamma$ are non-negative integers. Then
\[
h_n = 0 \iff \text{there exists $C \in \mathbb C[y, z]$ such that $C_r = s^\gamma C$ and $C_s = r^\beta C$}.
\]
\end{lemma}

\begin{proof}
Obvious.
\end{proof}

\newcommand\condNine{
\[
\begin{aligned}
A_0 & = a_0\\
B_1 & = b_1\\
d_3 & = -s_3 B_0 + 2 b_0 s_3 - 2 a_0^2 s_3 + c_1 r_2 - 4 a_0 b_1 r_2\\
  & + 16 a_0^2 a_1 r_2 + b_1 b_2 - 4 a_0 a_1 b_2 - 2 a_1^2 b_1 + 8 a_0 a_1^3\\
c_2 & = r_2 B_0 - 6 a_0^2 r_2 + 2 a_0 b_2 + 2 a_1 b_1 - 12 a_0 a_1^2.
\end{aligned}
\]
}

\begin{proof}[\textnormal{\textbf{Proof of \cref{mai:con} \cref{itm:con theo every SDS is some X}}}]
First, we prove that every sextic double solid $Y \subseteq \mathbb P(1, 1, 1, 1, 3)$ with a singular point $P$ (not necessarily of type~$cA_n$) is isomorphic to some $X$ in \cref{nota:con}, with the isomorphism sending $P$ to $P_x = [1, 0, 0, 0, 0]$. For this, it suffices to note that \cref{nota:con} describes all sextic double solids with a singular point at $P_x$, and that we can move any point of $Y$ to $P_x$ using an automorphism of $\mathbb P(1, 1, 1, 1, 3)$. This proves the case $n = 1$. For the rest of the proof, $X$ is given by some $f$ in \cref{nota:con}  with a (not necessarily isolated) $cA_n$ singularity at $P_x$ and $n$ is at least~$2$.

Let $X^{\mathrm{an}}$ denote the analytification of~$X$. By \cref{thm:ana stable equivalence,thm:ana lowest degree parts,thm:pre cAn equation}, after applying a suitable linear invertible coordinate change on $y, z, t$, Condition~(2) holds. This proves the case $n = 2$. For the rest of the proof, Condition~(2) holds and $n$ is at least~$3$.

Let
\[
X_x = \operatorname{Spec} \mleft(\mathbb C[y, z, t, w] / (f(1, y, z, t, w))\mright)
\]
denote the affine open of $X$ given by inverting~$x$. Let $g \in \mathbb C\{y, z, t\}$ and $h \in \mathbb C\{y, z\}$ be the unique convergent power series of multiplicity at least 2 such that
\[
f(1, y, z, t, w) = -w^2 + (t + g)^2 + h.
\]
Since by assumption $(X^{\mathrm{an}}, P_x)$ is a $cA_n$ singularity, \cref{thm:ana stable equivalence,thm:ana lowest degree parts,thm:pre cAn equation} imply that $h_2 = \ldots = h_n = 0$, where $h_j \in \mathbb C[x_3, x_4]$ is the homogeneous degree $j$ part of~$h$.

Using the explicit splitting lemma (\cref{thm:con splitting explicit}), it is straightforward to compute that $h_2 = \ldots = h_n = 0$ is equivalent to satisfying Condition~$(n)$ when $n \leq 6$, even if $P_x$ is not an isolated singularity. This proves the cases $n \in \{3, \ldots, 6\}$. For the rest of the proof, Condition~(6) holds, $(X^{\mathrm{an}}, P_x)$ is an isolated $cA_n$ singularity and $n$ is at least~$7$.

By \cref{thm:con family (7.5) not terminal}, $a_2 \neq 0$ or $b_3 \neq 0$. Define $q$ to be a homogeneous greatest common divisor of $a_2$ and~$b_3$. Define $r$ and $s \in \mathbb C[y, z]$ to be the unique homogeneous polynomials such that
\[
\begin{aligned}
  a_2 & = q r\\
  b_3 & = q s + 4 a_1 q r.
\end{aligned}
\]
Then $r$ and $s$ are coprime. Using the explicit splitting lemma (\cref{thm:con splitting explicit}), we compute
that
\[
h_7 = q (r (-12 a_0 q^2 r s+4 b_2 q s-2 b_1 q^2 r^2+2 c_3 q r-2 d_5) - s (2 c_4-4 a_1 q s)).
\]
Using \cref{thm:con solve equations key}, the equations $h_2 = \ldots = h_7 = 0$ imply the existence of a polynomial $e \in \mathbb C[y, z]$ such that
\[
\begin{aligned}
c_4 & = 2 a_1 q s - 6 a_0 q^2 r^2 + 8 a_1^2 q r + e r\\*
d_5 & = 2 b_2 q s - 8 a_1^2 q s - e s - b_1 q^2 r^2 + c_3 q r.
\end{aligned}
\]
Therefore, $h_2 = \ldots = h_7 = 0$ implies Condition~$(7.i)$, where $i$ is defined by
\[
i := \deg \gcd(a_2, b_3) + 1,
\]
where $\deg \gcd(a_2, b_3)$ is the degree of a greatest common divisor of $a_2 \in \mathbb C[y, z]$ and $b_3 \in \mathbb C[y, z]$. This proves $n = 7$.

Next, we show that if $h_2 = \ldots = h_8 = 0$ and one of Conditions $(7.2)$, $(7.3)$ and $(7.4)$ holds, then $r$ and $s$ have a common prime divisor or $q$ and $e$ have a common prime divisor, which contradicts \cref{thm:con cA7 no common prime divisor}. In Condition~$(7.2)$, we calculate that $h_8 + e^2 r^2$
is divisible by~$q$, giving $r = C q$ for some $C \in \mathbb C$. Substituting into $h_8$, we compute that $h_8 - 2 q e s^2$
is divisible by~$q^2$. Therefore $q$ and $s$ have a common prime divisor, giving that $r$ and $s$ have a common prime divisor, a contradiction. Conditions $(7.3)$ and $(7.4)$ are similar.

Hence, if $h_2 = \ldots = h_8 = 0$, then Condition $(7.1)$ holds. Using the explicit splitting lemma we calculate~$h_8$
and using \cref{thm:con solve equations key} we can show that $h_2 = \ldots = h_8 = 0$ implies Condition~$(8)$.
\end{proof}

\begin{proof}[\textnormal{\textbf{Proof of \cref{mai:con} part \labelcref{itm:con theo bound m leq 8}}}]
Assume that $X$ is a sextic double solid with an isolated $cA_n$ singularity where $n \geq 9$. Using the notation in the proof of \cref{mai:con} \cref{itm:con theo every SDS is some X}, we find that Condition~(8) holds and $h_2 = \ldots = h_9 = 0$. Using the explicit splitting lemma we compute~$h_9$
and using \cref{thm:con solve equations key} we find that there exists $B_0 \in \mathbb C$ such that
\condNine
Substituting into $f$ gives
\[
\begin{aligned}
x^3 a_3 + x^2 b_4 + x c_5 + d_6 & = (s_3 + 2 a_1 r_2 + x r_2)^2\\
x^3 a_2 + x^2 b_3 + x c_4 + d_5 & = (s_3 + 2 a_1 r_2 + x r_2) (-2 a_0 r_2 + b_2 - 2 a_1^2 + 2 x a_1 + x^2).
\end{aligned}
\]
Define the curve $C$ by the ideal $(w, t, s_3 + 2 a_1 r_2 + x r_2)$.
Taking partial derivatives, we find that $X$ is singular at every point of~$C$, a contradiction.
\end{proof}

\begin{proof}[\textnormal{\textbf{Proof of \cref{mai:con} \cref{itm:con theo cAm}}}]
Let $X^{\mathrm{an}}$ denote the analytification of~$X$. Using the explicit splitting lemma (\cref{thm:con splitting explicit}), we can compute that the complex space germ $(X^{\mathrm{an}}, P_x)$ is isomorphic to $(\mathbb V(-w^2 + t^2 + h))$, where $h \in \mathbb C[y, z]$ is zero or has multiplicity at least $\lfloor k\rfloor + 1$. By \cref{thm:ana stable equivalence,thm:ana lowest degree parts,thm:pre cAn equation}, $X$ has either a (possibly non-isolated) $cA_m$ singularity or the singularity $(\mathbb V(x_1^2 + x_2^2), \bm 0) \subseteq (\mathbb C^4, \bm 0)$ at~$P_x$, where $m \geq \lfloor k\rfloor$ and $\mathbb C^4$ has variables $x_1, x_2, x_3, x_4$.
\end{proof}

\subsection{\texorpdfstring
{Smoothness outside the isolated $cA_n$ point}
{Smoothness outside the isolated cAn point}}

In this section, we prove \cref{mai:con} \cref{itm:con theo general smooth} using dimension count arguments, showing that a general sextic double solid with an isolated $cA_n$ singularity is smooth outside the $cA_n$ point.

\begin{lemma} \label{thm:con X general t non-zero}
For every $k \in \mathrm{Inds}$, a general member of family $k$ in \cref{nota:con} is smooth at every point with $t$-coordinate non-zero.
\end{lemma}

\begin{proof}
Let $\widehat{\mathcal A}_k$ denote the set of closed points $Q$ of $\operatorname{Spec} R_k$ such that the fibre of $F_k \to \operatorname{Spec} R_k$ over $Q$ has a singular point at $P_t = [0, 0, 0, 1, 0]$. We find
\[
f(P_t) = d_0, \quad \frac{\partial f}{\partial x}(P_t) = 2 c_0, \quad \frac{\partial f}{\partial y}(P_t) = 2 \frac{\partial d_1}{\partial y}, \quad \frac{\partial f}{\partial z}(P_t) = 2 \frac{\partial d_1}{\partial z}, \quad \frac{\partial f}{\partial t}(P_t) = 6 d_0.
\]
By the Jacobian criterion (\cite[Exercise~4.2.10]{Liu06}), $\widehat{\mathcal A}_k$ is the set of closed points of
\[
\mathcal A_k = \mathbb V_{\operatorname{Spec} R_k}\mleft(d_0, c_0, \frac{\partial d_1}{\partial y}, \frac{\partial d_1}{\partial z}\mright).
\]
We see that $\dim \mathcal A_k = \dim \operatorname{Spec} R_k - 4$.

The $\mathbb C$-algebra automorphism $x \mapsto x + \alpha_x t$, $y \mapsto y + \alpha_y t$, $z \mapsto z + \alpha_z t$ of $\mathbb C[x, y, z, t, w]$ defines a morphism
\[
\pi_{\mathcal A_k}\colon \operatorname{Spec} \mathcal A_k \times \operatorname{Spec} \mathbb C[\alpha_x, \alpha_y, \alpha_z] \to \operatorname{Spec} R_k
\]
with closed image. The set of closed points $Q$ of $\operatorname{Spec} R_k$, where the fibre of $F_k \to \operatorname{Spec} R_k$ over $Q$ has a singular point with $t$-coordinate non-zero, is precisely the set of closed points of the image of~$\pi_{\mathcal A_k}$. The image of $\pi_{\mathcal A_k}$ has codimension at least~$1$.
\end{proof}

\begin{lemma} \label{thm:con X general t zero}
For every $k \in \mathrm{Inds}$, a general member of family $k$ in \cref{nota:con} is smooth at every point different from $P_x$ that has $t$-coordinate zero.
\end{lemma}

\begin{proof}
Let $P = [0, \beta, \gamma, 0, 0] \in \mathbb P(1, 1, 1, 1, 3)$, where $(\beta, \gamma) \in \mathbb C^2 \setminus \{(0, 0)\}$. We find
\[
f(P) = d_6(P),\ \frac{\partial f}{\partial x}(P) = c_5(P),\ \frac{\partial f}{\partial y}(P) = \frac{\partial d_6}{\partial y}(P),\ \frac{\partial f}{\partial z}(P) = \frac{\partial d_6}{\partial z}(P),\ \frac{\partial f}{\partial t}(P) = 2 d_5(P).
\]
Define the linear polynomial $l = \gamma y - \beta z$. By the Jacobian criterion (\cite[Exercise~4.2.10]{Liu06}), $P$ is a singular point of $X$ if and only if the following divisibility constraint is satisfied:
\begin{equation} \label[divi]{eqn:con divisibility constraint}
\text{$l$ divides $c_5$ and $d_5$ and $l^2$ divides~$d_6$.}
\end{equation}
The set of closed points $Q \in \operatorname{Spec} R_k$, where the fibre of $F_k \to \operatorname{Spec} R_k$ over $Q$ is singular at $P$ for some~$(\beta, \gamma) \in \mathbb C^2 \setminus \{(0, 0)\}$, coincides with the set of closed points of a closed subset $\mathcal B_k$ of $\operatorname{Spec} R_k$. We show that $\dim \mathcal B_k$ is at most $\dim \operatorname{Spec} R_k - 2$.

\begin{itemize}
\item If $k \leq 4$, then the $19$ coefficients of $c_5, d_5$ and $d_6$ are algebraically independent in~$R_k$. By \cref{eqn:con divisibility constraint}, $\dim \mathcal B_k = \dim \operatorname{Spec} R_k - 3$.

\item If $k = 5$, then the $20$ coefficients of $a_2, b_3, d_5$ and $d_6$ are algebraically independent in~$R_k$. We have $c_5 = a_2 (2 b_3 - 4 a_1 a_2)$.
If $l$ divides $c_5$, then $l$ divides $a_2$ or $l$ divides $b_3 - 2 a_1 a_2$.
By \cref{eqn:con divisibility constraint}, in both cases we have $3$ less degrees of freedom. More formally, $\mathcal B_k$ is the union of the images of two morphisms, both having codimension exactly $3$ in $\operatorname{Spec} R_k$. Therefore, $\dim \mathcal B_k = \dim \operatorname{Spec} R_k - 3$.

\item If $k = 6$, then the $23$ coefficients of $a_2, b_3, c_4, d_5$ are algebraically independent in~$R_k$. We have $c_5 = a_2 (2 b_3 - 4 a_1 a_2)$
and $d_6 = a_2 \cdot (2 c_4 + G) + b_3^2$ for a polynomial $G \in \mathbb C[y, z]$ homogeneous of degree~$4$ which does not contain~$c_4$.

If $l$ divides $a_2$, then using the divisibility constraint~\labelcref{eqn:con divisibility constraint}, we find that $l$ divides~$b_3$. Now, $l^2$ divides $a_2$ or $l$ divides $2 c_4 + G$. So, there are three less degrees of freedom in choosing $a_2$, $b_3$, $c_4$ and~$d_5$.

If $l$ does not divide $a_2$, then $l$ divides $b_3 - 2 a_1 a_2$,
so $b_3 = 2 a_1 a_2 + Q l$
for some homogeneous quadratic form $Q \in \mathbb C[y, z]$. From $l \mid d_6$, we find that $l$ divides $c_4 - a_2 b_2 + 2 a_0 a_2^2 + 2 a_1^2 a_2$,
so $c_4 = C l + a_2 b_2 - 2 a_0 a_2^2 - 2 a_1^2 a_2$
for some homogeneous cubic form $C \in \mathbb C[y, z]$. From $l^2 \mid d_6$, we find that $l$ divides $C - 4 Q a_1$.
Therefore, after fixing $a_0, a_1, a_2$ and~$b_2$, there are at least two less degrees of freedom in choosing $b_3$, $c_4$ and~$d_5$.

In both cases, we see that $\dim \mathcal B_k \leq \dim \operatorname{Spec} R_k - 2$.

\item If $\lfloor k\rfloor = 7$, then
\[
\begin{aligned}
c_5 & = 4 q^2 r (2 s + a_1 r)\\
d_5 & = -e s + q (2 b_2 s - a_1^2 s - 4 b_1 q r^2 + c_3 r)\\
d_6 & = 4 q (e r^2 + q (s^2 + a_1 r s - 8 a_0 q r^3 - b_2 r^2 + a_1^2 r^2)).
\end{aligned}
\]
Let us consider $f$ for a closed point in $\mathcal B_k$. If $l \mid q$, then since $q$ and $e$ are coprime, we have $l \mid r$ and $l \mid s$, a contradiction. If $l \mid r$, then since $l \mid d_6$, we find $l \mid s$, a contradiction. Therefore, $l$ divides neither $q$ nor~$r$.

So, $l$ divides $2 s + a_1 r$.
Using $l^2 \mid d_6$, we see that $l^2$ divides $-32 a_0 q^2 r - 4 b_2 q + 3 a_1^2 q + 4 e$.
After fixing $a_0, a_1, b_2, q$ and $r$, we see that there are at least two less degrees of freedom in choosing $s$ and~$e$. So, we have $\dim \mathcal B_k \leq \dim \operatorname{Spec} R_k - 2$.

\item If $k = 8$, then
\[
\begin{aligned}
c_5 & = 2 r_2 (s_3 + 2 a_1 r_2)\\
d_5 & = r_2 (r_2 B_1 - 8 s_3 A_0 - 8 a_1 r_2 A_0 + 6 a_0 s_3 - b_1 r_2 + 4 a_0 a_1 r_2 + 2 a_1 b_2 - 4 a_1^3)\\
  & + s_3 (b_2 - 2 a_1^2)\\
d_6 & = r_2 (8 r_2^2 A_0 + 4 a_1 s_3 - 8 a_0 r_2^2 + 4 a_1^2 r_2) + s_3^2.
\end{aligned}
\]
We consider $f$ for a closed point in~$\mathcal B_k$. If $l \mid r_2$, then $l \mid s_3$, a contradiction. So, $l$ divides $s_3 + 2 a_1 r_2$.
Since $l$ divides $d_6$, we have $l \mid r_2^3 (A_0 - a_0)$.
So, $A_0 = a_0$. Since $l$ divides $d_5$, we see that $l \mid r_2^2 (B_1 - b_1)$.
We find that the coefficients of $f$ have at least two less degrees of freedom, namely $A_0 = a_0$, and the polynomials $B_1 - b_1$ and $s_3 + 2 a_1 r_2$
have a common prime divisor. So, we have $\dim \mathcal B_k \leq \dim \operatorname{Spec} R_k - 2$.
\end{itemize}
The $\mathbb C$-algebra automorphism $x \mapsto x + \alpha y$ of $\mathbb C[x, y, z, t, w]$ defines a morphism
\[
\pi_{\mathcal B_k, 1}\colon \operatorname{Spec} \mathcal B_k \times \operatorname{Spec} \mathbb C[\alpha] \to \operatorname{Spec} R_k
\]
and the $\mathbb C$-algebra automorphism $x \mapsto x + \alpha z$ of $\mathbb C[x, y, z, t, w]$ defines a morphism
\[
\pi_{\mathcal B_k, 2}\colon \operatorname{Spec} \mathcal B_k \times \operatorname{Spec} \mathbb C[\alpha] \to \operatorname{Spec} R_k.
\]
Every closed point $Q$ of $\operatorname{Spec} R_k$, where the fibre of $F_k \to \operatorname{Spec} R_k$ over $Q$ has a singular point different from $P_x$ with $t$-coordinate zero, belongs to the image of $\pi_{\mathcal B_k, 1}$ or~$\pi_{\mathcal B_k, 2}$. The union of the images of $\pi_{\mathcal B_k, 1}$ and $\pi_{\mathcal B_k, 2}$ has codimension at least~$1$.
\end{proof}

\begin{proof}[\textnormal{\textbf{Proof of \cref{mai:con} part \labelcref{itm:con theo general smooth}}}]
It follows from \cref{thm:con X general t non-zero,thm:con X general t zero} that a general sextic double solid in family $k$ has exactly one singular point, namely the point~$P_x$. The singularity of $X$ at $P_x$ is of type $cA_{\lfloor k\rfloor}$ if the homogeneous part $h_{\lfloor k\rfloor + 1} \in \mathbb C[y, z]$ of $h$ is non-zero, where $h$ is as in the proof of \cref{mai:con} \cref{itm:con theo every SDS is some X}. Since this is an open condition, a general sextic double solid in family $k$ has a $cA_{\lfloor k\rfloor}$ singularity at $P_x$.
\end{proof}

\subsection{Factoriality}

\begin{lemma}[{\cite[Lemma~5.1]{Kaw88}}] \label{thm:con fact iff Q-fact}
A terminal Gorenstein Fano 3-fold is factorial if and only if it is $\mathbb Q$-factorial.
\end{lemma}

\begin{lemma} \label{thm:con no Q-fact term members in 7-4}
There are no $\mathbb Q$-factorial log terminal sextic double solids in family~$7.4$.
\end{lemma}

\begin{proof}
Let $X$ be a log terminal variety in family~7.4. The Cartier divisor $\mathbb V_X(t)$ is the sum of the two prime divisors $D_1 = \mathbb V(t, q - w)$ and $D_2 = \mathbb V(t, q + w)$. Let $l \in \mathbb C[y, z]$ be a non-zero linear form that does not divide~$q$. Define the curve $C = V(q + w, x, l)$. If $D_1$ is $\mathbb Q$-Cartier, then $D_1 \cdot C = 0$, which contradicts \cref{pro:int restriction map on Picard groups,thm:int non-Cartier}. Therefore, neither $D_1$ nor $D_2$ is $\mathbb Q$-Cartier.
\end{proof}

Our proof of factoriality relies on the following corollary of \cref{thm:int Nam97 Pic}:

\begin{corollary} \label{thm:con Nam97 Cl}
Let $X$ be a Gorenstein terminal Fano 3-fold which is smooth along its general effective anti-canonical divisor~$D$. Then the natural homomorphism $\operatorname{Cl}(X) \to \operatorname{Pic}(D)$ from the class group of $X$ is injective.
\end{corollary}

\begin{proof}
Let $U$ be any Zariski open set in the smooth locus of~$X$ that contains~$D$. By \cref{rem:int Nam97 typos}\labelcref{itm:con injection from Pic X to Pic U}, we have an isomorphism of class groups $\operatorname{Cl}(X) \cong \operatorname{Cl}(U)$. Since $U$ is smooth, we have an isomorphism $\operatorname{Cl}(U) \cong \operatorname{Pic}(U)$. It follows from the proof of \cite[Proposition~2]{Nam97} that we can choose a small enough $U$ such that $\operatorname{Pic}(U)$ injects into $\operatorname{Pic}(D)$.
\end{proof}

\begin{corollary} \label{thm:con smooth along D with Pic 1 implies factorial}
Let $X$ be a terminal Gorenstein Fano 3-fold and $D$ a smooth effective \mbox{anti-canonical} divisor such that $X$ is smooth along~$D$ and $D$ has Picard number~$1$. Then $X$ is factorial.
\end{corollary}

\begin{proof}
By adjunction, every smooth anti-canonical divisor of a Fano variety is a K3 surface.
A very general projective K3 surface has Picard number~$1$.
Therefore, $X$ is smooth along an analytically very general anti-canonical divisor with Picard number~$1$.
By \cref{thm:con Nam97 Cl}, $X$ is $\mathbb Q$-factorial.
By \cref{thm:con fact iff Q-fact}, $X$ is factorial.
\end{proof}

\begin{lemma} \label{thm:con very gen SDS has Pic(V(x)) isom to Z except 7-4}
For every $k \in \mathrm{Inds} \setminus \{7.4\}$ and for an analytically very general sextic double solid $X$ in family~$k$, the subvariety $\mathbb V_X(x)$ is smooth and has Picard number~$1$.
\end{lemma}

\begin{proof}
Let $S_k$ be the $\mathbb C$-algebra freely generated by the $28$ coefficients, considered as variables, of polynomials $g \in \mathbb C[y, z, t]$ homogeneous of degree~$6$.
By \cref{rem:con notation}\labelcref{ite:con notation closed point defines f}, closed points $P$ of $\operatorname{Spec} R_k$ bijectively correspond to polynomials $f_P \in \mathbb C[x, y, z, t, w]$ in \cref{nota:con}. Let $\theta\colon \mathbb C[x, y, z, t, w] \to \mathbb C[y, z, t]$ be the homomorphism $x \mapsto 0$, $w \mapsto 0$. Let $\pi_k\colon \operatorname{Spec} R_k \to \operatorname{Spec} S_k$ be the morphism of affine spaces given on closed points by $f_P \mapsto \theta(f_P)$. The $\mathbb C$-algebra automorphisms $t \mapsto \alpha y + \beta z + t$ of $\mathbb C[y, z, t]$ induce a morphism $\tau\colon \operatorname{Spec} S_k \times \mathbb A^2 \to \operatorname{Spec} S_k$. Define $\rho_k$ to be the composition
\[
\rho_k := \tau \circ (\pi_k \times \mathrm{id}_{\mathbb A^2})\colon \operatorname{Spec} R_k \times \mathbb A^2 \to \operatorname{Spec} S_k.
\]
We can compute that the rank of the Jacobian matrix of $\rho_k$ at some specified point is~$28$ for all $k \in \mathrm{Inds} \setminus \{7.4\}$. It follows that $\rho_k$ is a dominant morphism of affine spaces for all $k \in \mathrm{Inds} \setminus \{7.4\}$.

The closed points $Q$ of $\operatorname{Spec} S_k$ bijectively correspond to polynomials $g_Q \in \mathbb C[y, z, t]$ homogeneous of degree~$6$, and therefore also to subschemes $Z_Q$ of $\mathbb P(1, 1, 1, 3)$ with variables $y, z, t, w$ given by $-w^2 + g_Q$. Smooth schemes $Z_Q \subseteq \mathbb P(1, 1, 1, 3)$ are K3 surfaces that are called \emph{sextic double planes}. It is known that a very general projective K3 surface has Picard number~$1$.
It follows that an analytically very general sextic double solid $X$ in family $k \in \mathrm{Inds} \setminus \{7.4\}$ satisfies that $\mathbb V_X(x)$ has Picard number~$1$.
\end{proof}

\begin{proof}[\textnormal{\textbf{Proof of \cref{mai:con} part \labelcref{itm:con very gen mfs}}}]
By \cref{mai:con} part~\labelcref{itm:con theo general smooth}, a general sextic double solid in family $k$ is terminal and is smooth along the anti-canonical divisor $\mathbb V(x)$. By \cref{thm:con smooth along D with Pic 1 implies factorial,thm:con very gen SDS has Pic(V(x)) isom to Z except 7-4}, an analytically very general sextic double solid in family $k \neq 7.4$ is factorial.
\end{proof}

\begin{remark}
In some cases we can prove that it suffices if the sextic double solid is only \emph{general} in \cref{mai:con} \cref{itm:con very gen mfs} as opposed to \emph{analytically very general}:
\begin{enumerate}[label=(\alph*), ref=\alph*]
\item A general sextic double solid in family~$1$ has only one singularity and that singularity is an ordinary double point. Every sextic double solid which is smooth outside an ordinary double point is factorial and has Picard number~$1$, see \cite[Theorem~B]{CP10}.
\item A general sextic double solid in family~$4$ is factorial, since in \cref{subsec:mod cA4} we construct a Sarkisov link to a complete intersection $Z_{5, 6} \subseteq \mathbb P(1, 1, 1, 2, 3, 4)$ which is $\mathbb Q$-factorial if it is general.
\end{enumerate}
\end{remark}

\subsection{\texorpdfstring
{Other $cA_n$ singularities}
{Other cAn singularities}}

Although the primary interest is in isolated $cA_n$ singularities since these are terminal, it is also possible to study non-isolated singularities with the same methods.

We describe uncountably many examples of sextic double solids with a non-isolated $cA_n$ singularity for all $9 \leq n \leq 11$.

\begin{proposition} \label{thm:con cA11 examples}
Let $8 \leq n \leq 11$. Let $r_2$ and $s_3$ be coprime and let $q_0$ be non-zero. Let $X$ in \cref{nota:con} satisfy Condition~$(n)$ but not satisfy Condition~$(n+1)$, where Conditions (9)--(12) are defined below:
\begin{enumerate}
\setcounter{enumi}{8}
\item Condition~$(8)$ of \cref{nota:con} is satisfied and there exists $B_0 \in \mathbb C$ such that \condNine
\item Condition~$(9)$ is satisfied and
\[
\begin{aligned}
  B_0 & = b_0\\
  d_2 & = 2 c_0 r_2 - 8 a_0 b_0 r_2 + 16 a_0^3 r_2 + 2 b_0 b_2 - 4 a_0^2 b_2 + b_1^2 - 8 a_0 a_1 b_1 - 4 a_1^2 b_0 + 24 a_0^2 a_1^2\\
  c_1 & = 2 a_0 b_1 + 2 a_1 b_0 - 12 a_0^2 a_1,
\end{aligned}
\]
\item Condition~$(10)$ is satisfied and
\[
\begin{aligned}
c_0 & = 2 a_0 b_0 - 4 a_0^3\\
d_1 & = b_0 b_1 - 2 a_0^2 b_1 - 4 a_0 a_1 b_0 + 8 a_0^3 a_1,
\end{aligned}
\]
\item Condition~$(11)$ is satisfied and $\begin{aligned}[t]
d_0 = b_0^2 - 4 a_0^2 b_0 + 4 a_0^4.
\end{aligned}$
\end{enumerate}
Then $P_x$ is a $cA_n$ singularity of~$X$. Moreover, if $n \geq 9$, then the singularity is non-isolated.
\end{proposition}

\begin{proof}
Use the explicit splitting lemma (\cref{thm:con splitting explicit}) and repeatedly apply \cref{thm:con solve equations key} similarly to the proof of \cref{mai:con} \cref{itm:con theo every SDS is some X,itm:con theo every SDS is some X}.
\end{proof}

\begin{remark} \label{rem:con non-isol sing C1 times ODP}
\begin{enumerate}
\item By the proof of \cref{itm:con theo bound m leq 8} of \cref{mai:con}, if $X$ in \cref{thm:con cA11 examples} satisfies Condition~$(9)$, then $X$ is singular along the curve $C\colon \mathbb V(t, w, s_3 + 2a_1r_2 + xr_2)$ passing through~$P_x$. We can compute that at a general point of $C$, the singularity is locally analytically $\mathbb C^1 \times \operatorname{ODP}$, that is, it is isomorphic to the germ $(Z, \bm 0)$ where $Z$ is $\mathbb V(x_1^2 + x_2^2 + x_3^2) \subseteq \mathbb C^4$ with variables $x_1, x_2, x_3, x_4$.

\item Translating the point $P_t = [0, 0, 0, 1, 0]$ to $[1, 0, 0, 0, 0]$, we can find conditions similar to \cref{nota:con} for having a $cA_n$ singularity at $P_t \in X$, which can be used to construct general sextic double solids with two $cA_n$ singularities. The following is a simple example with $cA_5$ singularities at $P_x$ and at $P_t$:
\[
\mathbb V(-w^2 + x^4 t^2 + x^2 t^4 + y^6 + z^6) \subseteq \mathbb P(1, 1, 1, 1, 3).
\]
\end{enumerate}
\end{remark}

\section{\texorpdfstring
{Divisorial contractions with centre a $cA_n$ point}
{Divisorial contractions with centre a cAn point}} \label{sec:wei}

In this section, we discuss weighted blowups from both algebraic and local analytic points of view. In \cref{thm:wei cAn wt correct implies Kawakita blup} we show that to check whether a weighted blowup is a Kawakita blowup (see \cref{thm:pre Kawakita blowup}), it suffices to compute the weight of the defining power series. Using this, in the technical \cref{thm:wei cAn p v} we show how to algebraically construct Kawakita blowups of $cA_n$ points on affine hypersurfaces.

\subsection{Weight-respecting maps}

Let $n$ and $m$ be positive integers. Let $\bm x = (x_1, \ldots, x_n)$ and $\bm y = (y_1, \ldots, y_m)$ denote the coordinates on $\mathbb C^n$ and $\mathbb C^m$, respectively. Choose positive integer weights for $\bm x$ and~$\bm y$.

\begin{definition}
Let $X \subseteq \mathbb C^n$ and $X' \subseteq \mathbb C^m$ be complex analytic spaces. We say that a biholomorphic map $\psi\colon X \to X'$ taking $\bm 0$ to $\bm 0$ is \textbf{weight-respecting} if denoting its inverse by~$\theta$, we can locally analytically around the origins write $\psi = (\psi_1, \ldots, \psi_m)$ and $\theta = (\theta_1, \ldots, \theta_n)$ where for all $i$ and~$j$, the power series $\psi_j \in \mathbb C\{\bm x\}$ and $\theta_i \in \mathbb C\{\bm y\}$ satisfy $\operatorname{wt}(\psi_j) \geq \operatorname{wt}(y_j)$ and $\operatorname{wt}(\theta_i) \geq \operatorname{wt}(x_i)$.
\end{definition}

It is known that a biholomorphic map taking the origin to the origin lifts to a unique biholomorphic map of the blown-up spaces under the usual weights $(1, \ldots, 1)$ (see for example \cite[Remark~3.17.1(4)]{GLS07}). It is easy to come up with examples where a biholomorphic map does not lift under weighted blowups. We give one example below.

\begin{example} \label{exa:wei same wt does not imply equiv blup}
Let $X \subseteq \mathbb C^3$ be the complex analytic space given by $\mathbb V(f)$ where
\[
f = x_2^2 x_3 + x_1^3 + a x_1 x_3^2 + b x_3^3
\]
for some $a, b \in \mathbb C^*$. Define $X' \subseteq \mathbb C^3$ by $\mathbb V(f')$ where $f' = f(x_1, x_2, -x_2 + x_3)$. Choose weights $(1, 1, 2)$ for $(x_1, x_2, x_3)$. Then, $X$ and $X'$ are biholomorphic and $\operatorname{wt} f = \operatorname{wt} f'$, but the weighted blowups of $X$ and $X'$ are not locally analytically equivalent.
\end{example}

\begin{proof}
Let $\psi\colon X \to X'$ be any local biholomorphism taking the origin to the origin. Composing with a suitable weight-respecting biholomorphic map and using \cref{thm:wei wt-resp implies equiv blup}, it suffices to consider the case where $\psi$ is a linear biholomorphism. Since the elliptic curve defined by $f$ in $\mathbb P^2$ with variables $x_1, x_2, x_3$ has only two automorphisms, there are only four possibilities for a linear biholomorphism $X \to X'$, namely $(x_1, x_2, x_3) \mapsto (x_1, \pm x_2, \pm x_2 + x_3)$.

Let $Y \to X$ and $Y' \to X'$ be the $(1, 1, 2)$-blowups of $X$ and $X'$ respectively. Then $Y$ is given by $\mathbb V(g)$ where
\[
g(u, x_1, x_2, x_3) = u x_2^2 x_3 + x_1^3 + a u^2 x_1 x_3^2 + b u^3 x_3^3.
\]
Denoting the points of $Y$ and $Y'$ by $[u, x_1, x_2, x_3]$, the lifted map $\psi_Y\colon Y \to Y'$ is given by $[u, x_1, x_2, x_3] \mapsto [u, x_1, \pm x_2, \pm x_2/u + x_3]$, which is not holomorphic on the exceptional locus $\mathbb V(u)$.
\end{proof}

On the other hand, a weight-respecting coordinate change does lift to weighted blowups, see \cref{thm:wei wt-resp implies equiv blup}.

\begin{lemma} \label{thm:wei wt-resp equiv to isom of algebras}
Let $X \subseteq \mathbb C^n$ and $X' \subseteq \mathbb C^m$ be complex analytic spaces strictly containing the origins and $\psi\colon X \to X'$ a biholomorphism. Then $\psi$ is weight-respecting if and only if $\psi$ induces an isomorphism of the $\mathbb Z_{\geq0}$-graded $\mathcal O_{X'}$-algebras $\mathcal B_{X'}$ and~$\psi_* \mathcal B_X$ of \cref{def:pre wei blup projan}.
\end{lemma}

\begin{proof}
``$\Longrightarrow$''.
The induced morphism $\mathcal B_{X'} \to \psi_* \mathcal B_X$ is given by
\[
\begin{aligned}
  \mathcal B_{X'}(U) & \to \mathcal B_X(\psi^{-1} U)\\
  t^d \bar y_j & \mapsto t^d \bar{\psi}_j,
\end{aligned}
\]
where $U \subseteq X'$ is open.
Since $\operatorname{wt} \psi_j \geq \operatorname{wt} y_j$, the morphism $\mathcal B_{X'} \to \psi_* \mathcal B_X$ is well-defined. Similarly, we define a morphism $\mathcal B_{X'} \gets \psi_* \mathcal B_X$, which is its inverse.

``$\Longleftarrow$''. Let $t^{\operatorname{wt} y_j} \Psi_j$ be the image of $t^{\operatorname{wt} y_j} y_j$ under $\mathcal B_{X'}(X') \to \psi_* \mathcal B_X(X)$. There exists $\psi_j \in \mathbb C\{\bm x\}$ such that $\operatorname{wt} \psi_j \geq \operatorname{wt} y_j$ and $\bar{\psi}_j = \Psi_j$. Similarly, we can find $\theta_i$, showing that $\psi$ is weight-respecting.
\end{proof}

\begin{corollary} \label{thm:wei wt-resp implies equiv blup}
A weight-respecting biholomorphism $\psi$ from $X \subseteq \mathbb C^n$ to $X' \subseteq \mathbb C^m$ lifts to the weighted blown-up spaces.
\end{corollary}

\subsection{Kawakita blowup in analytic neighbourhoods}

In the following, we focus on Kawakita blowups (see \cref{thm:pre Kawakita blowup}). Unlike \cref{exa:wei same wt does not imply equiv blup}, for $cA_n$ singularities, having the correct weight for the defining power series is enough for the local analytic equivalence of weighted blowups.

\begin{notation} \label{nota:wei}
We choose positive integer weights $\bm w = (r_1, r_2, a, 1)$ for variables $\bm x = (x_1, x_2, x_3, x_4)$ on $\mathbb C^4$ and define $n = (r_1 + r_2)/a - 1$ such that
\begin{itemize}
\item $a$ divides $r_1 + r_2$ and is coprime to both $r_1$ and $r_2$,
\item $r_1 \geq r_2$, and
\item $n \geq 2$.
\end{itemize}
\end{notation}

\begin{proposition} \label{thm:wei cAn wt correct implies Kawakita blup}
Using \cref{nota:wei}, let $f \in \mathbb C\{\bm x\}$ be such that $\mathbb V(f)$ has an isolated $cA_n$ singularity at the origin and $f$ has weight $r_1 + r_2$. Then, the $\bm w$-blowup of $\mathbb V(f) \subseteq \mathbb C^4$ is a $\bm w$-Kawakita blowup.
\end{proposition}

\begin{proof}
First, we remind that the terms \emph{homogeneous}, \emph{degree} and \emph{multiplicity} are with respect to the standard weights $(1, \ldots, 1)$.
Let the \emph{quadratic part} of $f$ denote the homogeneous part of~$f$ of degree~$2$.
After a suitable invertible linear weight-respecting coordinate change, the quadratic part of $f$ is $x_1 x_2$.

We find that $f = x_1 x_2 + x_1 G + H$, where $G \in \mathbb C\{x_1, \ldots, x_4\}$ has weight at least $r_2$ and multiplicity $m \geq 2$, and $H \in \mathbb C\{x_2, x_3, x_4\}$. The coordinate change $x_2 \mapsto x_2 - G_m$, where $G_m$ is the homogeneous degree $m$ part of~$G$, takes $f$ to $x_1 x_2 + x_1 G' + H'$, where $G'$ has multiplicity at least $m + 1$. By induction, this defines the unique formal power series $K \in \mathbb C[[x_1, \ldots, x_4]]$ of multiplicity at least~$2$ and weight at least~$r_2$ such that the transformation $x_2 \mapsto x_2 + K$ takes $f$ to the form $x_1 x_2 + H''$ where $H'' \in \mathbb C[[x_2, x_3, x_4]]$. Similarly, we transform $f$ into $x_1 x_2 + h$ where $h \in \mathbb C[[x_3, x_4]]$, using $x_1 \mapsto x_1 + L$ where $L \in \mathbb C[[x_2, x_3, x_4]]$.

We show how to find a convergent weight-respecting coordinate change which changes $f$ to $x_1 x_2 + h$. Instead of the coordinate changes $x_2 \mapsto x_2 + K$, $x_1 \mapsto x_1 + L$, which might not be convergent, we do a coordinate change $\Theta_N$ with truncated power series $K_{\leq N}$ and~$L_{\leq N}$ of homogeneous parts of $K$ and~$L$ of degree at most~$N$. The coordinate change $\Psi\colon x_1 \mapsto x_1 + i x_2$, $x_2 \mapsto x_1 - i x_2$ takes $x_1 x_2$ into $x_1^2 + x_2^2$. Now we use the splitting lemma, which gives a convergent coordinate change $\Phi_N$ which respects the weighting when $N$ is large enough, to give $f$ the form $x_1^2 + x_2^2 + h(x_3, x_4)$ where $h$ converges. Applying $\Psi^{-1}$, we get $x_1 x_2 + h$. Note that the coordinate changes $\Psi$ and $\Psi^{-1}$ might not respect the weighting~$\bm w$, but the total coordinate change $\Psi^{-1} \circ \Phi_N \circ \Psi \circ \Theta_N$ is weight-respecting if $N$ is large enough.

Since the singularity is $cA_n$ where $n = (r_1 + r_2)/a - 1$, $h$ must contain a monomial of degree $(r_1 + r_2)/a$. Since $x_1 x_2 + h$ has weight $r_1 + r_2$, if $a > 1$, then the coefficient of $x_3^{(r_1 + r_2)/a}$ in $h$ is non-zero. If $a = 1$, then after a suitable invertible linear coordinate change on $\mathbb C\{x_3, x_4\}$, the coefficient of $x_3^{(r_1 + r_2)/a}$ in $h$ is non-zero.

We found that we can transform $f$ into the form $x_1 x_2 + h$ where the coefficient of $x_3^{(r_1 + r_2)/a}$ in $h$ is non-zero, by using only weight-respecting coordinate changes. By \cref{thm:wei wt-resp implies equiv blup}, the weighted blowup of $f$ is locally analytically equivalent to the weighted blowup of $x_1 x_2 + h$, which is precisely a Kawakita blowup.
\end{proof}

Given a variety $X$ with an isolated $cA_n$ point $P$, we show that any two $\bm w$-Kawakita blowups $Y \to X$ and $Y' \to X$ of the point $P$ are locally analytically equivalent. Note that they need not be globally algebraically equivalent. For example, \cite[Remark~2.4]{CM04} describes two different $(2, 1, 1, 1)$-Kawakita blowups of a $cA_2$ singularity on a quartic 3-fold.

\begin{proposition} \label{thm:wei the Kawakita blowup}
Any two $\bm w$-Kawakita blowups of locally biholomorphic singularities are locally analytically equivalent.
\end{proposition}

\begin{proof}
Let $f = x_1 x_2 + g(x_3, x_4)$ and $f' = x_1 x_2 + g'(x_3, x_4)$ be contact equivalent, where $g, g' \in \mathbb C\{x_3, x_4\}$ have weight $r_1 + r_2$ and $x_3^{(r_1 + r_2)/a}$ appears in both $g$ and in $g'$ with non-zero coefficient. It suffices to show that there exist a weight-respecting map from $\mathbb V(f)$ to $\mathbb V(f')$.

Since $f$ and $f'$ are contact equivalent, there exists a unit $u \in \mathbb C\{\bm x\}$ and a local biholomorphism $\psi\colon (\mathbb C^4, \bm 0) \to (\mathbb C^4, \bm 0)$ such that $f' = u(f \circ \psi)$. Note that $f'$ and $f \circ \psi$ have the same weight $r_1 + r_2$, and $x_3^{(r_1 + r_2)/a}$ appears in $f \circ \psi$ with non-zero coefficient. It suffices to show that there exist a weight-respecting map from $\mathbb V(f)$ to $\mathbb V(f \circ \psi)$.

Using arguments similar to the proof of \cref{thm:wei cAn wt correct implies Kawakita blup}, we can find a weight-respecting biholomorphic map germ $\theta\colon (\mathbb C^4, \bm 0) \to (\mathbb C^4, \bm 0)$ such that $f \circ \psi \circ \theta$ is of the form $x_1 x_2 + g''$ where $g'' \in \mathbb C\{x_3, x_4\}$ contains $x^{(r_1 + r_2)/a}$ and has weight $r_1 + r_2$. It suffices to show that there exist a weight-respecting map from $\mathbb V(f)$ to $\mathbb V(f \circ \psi \circ \theta)$.

By \cref{thm:ana stable equivalence}, $g$ and $g''$ are right equivalent, meaning there exists an automorphism $\Phi$ of $\mathbb C\{x_3, x_4\}$ such that $\Phi(g) = g''$. Since $x_3^{(r_1 + r_2)/a}$ has non-zero coefficient in both $g$ and $g''$, and both $g$ and $g''$ have weight $r_1 + r_2$, the image of $x_3$ has weight~$a$ under both $\Phi$ and~$\Phi^{-1}$. Define the biholomorphic map germ $\varphi\colon (\mathbb V(f \circ \psi \circ \theta), \bm 0) \to (\mathbb V(f), \bm 0)$ by $\bm x \mapsto (x_1, x_2, \Phi(x_3), \Phi(x_4))$. By \cref{thm:wei wt-resp implies equiv blup}, the $\bm w$-blowups of $\mathbb V(f \circ \psi \circ \theta) \subseteq \mathbb C^4$ and $\mathbb V(f) \subseteq \mathbb C^4$ are locally analytically equivalent.
\end{proof}

\subsection{Kawakita blowups on affine hypersurfaces}

In this section, we see how to construct weighted blowups for affine hypersurfaces with a $cA_n$ singularity where $n \geq 2$ such that locally analytically they are Kawakita blowups.

Most $cA_n$ singularities do not admit $(r_1, r_2, a, 1)$-Kawakita blowups where $a \geq 2$. Below we define the \emph{type} of an isolated $cA_n$ singularity, which for $n \geq 2$ is equal to the highest integer $a$ such that it admits some $(r_1, r_2, a, 1)$-Kawakita blowup locally analytically. General sextic double solids with an isolated $cA_n$ singularity have a type~$1$ $cA_n$ singularity.

\begin{definition} \label{def:wei type}
Let $(X, P)$ be the complex analytic space germ of an isolated $cA_n$ singularity. Let $a$ be the largest integer such that $(X, P)$ is isomorphic to some germ $(\mathbb V(x_1 x_2 + g), \bm 0)$ where $g \in \mathbb C\{x_3, x_4\}$ has weight $a(n+1)$ under the weighting $(a, 1)$ for $(x_3, x_4)$. Then, we say that the $cA_n$ singularity is of \textbf{type $a$}.
\end{definition}

It is not obvious how to globally algebraically construct a Kawakita blowup for a variety with a $cA_n$ singularity. We show this for affine hypersurfaces in the technical \cref{thm:wei cAn p v}. We use a projectivization of \cref{thm:wei SDS cAn p} in \cref{sec:mod} for constructing Kawakita blowups of sextic double solids.

We describe the notation for \cref{thm:wei cAn p v}. Choose positive integers $n, r_1, r_2$ and $a$ as in \cref{nota:wei}. Let $F \in \mathbb C[x_1, x_2, x_3, x_4]$ have multiplicity at least~$3$, and let
\[
f = -x_1^2 + x_2^2 + F
\]
be such that $\mathbb V(f) \subseteq \mathbb C^4$ has terminal singularities and has a $cA_n$ singularity of type at least $a$ at the origin. Let $q, w$ be the power series when splitting with respect to $x_1$ (\cref{thm:con splitting}), and $p, v$ be the power series when splitting with respect to $x_2$, that is,
\begin{equation} \label{eqn:wei affine hyp f}
f = -((x_1 + q)w)^2 + ((x_2 + p)v)^2 + h
\end{equation}
where $q \in \mathbb C\{x_2, x_3, x_4\}$ and $p \in \mathbb C\{x_3, x_4\}$ both have multiplicity at least~2, and $w \in \mathbb C\{x_1, x_2, x_3, x_4\}$ and $v \in \mathbb C\{x_2, x_3, x_4\}$ are units, and $h \in \mathbb C\{x_3, x_4\}$ has multiplicity at least~3. If $a > 1$, then perform a coordinate change on $x_3, x_4$ for $f$ such that $h$ has weight $r_1 + r_2$.

Now choose weights
\[
\bm w = \operatorname{wt}(\alpha, \beta, x_3, x_4) = (r_1, r_2, a, 1)
\]
for the variables $\alpha, \beta, x_3, x_4$ on $\mathbb C^4$ and
\[
\bm w' = \operatorname{wt}(\alpha, \beta, x_1, x_2, x_3, x_4) = (r_1,\ r_2,\ m,\ \min(r_2, \operatorname{mult} p),\ a,\ 1)
\]
for the variables $\alpha, \beta, x_1, x_2, x_3, x_4$ on~$\mathbb C^6$, where $m = \min(r_2, \operatorname{mult} q)$. Writing a power series $s \in \mathbb C\{x_1, x_2, x_3, x_4\}$ as a sum of its $\bm w'$-weighted homogeneous parts $s = \sum_{i=0}^\infty s_i$, let $s_{<k}$ denote $\sum_{i<k} s_i$ and $s_{\geq k}$ denote $\sum_{i\geq k} s_i$. Define the ideal
\[
I = (f,\ -\alpha + (x_1 + q_{<r_1})w_{<r_1-m} + (x_2 + p_{<r_1})v_{<r_1-r_2},\ -\beta + x_2 + p_{<r_2})
\]
of $\mathbb C[\alpha, \beta, x_1, x_2, x_3, x_4]$, where $v_{<r_1-r_2}$ is defined to be $1$ when $r_1 = r_2$ and where $w_{<r_1-m}$ is defined to be $1$ when $r_1 = m$. Note that the affine varieties $\mathbb V(f) \subseteq \mathbb C^4$ and $\mathbb V(I) \subseteq \mathbb C^6$ are isomorphic.

\begin{lemma} \label{thm:wei cAn p v}
Using the notation above, the $\bm w'$-blowup of $\mathbb V(I)$ is a $\bm w$-Kawakita blowup.
\end{lemma}

\begin{proof}
The morphism
\[
\begin{aligned}
\varphi\colon \mathbb C^4 & \to \mathbb C^4\\
(x_1, x_2, x_3, x_4) & \mapsto ((x_1 + q_{<r_1})w_{<r_1-m} + (x_2 + p_{<r_1})v_{<r_1-r_2},\ x_2 + p_{<r_2},\ x_3,\ x_4)
\end{aligned}
\]
has a local analytic inverse~$\varphi^{-1}$, given by
\[
\begin{aligned}
\varphi^{-1}\colon (\mathbb C^4, \bm 0) & \to (\mathbb C^4, \bm 0)\\
(\alpha, \beta, x_3, x_4) & \mapsto ((\alpha - (\beta - p_{<r_2} + p_{<r_1})v')u - q',\ \beta - p_{<r_2},\ x_3,\ x_4)
\end{aligned}
\]
where $u \in \mathbb C\{\alpha, \beta, x_3, x_4\}$ is a unit, $v' = v_{<r_1-r_2}(\beta - p_{<r_2}, x_3, x_4)$ and $q' = q_{<r_1}(\beta - p_{<r_2}, x_3, x_4)$. Define the map germ
\[
\begin{aligned}
\psi\colon (\mathbb C^4, \bm 0) & \to (\mathbb C^6, \bm 0)\\
(\alpha, \beta, x_3, x_4) & \mapsto (\alpha, \beta, \varphi^{-1}(\alpha, \beta, x_3, x_4)).
\end{aligned}
\]
The restriction of $\psi$ to $\mathbb V(I) \to \mathbb V(f \circ \psi)$ is a weight-respecting local biholomorphism, whose inverse is a projection. Therefore, the $\bm w$-blowup of $\mathbb V(f \circ \psi)$ is equivalent to the $\bm w'$-blowup of $\mathbb V(I)$. If the $\bm w$-weight of $f \circ \psi$ is $r_1 + r_2$, then by \cref{thm:wei cAn wt correct implies Kawakita blup}, the $\bm w$-blowup of $\mathbb V(f \circ \psi)$ is the $\bm w$-Kawakita blowup map germ. Using \cref{eqn:wei affine hyp f}, it suffices to show that
\begin{align}
\operatorname{wt}[((x_1 + q)w + (x_2 + p)v) \circ \psi] & = r_1 \label{eqn:wei affine hyp r1}\\*
\operatorname{wt}[(-(x_1 + q)w + (x_2 + p)v) \circ \psi] & = r_2 \label{eqn:wei affine hyp r2}.
\end{align}
Since $\psi$ is weight-respecting, we have
\begin{align*}
\operatorname{wt}[(x_1 + q)w_{\geq r_1-m} \circ \psi)] & \geq r_1\\*
\operatorname{wt}[q_{\geq r_1}w_{<r_1-m} \circ \psi)] & \geq r_1\\
\operatorname{wt}[(x_2 + p)v_{\geq r_1-r_2} \circ \psi] & \geq r_1\\*
\operatorname{wt}[p_{\geq r_1} v_{<r_1-r_2} \circ \psi] & \geq r_1.
\end{align*}
Since $((x_1 + q_{<r_1})w_{<r_1-m} + (x_2 + p_{<r_1})v_{<r_1-r_2}) \circ \psi = \alpha$, this proves \cref{eqn:wei affine hyp r1}. Using in addition that $\operatorname{wt}[(x_2 + p_{<r_1})v_{<r_1-r_2} \circ \psi] = r_2$, \cref{eqn:wei affine hyp r2} follows.
\end{proof}

\begin{corollary} \label{thm:wei SDS cAn p}
Using the notation above, if $F \in \mathbb C[x_2, x_3, x_4]$, or equivalently, if $q = 0$ and $w = 1$, then define the ideal $J \subseteq \mathbb C[\alpha, \beta, x_2, x_3, x_4]$ by
\begin{equation} \label{eqn:wei cAn p v - gen}
J = (-(\alpha - (x_2 + p_{<r_1})v_{<r_1-r_2})^2 + x_2^2 + F,\ -\beta + x_2 + p_{<r_2}),
\end{equation}
where $v_{<r_1-r_2}$ is defined to be $1$ if $r_1 = r_2$. Then, $\mathbb V(J)$ and $\mathbb V(f)$ are isomorphic affine varieties, and the $(r_1, r_2, \min(r_2, \operatorname{mult} p), a, 1)$-blowup of $\mathbb V(J)$ is a $\bm w$-Kawakita blowup. If in addition $r_1 = r_2$, then define the ideal $J' \subseteq \mathbb C[x_1, \beta, x_2, x_3, x_4]$ by
\begin{equation} \label{eqn:wei cAn p v - r1 r2 equal}
J' = (f, -\beta + x_2 + p_{<r_2}).
\end{equation}
Then, $\mathbb V(J')$ and $\mathbb V(f)$ are isomorphic affine varieties, and the $(r_1, r_2, \min(r_2, \operatorname{mult} p), a, 1)$-blowup of $\mathbb V(J')$ is a $\bm w$-Kawakita blowup.
\end{corollary}

\begin{proof}
The isomorphism between $\mathbb V(I)$ and $\mathbb V(J)$ is a projection, with inverse given by $x_1 \mapsto \alpha - (\beta - p_{<r_2} + p_{<r_1})v_{<r_1-r_2}$, which is weight-respecting. If $r_1 = r_2$, the isomorphism between $\mathbb V(J)$ and $\mathbb V(J')$ is given by $x_1 \mapsto \alpha - \beta$, which is weight-respecting.
\end{proof}

The power series $p$, $v$, $q$, $w$ can be expressed in terms of the coefficients of $F$ using the explicit splitting lemma, \cref{thm:con splitting explicit}.

\section{Birational models of sextic double solids} \label{sec:mod}

In this section, we prove \cref{mai:mod} on birational non-rigidity of certain sextic double solids. First, we give the generality conditions that we use.

\begin{definition} \label{def:mod generality conds}
Let $X$ be a sextic double solid in \cref{nota:con}. Let $\mathbb P(1, 1, 3)$ have variables $y, z, w$ and let $\mathbb P^1$ have variables $y, z$. Define the following generality conditions, depending on the family that $X$ lies in:
\begin{enumerate}[leftmargin=4em]
\item[{\hyperref[subsec:mod cA4]{$(4)$}}] $\mathbb V(2 w a_2 + c_5, w^2 - d_6) \subseteq \mathbb P(1, 1, 3)$ is $10$ distinct points,
\item[{\hyperref[subsec:mod cA5]{$(5)$}}] $\mathbb V(a_2, -w^2 + d_6) \subseteq \mathbb P(1, 1, 3)$ is $4$ distinct points,
\item[{\hyperref[subsec:mod cA6]{$(6)$}}] $c_4 - 2 a_1 b_3 - a_2 b_2 + 2 a_0 a_2^2 + 6 a_1^2 a_2 \in \mathbb C[y, z]$
is non-zero, and $\mathbb V(a_2) \subseteq \mathbb P^1$ is two distinct points, and for both of these points~$P$, one of $b_3(P)$, $c_4(P)$ or $d_5(P)$ is non-zero,
\item[{\hyperref[subsec:mod cA7-1]{$(7.1)$}}] $\mathbb V(-e_2 + 4 a_0 r_2 + b_2 - 6 a_1^2) \subseteq \mathbb P^1$
is two distinct points,
\item[{\hyperref[subsec:mod cA7-2]{$(7.2)$}}] $r_1$ and $q_1$ are coprime in~$\mathbb C[y, z]$,
\item[{\hyperref[subsec:mod cA7-3]{$(7.3)$}}] $q_2 \in \mathbb C[y, z]$ is not a square,
\item[{\hyperref[subsec:mod cA8]{$(8)$}}] $a_0 \neq A_0$.
\end{enumerate}
\end{definition}

\begin{maintheorem} \label{mai:mod}
Every terminal $\mathbb Q$-factorial sextic double solid with a $cA_n$ singularity with $n \geq 4$ that satisfies the generality conditions in \cref{def:mod generality conds} has a Sarkisov link starting with a weighted blowup with centre the $cA_n$ point.
\end{maintheorem}

We treat each of the 7 families separately. We use the notation in \cref{cons:pre,exa:pre cA4 toric link} for the 2-ray links. We write the $cA_4$ case in more detail. Below, when we say that a birational map is $m$ Atiyah flops, then we mean that the base of the flop is $m$ points, above each we are contracting a curve and extracting a curve, and locally analytically above each of the points it is an Atiyah flop (see \cite[Section~1.3]{Rei92} for the Atiyah flop). Similarly for flips. Below, for a morphism $\Phi\colon T_0 \to \mathbb P$, we let $\Phi^*\colon \operatorname{Cox} \mathbb P \to \operatorname{Cox} T_0$ denote a corresponding $\mathbb C$-algebra homomorphism of Cox rings (this is described explicitly in the proof of \cref{thm:mod cA4}).

\subsection{\texorpdfstring
{Singularities after divisorial contraction}
{Singularities after divisorial contraction}}

The non-Gorenstein singularities on $Y$ for an ordinary type divisorial contraction $Y \to X$ with centre a $cA_n$ singularity can be easily found using the result by Kawakita, \cref{thm:pre Kawakita blowup}. On the other hand, the structure or the number of Gorenstein singularities is unclear. We show in \cref{thm:mod Y only up to two quot sing} that if $X$ in one of the 11 families is general, then $Y$ has no Gorenstein singularities. We do not give the generality conditions of \cref{thm:mod Y only up to two quot sing} explicitly. We do not need \cref{thm:mod Y only up to two quot sing} for proving \cref{mai:mod}.

\begin{lemma} \label{thm:mod elementary}
Let $a, b \in \mathbb C[y, z]$ be non-zero homogeneous polynomials with $\deg a \geq \deg b$ such that for every homogeneous polynomial $c \in \mathbb C[y, z]$ of degree $\deg a - \deg b$, the polynomial $a + bc$ is divisible by the square of a linear form. Then $a$ and $b$ are both divisible by the square of the same linear form.
\end{lemma}

\begin{proof}
Suffices to prove that for polynomials $f, g \in \mathbb C[x]$, if $f + \lambda g$ has a repeated root for infinitely $\lambda \in \mathbb C$, then $f$ and $g$ have a common repeated root. Dividing $f$ and $g$ by suitable linear polynomials, it suffices to consider the case where every common root of $f$ and $g$ is a common repeated root of $f$ and~$g$.

If the set
\[
A = \mleft\{ \alpha \in \mathbb C \;\middle|\; \text{$\alpha$ is a repeated root of $f + \lambda_\alpha g$ for some $\lambda_\alpha \in \mathbb C$} \mright\}
\]
is finite, then there exist $\alpha \in \mathbb C$ and $\lambda_1 \neq \lambda_2$ such that $\alpha$ is a repeated root of both $f + \lambda_1 g$ and $f + \lambda_2 g$. It follows that $\alpha$ is a repeated root of both $f$ and $g$.

Without loss of generality, both $f$ and $g$ are non-constant. Subtracting
$g \cdot \dfrac{\mathrm d (f + \lambda g)}{\mathrm d x}$
from
$\dfrac{\mathrm d g}{\mathrm d x} \cdot (f + \lambda g)$,
we find that a repeated root of $f + \lambda g$ is necessarily a root of
$f \dfrac{\mathrm d g}{\mathrm d x} - g \dfrac{\mathrm d f}{\mathrm d x}$.
If $f \dfrac{\mathrm d g}{\mathrm d x} = g \dfrac{\mathrm d f}{\mathrm d x}$, then a prime factor of $g$ is a prime factor of~$f$. If
$f \dfrac{\mathrm d g}{\mathrm d x} \neq g \dfrac{\mathrm d f}{\mathrm d x}$,
then the set $A$ is finite. In both cases, $f$ and $g$ have a common repeated root.
\end{proof}

\begin{proposition} \label{thm:mod Y only up to two quot sing}
Let $X$ be a member of family $k \in \mathrm{Inds}$ of \cref{nota:con} which is smooth outside a $cA_{\lfloor k\rfloor}$ singularity at $P_x = [1, 0, 0, 0, 0]$. Let $Y \to X$ be a divisorial contraction with centre~$P_x$, which is a $(r_1, r_2, 1, 1)$-Kawakita blowup. If $X$ is general, then $Y$ has a quotient singularity $1/r_1(1, 1, r_1 - 1)$ if $r_1 > 1$ and a quotient singularity $1/r_2(1, 1, r_2 - 1)$ if $r_2 > 1$ and is smooth elsewhere.
\end{proposition}

\begin{proof}
We show that $Y$ has only up to two quotient singularities on the exceptional divisor and is smooth elsewhere. Since $Y \to X$ is a $(r_1, r_2, 1, 1)$-Kawakita blowup, we can consider the local analytic coordinate system around $P_x$ where $X$ is given by $wt + h(y, z)$ where $h \in \mathbb C\{y, z\}$ has multiplicity~$n+1$. The variety $Y$ is locally analytically around the exceptional divisor given by $wt + \frac{1}{u^{n+1}}h(uy, uz)$ inside the geometric quotient $(\mathbb C^5 \setminus \mathbb V(w, t, y, z))/\mathbb C^*$ where the $\mathbb C^*$-action is given by $\lambda \cdot (u, w, t, y, z) = (\lambda^{-1} u, \lambda^{r_1} w, \lambda^{r_2} t, \lambda y, \lambda z)$. The singular locus of $Y$ is given by
\[
\operatorname{Sing} Y = \mathbb V\mleft( u, w, t, h_{n+1}, \frac{\partial h_{n+1}}{\partial y}, \frac{\partial h_{n+1}}{\partial z}, h_{n+2} \mright) \cup \{P_w\}_{\text{if $r_1 > 1$}} \cup \{P_t\}_{\text{if $r_2 > 1$}},
\]
where $h_i$ denotes the homogeneous degree $i$ part of~$h$, and $P_w$ and $P_t$ are the points $[0, 1, 0, 0, 0]$ and $[0, 0, 1, 0, 0]$, respectively. We see that $Y$ is singular outside of $P_w$ and $P_t$ if and only if there exists a homogeneous linear form $L \in \mathbb C[y, z]$ such that $L^2$ divides $h_{n+1}$ and $L$ divides $h_{n+2}$. For all $k \in \mathrm{Inds}$, exactly one of the following holds:
\begin{itemize}
\item $Y \setminus \{P_w, P_t\}$ is smooth for a general $X$ in family~$k$, or
\item for all $X$ in family~$k$, there exists a homogeneous linear form $L \in \mathbb C[y, z]$ such that $L^2$ divides $h_{n+1}$ and $L$ divides $h_{n+2}$.
\end{itemize}

We write the proof for the family 8 in detail, the proofs for the other 10 families are similar. Using the explicit splitting lemma (\cref{thm:con splitting explicit}), we compute that
\[
h_9 = Q - 2 d_3 r_2^3 = 8 (a_0 - A_0) s_3^3 + r_2 R,
\]
where $Q, R \in \mathbb C[y, z]$ are homogeneous of degrees $9$ and $7$ respectively, and $Q$ does not contain the polynomial~$d_3$. Assume that for all $X$ in family~8, there exists $L$ such that $L^2$ divides~$h_9$. Using \cref{thm:mod elementary} with $(a, b, c) = (Q,\, r_2^3,\, -2 d_3)$, we find that a prime divisor of $r_2$ divides~$h_9$. Therefore, a general member $X$ of family~8 satisfies that $r_2$ and $s_3$ have a common prime divisor, contradicting \cref{mai:con} \cref{itm:con theo general smooth} and \cref{thm:con cA7 no common prime divisor}. So, for a general $X$ in family~8, $Y \setminus \{P_w, P_t\}$ is smooth.
\end{proof}

\subsection{\texorpdfstring
{$cA_4$ model}
{cA4 model}} \label{subsec:mod cA4}

Note that Okada described a Sarkisov link starting from a general complete intersection $Z_{5, 6} \subseteq \mathbb P(1, 1, 1, 2, 3, 4)$ to a sextic double solid (see entry No.~9 of the table in \cite[Section~9]{Oka14}). We show the converse:

\begin{proposition} \label{thm:mod cA4}
A sextic double solid with a $cA_4$ singularity satisfying \cref{def:mod generality conds} has a Sarkisov link to a complete intersection $Z_{5, 6} \subseteq \mathbb P(1, 1, 1, 2, 3, 4)$, starting with a $(3, 2, 1, 1)$-blowup of the $cA_4$ point, then $10$ Atiyah flops, and finally a Kawamata divisorial contraction (see \cite{Kaw96}) to a terminal quotient $1/4(1, 1, 3)$ point. Under further generality conditions (\cref{thm:mod Y only up to two quot sing}), $Z$ is quasismooth.
\end{proposition}

\begin{proof}
We exhibit the diagram below.
\begin{equation*}
\begin{tikzcd}[column sep = small]
& \arrow[ld, "{(3, 2, 1, 1)}"'] Y_0 \arrow[rd, ""'] \arrow[rr, dashed, "{10 \times (1, 1, -1, -1)}"] & & Y_1 \arrow[ld, ""'] \arrow[rd, "{(\frac14, \frac14, \frac34)}"] \\
cA_4 \in X \subseteq \mathbb P(1^4, 3) \hspace{-3em} & & W_0 & & \hspace{-2em} 1/4(1, 1, 3) \in Z_{5, 6} \subseteq \mathbb P(1^3, 2, 3, 4)
\end{tikzcd}
\end{equation*}
The corresponding diagram for the ambient toric spaces is given in detail in \cref{exa:pre cA4 toric link}.

By \cref{mai:con}, every sextic double solid $\hat X$ with an isolated $cA_4$ singularity can be given by
\[
\hat X\colon \mathbb V(\hat f) \subseteq \mathbb P(1, 1, 1, 1, 3)
\]
with variables $x, y, z, t, w$ where
\[
\hat f = -w^2 + x^4 t^2 + 2 x^3 t a_2 + x^3 t^2 A_1 + x^2 a_2^2 + x^2 t B_3 + x C_5 + D_6,
\]
where $a_2 \in \mathbb C[y, z]$ is homogeneous of degree~$2$, and $A_i, B_i, C_i, D_i \in \mathbb C[y, z, t]$ are homogeneous of degree~$i$.

Below, we perform the following constructions:
\begin{enumerate}
\item \label{itm:mod cA4 step P} we define a weighted projective space $\mathbb P = \mathbb P(1, 1, 1, 1, 3, 5)$,
\item \label{itm:mod cA4 step X} we define a subvariety $X$ of $\mathbb P$ by explicitly describing a homogeneous ideal,
\item we show that $X$ and $\hat X$ are isomorphic by constructing an explicit isomorphism,
\item \label{itm:mod cA4 step T0} we construct a toric variety $T_0$,
\item we define a morphism $\Phi\colon T_0 \to \mathbb P$,
\item \label{itm:mod cA4 step Y0} we construct a subvariety $Y_0$ of $T_0$ by explicitly describing a bihomogeneous ideal $I_Y$ of the Cox ring of~$T_0$,
\item we restrict the morphism $\Phi$ to $Y_0$ and check that its image is~$X$.
\end{enumerate}
Although computational, the above steps are completely elementary. The reason for these constructions is that, as we prove below, the morphism $Y_0 \to X$ is the $(3, 2, 1, 1)$-Kawakita blowup and $I_Y$ 2-ray follows~$T_0$.

The reader might have the philosophical question of how the author found the varieties $\mathbb P, X, T_0$ and $Y_0$, described below, and why they are defined exactly as they are. In \cref{thm:mod cA4 finding X}, we describe the methods we used to arrive at the construction of $\mathbb P, X, T_0$ and $Y_0$. Note that the choices involved in \labelcref{itm:mod cA4 step P,itm:mod cA4 step X,itm:mod cA4 step T0,itm:mod cA4 step Y0} above are somewhat arbitrary. Namely, there exist other varieties $\mathbb P, X, T_0$ and $Y_0$ such that $Y_0 \to X$ is the $(3, 2, 1, 1)$-Kawakita blowup and $I_Y$ 2-ray follows~$T_0$.

We start by constructing~$X$. Define the bidegree $(5, 6)$ complete intersection $X$, isomorphic to~$\hat X$, by
\[
X\colon \mathbb V(f, -x \xi + \alpha^2 - D_6) \subseteq \mathbb P(1, 1, 1, 1, 3, 5)
\]
with variables $x, y, z, t, \alpha, \xi$, where
\[
f = -\xi + 2 \alpha a_2 + 2 \alpha x t + x^2 t^2 A_1 + x t B_3 + C_5.
\]
The isomorphism is given by
\[
\begin{aligned}
\hat X & \to X\\
[x, y, z, t, w] & \mapsto \mleft[ x, y, z, t, \alpha', 2 \alpha' a_2 + 2 \alpha' x t + x^2 t^2 A_1 + x t B_3 + C_5 \mright]
\end{aligned}
\]
where $\alpha' = w + x^2 t + x a_2$,
with inverse
\[
\mleft[ x, y, z, t, \alpha, \xi \mright] \mapsto [x, y, z, t, \alpha - x^2 t - x a_2].
\]

We describe the divisorial contraction $\varphi\colon Y_0 \to X$. Define the toric variety
\[
\begin{array}{ccc|ccccccccc}
            &  u & x & y & z & \alpha & \xi & t &\\
\Ta{T_0\colon} &  0 & 1 & 1 & 1 & 3 & 5 & 1 & \Tb{,}\\
            & -1 & 0 & 1 & 1 & 3 & 6 & 2 &
\end{array}
\]
as in \cref{exa:pre cA4 toric link}. Let $\Phi$ be the ample model of $\mathbb V(x)$, that is,
\[
\begin{aligned}
\Phi\colon T_0 & \to \mathbb P(1, 1, 1, 1, 3, 5)\\
[u, x, y, z, \alpha, \xi, t] & \mapsto [x, uy, uz, u^2t, u^3\alpha, u^6\xi].
\end{aligned}
\]
Let $Y_0$ be the strict transform of~$X$. Let $\Phi^*$ denote the corresponding $\mathbb C$-algebra homomorphism, namely
\begin{gather*}
\Phi^*\colon \mathbb C[x, y, z, t, \alpha, \xi] \to \mathbb C[u, x, y, z, \alpha, \xi, t]\\*
\Phi^*\colon x \mapsto x,\ y \mapsto uy,\ z \mapsto uz,\ t \mapsto u^2 t,\ \alpha \mapsto u^3 \alpha,\ \xi \mapsto u^6 \xi.
\end{gather*}
Define
\[
A_Y = A_1(y, z, ut), \qquad B_Y = B_3(y, z, ut), \qquad C_Y = C_5(y, z, ut), \qquad D_Y = D_6(y, z, ut)
\]
and define the polynomial $g = \Phi^*f / u^5$, that is,
\[
\begin{aligned}
g = -u \xi + 2 \alpha a_2 + 2 \alpha x t + x^2 t^2 A_Y + x t B_Y + C_Y.
\end{aligned}
\]
Then, $Y_0$ is given by
\[
Y_0\colon \mathbb V(I_Y) \subseteq T_0 \ \text{ where } \ I_Y = (g,\, -x \xi + \alpha^2 - D_Y).
\]
We will see later that $I_Y$ 2-ray follows~$T_0$. Note that there exist other ideals that define the same variety $Y_0 \subseteq T_0$ (see \cite[Corollary~3.9]{Cox95}), but where the ideal might not 2-ray follow~$T_0$. Also note that we have not (and do not need to) prove that the ideal $I_Y$ is saturated with respect to~$u$, although in general, saturating might help in finding the ideal that 2-ray follows~$T_0$. The morphism $Y_0 \to X$ is the restriction of $T_0 \to \mathbb P(1, 1, 1, 1, 3, 5)$. Locally, $(Y_0)_x \to X_x$ is the $(3, 2, 1, 1)$-blowup of $\mathbb V(f') \subseteq \mathbb C^4$ with variables $\alpha, t, y, z$, where
\[
f' = -\alpha^2 + 2 \alpha a_2 + 2 \alpha t + t^2 A_1 + t B_3 + C_5 + D_6.
\]
Since $\operatorname{wt} f' = 5$, by \cref{thm:wei cAn wt correct implies Kawakita blup}, $(Y_0)_x \to X_x$ is a $(3, 2, 1, 1)$-Kawakita blowup.

The first diagram in the 2-ray game for $Y_0$ is $10$ Atiyah flops, under \cref{def:mod generality conds}. We describe the diagram $Y_0 \to W_0 \gets Y_1$ globally. Multiplying the action matrix of $T_0$ by the matrix $\smat{1 & 0\\-1 & 1}$, define
\[
\begin{array}{cccccc|ccc}
            &  u &  x & y & z & \alpha & \xi & t & \\
\Ta{T_1\colon} &  0 &  1 & 1 & 1 & 3 & 5 & 1 & \Tb{.}\\
            & -1 & -1 & 0 & 0 & 0 & 1 & 1 &
\end{array}
\]
Define $Y_1$ by $\mathbb V(I_Y) \subseteq T_1$. Define the morphisms $Y_0 \to W_0$ and $Y_1 \to W_0$ as the ample models of $\mathbb V(y)$. The exceptional locus of $Y_0 \to W_0$ is $E_0^- = \mathbb V(\xi, t) \subseteq Y_0$, the exceptional locus of $Y_1 \to W_0$ is $E_1^+ = \mathbb V(u, x) \subseteq Y_1$, and the base of the flop is
\[
\{P_i\} = \mathbb V(2 \alpha a_2 + C_5(y, z, 0),\, \alpha^2 - D_6(y, z, 0)) \subseteq \mathbb P(1, 1, 3) \subseteq W_0,
\]
where $\mathbb P(1, 1, 3)$ has variables $y, z, \alpha$. If $a_2$, $C_5(y, z, 0)$ and $D_5(y, z, 0)$ are general enough, that is, if \cref{def:mod generality conds} is satisfied, then the base of the flop is $10$ points~$\{P_i\}_{1\leq i\leq10}$, and both $E_0^-$ and $E_1^+$ are $10$ disjoint curves mapping to~$\{P_i\}_{1\leq i\leq10}$.

We show that locally analytically, the diagram $Y_0 \to W_0 \gets Y_1$ is $10$ Atiyah flops. Let $P \in W_0$ be any point in the base of the flop. Then, $P$ has $y$ or $z$ coordinate non-zero. We consider the case where the $y$-coordinate is non-zero, the other case is similar. Since the base of the flop is $10$ points, the point $P$ is smooth in $\mathbb P(1, 1, 3)$. By the implicit function theorem, we can locally analytically equivariantly express $\alpha$ and $z$ in terms of the variables $u, x, \xi, t$ on the patches $(Y_0)_y$, $(W_0)_y$ and~$(Y_1)_y$. So, the flop $Y_0 \to W_0 \gets Y_1$ is locally analytically a $(1, 1, -1, -1)$-flop, the so-called Atiyah flop, around~$P$.

The last morphism $Y_1 \to Z$ in the link for $X$ is a divisorial contraction. Multiplying the action matrix of $T_0$ by the matrix $\smat{6 & -5\\2 & -1}$ with determinant~$4$, we see that
\[
\begin{array}{cccccc|ccc}
              & u & x & y & z & \alpha & \xi &  t & \\
\Ta{T_1\cong} & 5 & 6 & 1 & 1 & 3 & 0 & -4 & \Tb{.}\\
              & 1 & 2 & 1 & 1 & 3 & 4 &  0 &
\end{array}
\]
Let $Y_1 \to Z$ be the ample model of $\frac14 \mathbb V(\xi)$, that is,
\[
\begin{aligned}
Y_1 & \to Z\\
[u, x, y, z, \alpha, \xi, t] & \mapsto \mleft[ t^{\frac54} u, t^{\frac14} y, t^{\frac14} z, t^{\frac32} x, t^{\frac34} \alpha, \xi \mright].
\end{aligned}
\]
Then $Z$ is the bidegree $(5, 6)$ complete intersection
\[
Z\colon \mathbb V(h, -x \xi + \alpha^2 - D_6(y, z, u)) \subseteq \mathbb P(1, 1, 1, 2, 3, 4)
\]
with variables $u, y, z, x, \alpha, \xi$, where the $h$ is given by applying the $\mathbb C$-algebra homomorphism $t \mapsto 1$ to~$g$.
The morphism $Y_1 \to Z$ contracts the exceptional divisor $\mathbb V(t) \subseteq Y_1$ to the point~$P_\xi = [0, 0, 0, 0, 0, 1]$. On the quasiprojective patch~$(Y_1)_\xi$, we can express $u$ and $x$ locally analytically equivariantly in terms of $y, z, \alpha, t$. So, the morphism $Y_1 \to Z$ is locally analytically the Kawamata weighted blowdown (see \cite{Kaw96}) to the terminal quotient singular point $P_\xi$ of type $1/4(1, 1, 3)$.
\end{proof}

\begin{remark} \label{thm:mod cA4 finding X}
We explain below how we found the variety~$X$ used in \cref{thm:mod cA4}. We start with the variety $\hat X$, given by \cref{mai:con}. Note that it is not possible to assign weights to the coordinates of $\mathbb P(1, 1, 1, 1, 3)$ such that the corresponding weighted blowup of $\hat X$ would be a $(3, 2, 1, 1)$-Kawakita blowup. To amend this, we first replace the variety $\hat X$ by a variety $\bar X$ such that choosing the weights appropriately, the weighted blowup of $\bar X$ is the $(3, 2, 1, 1)$-Kawakita blowup. So far the process is algorithmic. Unfortunately, as we see below, the constructed ideal $I_{\bar Y}$ does not 2-ray follow the ambient toric variety~$\bar T_0$. Using the technique known as ``unprojection'', we construct another toric variety $T_0$ and a subvariety~$Y_0$ given by an ideal~$I_Y$. This time we are lucky, since as the proof of \cref{thm:mod cA4} shows, the ideal $I_Y$ 2-ray follows~$T_0$. Note that the variety $Y_0$ has higher codimension in $T_0$ than $\bar Y_0$ had in~$\bar T_0$. We give details below.

We perform the coordinate change $\hat X \to \bar X$ given in \cref{eqn:wei cAn p v - gen} of \cref{thm:wei SDS cAn p}, with $(r_1, r_2, a, 1) = (3, 2, 1, 1)$, $p_2 = a_2$ and $v_0 = 1$. We see that $\hat X$ is isomorphic to
\[
\bar X\colon \mathbb V(\bar f) \subseteq \mathbb P(1, 1, 1, 1, 3)
\]
with variables $x, y, z, t, \alpha$, where
\[
\bar f = \alpha (-\alpha + 2 x^2 t + 2 x a_2) + x^3 t^2 A_1 + x^2 t B_3 + x C_5 + D_6.
\]

We construct a $(3, 2, 1, 1)$-Kawakita blowup $\bar Y_0 \to \bar X$. Define the toric variety $\bar T_0$ by
\[
\begin{array}{ccc|ccccccccc}
                 &  u & x & y & z & \alpha & t &\\
\Ta{\bar T_0\colon} &  0 & 1 & 1 & 1 & 3 & 1 & \Tb{.}\\
                 & -1 & 0 & 1 & 1 & 3 & 2 &
\end{array}
\]
In other words, $\bar T_0$ is given by the geometric quotient
\[
\bar T_0 = \frac{\mathbb C^6 \setminus \mathbb V((u, x) \cap (y, z, \alpha, t))}{(\mathbb C^*)^2}.
\]
Let $\bar{\Phi}$ be the ample model of $\mathbb V(x)$, and let $\bar Y_0 \subseteq \bar T_0$ be the strict transform of~$\bar X$. By \cref{thm:wei SDS cAn p}, $\bar Y_0 \to \bar X$ is a $(3, 2, 1, 1)$-Kawakita blowup. Alternatively, define $\bar Y_0$ by $\mathbb V(\bar g) \subseteq T_0$ where
\[
\bar g = \alpha (-u \alpha + 2 x^2 t + 2 x a_2) + x^3 t^2 A_Y + x^2 t B_Y + x C_Y + u D_Y
\]
and use \cref{thm:wei cAn wt correct implies Kawakita blup} on the patch $(\bar Y_0)_x \to \bar X_x$ to show that $\bar Y_0 \to \bar X$ is a $(3, 2, 1, 1)$-Kawakita blowup, similarly to \cref{thm:mod cA4}.

We show that $I_{\bar Y}$ does not 2-ray follow~$\bar T_0$. We describe the next (and the final) map in the 2-ray game for~$\bar T_0$. Acting by the matrix $\smat{1 & -1\\2 & -1}$, we can write $\bar T_0$ by
\[
\begin{array}{ccc|ccccccccc}
                   & u & x & y & z & \alpha &  t &\\
\Ta{\bar T_0\cong} & 1 & 1 & 0 & 0 & 0 & -1 & \Tb{.}\\
                   & 1 & 2 & 1 & 1 & 3 &  0 &
\end{array}
\]
The ample model of the divisor $\mathbb V(y)$ is the weighted blowup
\[
\begin{aligned}
\bar T_0 & \to \mathbb P(1, 1, 1, 2, 3)\\
[u, x, y, z, \alpha, t] & \mapsto [y, z, ut, xt, \alpha],
\end{aligned}
\]
where the centre is the surface $\mathbb P(1, 1, 3)$ given by $\mathbb V(u, x) \subseteq \mathbb P(1, 1, 1, 2, 3)$ with variables $y, z, u, x, \alpha$. Above every point in $\mathbb P(1, 1, 3)$, the fibre is~$\mathbb P^1$. Define
\[
\bar Z\colon \mathbb V(\bar h) \subseteq \mathbb P(1, 1, 1, 2, 3)
\]
where
\[
\bar h = \alpha (-u \alpha + 2 x^2 + 2 x a_2) + x^3 A_Z + x^2 B_Z + x C_Z + u D_Z,
\]
where
\[
A_Z = A_1(y, z, u), \qquad B_Z = B_3(y, z, u), \qquad C_Z = C_5(y, z, u), \qquad D_Z = D_6(y, z, u).
\]
We show that when restricting the weighted blowup to~$\bar Y_0 \to \bar Z$, the exceptional locus is $1$-dimensional. After restricting to $\bar Y_0$, the exceptional divisor $\mathbb V(t)$ becomes $\mathbb V(t, x(2\alpha a_2 + C_5(y, z, 0)) + u(-\alpha^2 + D_6(y, z, 0)))$. By \cref{def:mod generality conds}, there are exactly $10$ points $P_1, \ldots, P_{10} \in \mathbb P(1, 1, 3) \subseteq \bar Z$ such that $2\alpha a_2 + C_5(y, z, 0)$ and $-\alpha^2 + D_6(y, z, 0)$ have a common solution. Above each of those points, the fibre is $\mathbb P^1$. Above any other point, the fibre is just one point. Therefore, the morphism $\bar Y_0 \to \bar Z$ contracts 10 curves onto 10 points, and is an isomorphism elsewhere. This shows that $\bar Y_0$ does not 2-ray follow~$\bar T_0$, since $\bar Z$ is not $\mathbb Q$-factorial and a 2-ray link ends with either a fibration or a divisorial contraction.

The problem with the previous embedding was that $\bar g$ belonged to the irrelevant ideal $(u, x)$. We ``unproject'' the divisor $\mathbb V(u, x)$, to embed $\bar Y_0$ into a toric variety $T_0$ such that $Y_0$ 2-ray follows~$T_0$. The varieties $Y_0 \subseteq T_0$ are defined as in the proof of \cref{thm:mod cA4}. We see that $\bar Y_0$ is isomorphic to $Y_0$ through the map
\[
[u, x, y, z, \alpha, t] \mapsto \mleft[ u, x, y, z, \alpha, \frac{\alpha^2 - D_Y}{x}, t \mright].
\]
The map is a morphism, since we have the equality
\[
\frac{\alpha^2 - D_Y}{x} = \frac{2 \alpha a_2 + 2 \alpha x t + x^2 t^2 A_Y + x t B_Y + C_Y}{u}
\]
in the field of fractions of $\mathbb C[u, x, y, z, \alpha, t]$, and $x$ or $u$ is non-zero at every point of~$T_0$. For more details on this kind of ``unprojection'', see \cite[Section~2]{Rei00} or \cite[Section~2.3]{PR04}.

The coordinate change $\bar Y_0 \to Y_0$ induces a coordinate change $\bar X \to X$, where $X$ is defined as in the proof of \cref{thm:mod cA4}.
\end{remark}

\subsection{\texorpdfstring
{$cA_5$ model}
{cA5 model}} \label{subsec:mod cA5}

\begin{proposition} \label{thm:mod cA5}
A sextic double solid $X$ which is a Mori fibre space with a $cA_5$ singularity satisfying \cref{def:mod generality conds} has a Sarkisov link to a sextic double solid $Z$ with a $cA_5$ singularity, starting with a $(3, 3, 1, 1)$-blowup of the $cA_5$ point in~$X$, then four Atiyah flops, and finally a $(3, 3, 1, 1)$-blowdown to a $cA_5$ point. If in addition $c_4$ is general after fixing $a_i$, $b_i$ and~$d_6$ in \cref{nota:con}, then $X$ and $Z$ are not isomorphic. Under further generality conditions (\cref{thm:mod Y only up to two quot sing}), both $X$ and $Z$ are smooth outside the $cA_5$ point.
\end{proposition}

\begin{proof}
We exhibit the diagram below.
\begin{equation*}
\begin{tikzcd}[column sep = small]
& \arrow[ld, "{(3, 3, 1, 1)}"'] Y_0 \arrow[rd, ""'] \arrow[rr, dashed, "{4 \times (1, 1, -1, -1)}"] & & Y_1 \arrow[ld, ""'] \arrow[rd, "{(3, 3, 1, 1)}"] \\
cA_5 \in X \subseteq \mathbb P(1^4, 3) \hspace{-3em} & & W_0 & & cA_5 \in Z \subseteq \mathbb P(1^4, 3)
\end{tikzcd}
\end{equation*}

We construct $X$ and a $(3, 3, 1, 1)$-Kawakita blowup $Y_0 \to X$. Using \cref{mai:con}, and performing the coordinate change in \cref{eqn:wei cAn p v - r1 r2 equal} of \cref{thm:wei SDS cAn p} (with $p_2 = a_2$), we can write a sextic double solid $X$ with a $cA_5$ singularity by
\[
X\colon \mathbb V(f, -\beta + x t + a_2) \subseteq \mathbb P(1, 1, 1, 1, 2, 3),
\]
with variables $x, y, z, t, \beta, w$ where
\[
f = -w^2 + x \beta (2 b_3 - 4 \beta a_1 + 8 x t a_1 + x \beta) + 4 x^3 t^3 a_0 + x^2 t^2 B_2 + x t C_4 + D_6,
\]
where $B_i$, $C_i$, $D_i \in \mathbb C[y, z, t]$ are homogeneous of degrees~$i$. Define $T_0$ by
\[
\begin{array}{ccc|cccccc}
            &  u & x & y & z & w & \beta & t &\\
\Ta{T_0\colon} &  0 & 1 & 1 & 1 & 3 & 2 & 1 & \Tb{.}\\
            & -1 & 0 & 1 & 1 & 3 & 3 & 2 &
\end{array}
\]
Let $\Phi\colon T_0 \to \mathbb P(1, 1, 1, 1, 2, 3)$ be the ample model of $\mathbb V(x)$, $Y_0 \subseteq T_0$ the strict transform of~$X$, and $Y_0 \to X$ the restriction of~$\Phi$. Then, $Y_0$ is given by
\[
Y_0\colon \mathbb V(I_Y) \subseteq T_0 \ \text{ where } \ I_Y = (\Phi^*f / u^6,\, -u \beta + x t + a_2),
\]
and $Y_0 \to X$ is a $(3, 3, 1, 1)$-Kawakita blowup.

We show that the first map in the 2-ray game for $Y_0$ is a flop, locally analytically $4$ Atiyah flops, under \cref{def:mod generality conds}. Acting by the matrix $\smat{1 & 0\\ -1 & 1}$, we find
\[
\begin{array}{ccc|cccccc}
               &  u &  x & y & z & w & \beta & t &\\
\Ta{T_0 \cong} &  0 &  1 & 1 & 1 & 3 & 2 & 1 & \Tb{.}\\
               & -1 & -1 & 0 & 0 & 0 & 1 & 1 &
\end{array}
\]
The base of the flop in $\mathbb P(1, 1, 3) \subseteq W_0$ is given by $\mathbb V(a_2, -w^2 + D_6(y, z, 0)) \subseteq \mathbb P(1, 1, 3)$. If $a_2$ and $D_6(y, z, 0)$ are general, that is, \cref{def:mod generality conds} is satisfied, then this is exactly $4$ points. In this case, any such point $P$ is a smooth point in $\mathbb P(1, 1, 3)$. Consider the case where the $y$-coordinate of $P$ is non-zero, the case where $z$ is non-zero is similar. Locally analytically equivariantly, we can express $z$ and $w$ in terms of $u, x, \beta, t$ in $Y_0$, $W_0$ and~$Y_1$. So, the diagram $Y_0 \to W_0 \gets Y_1$ is locally analytically four Atiyah flops.

The last map in the 2-ray game of $Y_0$ is a weighted blowdown $Y_1 \to Z$. After acting by $\smat{3 & -2\\2 & -1}$ on the initial matrix of~$T_0$, we find that $T_1$ is given by
\[
\begin{array}{cccccc|ccc}
            & u & x & y & z & w & \beta &  t &\\
\Ta{T_1\colon} & 2 & 3 & 1 & 1 & 3 & 0 & -1 & \Tb{.}\\
            & 1 & 2 & 1 & 1 & 3 & 1 &  0 &
\end{array}
\]
We see that $Z \subseteq \mathbb P(1, 1, 1, 1, 2, 3)$ with variables $\beta, u, y, z, x, w$ is given by the ideal
\[
I_Z = (h, -u \beta + x + a_2),
\]
where $h$ is given by sending $t$ to $1$ in $\Phi^*f / u^6$, namely
\[
h = -w^2 + x \beta (2 b_3 - 4 u \beta a_1 + 8 x a_1 + x \beta) + 4 x^3 a_0 + x^2 B_Z + x C_Z + D_Z
\]
and
\[
B_Z = B_2(y, z, u), \qquad C_Z = C_4(y, z, u), \qquad D_Z = D_6(y, z, u).
\]
Substituting $x = u \beta - a_2$
into~$h$, we find that $Z$ is a sextic double solid. Applying the explicit splitting lemma (\cref{thm:con splitting explicit}), we find that the complex analytic space germ $(Z, P_\beta)$ is isomorphic to $(\mathbb V(h_{\operatorname{ana}}), \bm 0) \subseteq (\mathbb C^4, \bm 0)$ with variables $w, u, y, z$, where
\[
h_{\operatorname{ana}} = -w^2 + u^2 + d_6 - (b_3 - 2 a_1 a_2)^2 + \text{(h.o.t in $y, z$)},
\]
where $\text{(h.o.t in $y, z$)}$ stands for higher order terms in the variables $y, z$. So, $P_\beta \in Z$ is a $cA_5$ singularity. On the patch where $\beta$ is non-zero, we can substitute $u = x t + a_2$,
so the morphism $(Y_1)_\beta \to Z_\beta$ is a weighted blowup of a hypersurface given by a weight $6$ polynomial. By \cref{thm:wei cAn wt correct implies Kawakita blup}, $Y_1 \to Z$ is a $(3, 3, 1, 1)$-Kawakita blowup.

We show that $X$ and $Z$ are not isomorphic when $a_2 \neq 0$ and $c_4$ is general, using a dimension counting argument similar to \cite[Theorem~2.55]{GLS07}. Using the explicit splitting lemma, we find that the complex analytic space germ $(X, P_x)$ is isomorphic to $(\mathbb V(f_{\operatorname{ana}}), \bm 0) \subseteq (\mathbb C^4, \bm 0)$ with variables $w, t, y, z$ where
\[
f_{\operatorname{ana}} = -w^2 + t^2 + d_6 - 2 a_2 c_4 + 2 a_2^2 b_2 - 4 a_0 a_2^3 - (b_3 - 4 a_1 a_2)^2 + \text{(h.o.t in $y, z$)}.
\]
If $X$ and $Z$ are isomorphic, then this implies that the complex analytic space germs $(X, P_x)$ and $(Z, P_\beta)$ are isomorphic, implying by \cref{thm:ana stable equivalence,thm:ana lowest degree parts} that the degree $6$ parts of $f_{\operatorname{ana}}(0, 0, y, z)$ and $h_{\operatorname{ana}}(0, 0, y, z)$ are the same up to an invertible linear coordinate change on $y, z$. Fixing $a_0$, $a_1$, $a_2$, $b_2$, $b_3$ and~$d_6$, we see that $h_{\operatorname{ana}}(0, 0, y, z)$ is fixed, but $f_{\operatorname{ana}}(0, 0, y, z)$ has $5$ degrees of freedom. Since there are only $4$ degrees of freedom in picking an element of $\operatorname{GL}(2, \mathbb C)$, the polynomials $f_{\operatorname{ana}}(0, 0, y, z)$ and $h_{\operatorname{ana}}(0, 0, y, z)$ are not related by an invertible linear coordinate change when $c_4$ is general. This shows that if $X$ is general, then the varieties $X$ and $Z$ are not isomorphic.
\end{proof}

\subsection{\texorpdfstring
{$cA_6$ model}
{cA6 model}} \label{subsec:mod cA6}

\begin{proposition} \label{thm:mod cA6}
A sextic double solid that is a Mori fibre space with a $cA_6$ singularity satisfying \cref{def:mod generality conds} has a Sarkisov link to a hypersurface $Z_5 \subseteq \mathbb P(1, 1, 1, 1, 2)$ with a $cA_3$ singularity, starting with a $(4, 3, 1, 1)$-blowup of the $cA_6$ point, then two $(1, 1, -1, -1)$-flops, then a $(4, 1, 1, -2, -1; 2)$-flip, and finally a $(2, 2, 1, 1)$-blowdown to a $cA_3$ point. Under further generality conditions, the singular locus of $Z$ consists of 3 points, namely the $cA_3$ point, the $1/2(1, 1, 1)$ quotient singularity and an ordinary double point.
\end{proposition}

\begin{proof}
We exhibit the diagram below.
\begin{equation*}
\begin{tikzcd}[column sep = small]
& \arrow[ld, "{(4, 3, 1, 1)}"'] Y_0 \arrow[rd, ""] \arrow[rr, dashed, "{2 \times (1, 1, -1, -1)}"] & & Y_1 \arrow[ld, ""] \arrow[rd, ""] \arrow[rr, dashed, "{(4, 1, 1, -2, -1; 2)}"] & & Y_2 \arrow[ld, ""] \arrow[rd, "{(3, 1, 1, 1)}"] \\
cA_6 \in X \subseteq \mathbb P(1^4, 3) \hspace{-3em} & & W_0 & & W_1 & & cA_3 \in Z_5 \subseteq \mathbb P(1^4, 2)
\end{tikzcd}
\end{equation*}

We construct $X$ and a $(4, 3, 1, 1)$-Kawakita blowup $Y_0 \to X$. Using \cref{mai:con} and \cref{thm:wei SDS cAn p} with $p_2 = a_2$ and $p_3 = b_3 - 4 a_1 a_2$,
we can write a sextic double solid $X$ with a $cA_6$ singularity by
\[
X\colon \mathbb V(f, -\beta + x t + a_2) \subseteq \mathbb P(1, 1, 1, 1, 2, 3),
\]
with variables $x, y, z, t, \beta, w$ where
\[
\begin{aligned}
f & = \alpha (-\alpha + 2 (b_3 - 4 \beta a_1 + 4 x t a_1 + x \beta))\\
  & + 2 \beta (c_4 - \beta b_2 + 2 x t b_2 + 2 x \beta a_1 + 2 \beta^2 a_0 - 6 x t \beta a_0 + 6 x^2 t^2 a_0)\\
  & + x^2 t^3 B_1 + x t^2 C_3 + t D_5
\end{aligned}
\]
where $B_i$, $C_i$, $D_i \in \mathbb C[y, z, t]$ are homogeneous of degree~$i$. Define $T_0$ by
\[
\begin{array}{ccc|cccccc}
            &  u & x & y & z & \alpha & \beta & t &\\
\Ta{T_0\colon} &  0 & 1 & 1 & 1 & 3 & 2 & 1 & \Tb{.}\\
            & -1 & 0 & 1 & 1 & 4 & 3 & 2 &
\end{array}
\]
Let $\Phi\colon T_0 \to \mathbb P(1, 1, 1, 1, 2, 3)$ be the ample model of $\mathbb V(x)$, $Y_0 \subseteq T_0$ the strict transform of~$X$, and $Y_0 \to X$ the restriction of~$\Phi$. Then, $Y_0$ is given by
\[
Y_0\colon \mathbb V(I_Y) \subseteq T_0 \ \text{ where } \ I_Y = (\Phi^*f / u^7,\, -u \beta + x t + a_2),
\]
and $Y_0 \to X$ is a $(4, 3, 1, 1)$-Kawakita blowup.

We show that the first diagram $Y_0 \to W_0 \gets Y_1$ in the 2-ray game for $Y_0$ is locally analytically two Atiyah flops under \cref{def:mod generality conds}, namely that $\mathbb V(a_2) \subseteq \mathbb P^1$ with variables $y, z$ consists of exactly two points, and for both of the points~$P$, one of $b_3(P)$, $c_4(P)$ or $d_5(P)$ is non-zero, where $D_5 = t^5 d_0 + 2 t^4 d_1 + t^3 d_2 + 2 t^2 d_3 + t d_4 + 2 d_5$.
Acting by the matrix $\smat{4 & -3\\-1 & 1}$, we find
\[
\begin{array}{ccc|cccccc}
            &  u &  x & y & z & \alpha &  \beta &  t &\\
\Ta{T_0\colon} &  3 &  4 & 1 & 1 & 0 & -1 & -2 & \Tb{.}\\
            & -1 & -1 & 0 & 0 & 1 &  1 &  1 &
\end{array}
\]
Under the above condition, after a suitable linear change of coordinates on $y, z$, we find that $a_2 = yz$. Let $P = \mathbb V(z) \in \mathbb P^1 \subseteq W_0$, the case where $P = \mathbb V(y)$ is similar.
On the patch where $y$ is non-zero, we can substitute $z = u \beta - x t$.
The contracted locus is $\mathbb P^1 \cong \mathbb V(\alpha, \beta, t) \subseteq (Y_0)_y$, and the extracted locus is $\mathbb V(u, x) = \mathbb V(u, x, \alpha b_3(1, 0) + \beta c_4(1, 0) + t d_5(1, 0)) \subseteq (Y_1)_y$.
By \cref{def:mod generality conds}, we can express one of $\alpha, \beta, t$ equivariantly locally analytically in the other variables. So, the flop diagram $Y_0 \to W_0 \gets Y_1$ is locally analytically a $(1, 1, -1, -1)$-flop above both of the points.

We show that the next diagram in the 2-ray game of $Y_0$ is a $(4, 1, 1, -2, -1; 2)$-flip (this is case~(1) in \cite[Theorem~8]{Bro99}). The toric variety $T_1$ is given by
\[
\begin{array}{ccccc|cccc}
            &  u &  x & y & z & \alpha &  \beta &  t &\\
\Ta{T_1\colon} &  3 &  4 & 1 & 1 & 0 & -1 & -2 & \Tb{.}\\
            & -1 & -1 & 0 & 0 & 1 &  1 &  1 &
\end{array}
\]
The base of the flip is $P_\alpha = [0, 0, 0, 0, 1, 0, 0]$.
On the patch where $\alpha$ is non-zero, we can express $u$ locally analytically and equivariantly in terms of $x, y, z, \beta, t$. After substitution, the ideal is principal, with generator $f' = -\beta \cdot (2 x + \ldots) + x t + a_2$.
Under \cref{def:mod generality conds}, $a_2$ has a non-zero coefficient in~$f'$, so the flip diagram corresponds to case~(1) in \cite[Theorem~8]{Bro99}. The flips contracts a curve containing a $1/4(1, 1, 3)$ singularity and extracts a curve containing a $1/2(1, 1, 1)$ singularity and an ordinary double point. The ordinary double point on $Y_2$ is at $[u_0, 0, 0, 0, 2, 1, 1]$ for some $u_0 \in \mathbb C$.

We show that the last map in the 2-ray game of $Y_0$ is a weighted blowup $Y_2 \to Z$, where $Z$ is isomorphic to a hypersurface $Z_5 \subseteq \mathbb P(1, 1, 1, 1, 2)$ with variables $u, y, z, \beta, \alpha$. Acting by the matrix $\smat{3 & -2\\2 & -1}$ on the initial action-matrix of $T_0$, we find that $T_2$ is given by
\[
\begin{array}{cccccc|ccc}
            & u & x & y & z & \alpha & \beta &  t &\\
\Ta{T_2\colon} & 2 & 3 & 1 & 1 & 1 & 0 & -1 & \Tb{.}\\
            & 1 & 2 & 1 & 1 & 2 & 1 &  0 &
\end{array}
\]
Define the bidegree $(5, 2)$ complete intersection $Z\colon \mathbb V(h, a_2 - u \beta + x) \subseteq \mathbb P(1, 1, 1, 1, 2, 2)$
with variables $u, y, z, \beta, x, \alpha$, where
\[
\begin{aligned}
h & = \alpha (-u \alpha + 2 (b_3 - 4 u \beta a_1 + 4 x a_1 + x \beta))\\
  & + 2 \beta (c_4 - u \beta b_2 + 2 x b_2 + 2 x \beta a_1 + 2 u^2 \beta^2 a_0 - 6 u x \beta a_0 + 6 x^2 a_0)\\
  & + x^2 B_Z + x C_Z + D_Z,
\end{aligned}
\]
where
\[
B_Z = B_1(y, z, u), \qquad C_Z = C_3(y, z, u), \qquad D_Z = D_5(y, z, u).
\]
The morphism $Y_2 \to Z$ given by the ample model of $\mathbb V(\beta)$ is a weighted blowdown with centre~$P_\beta$ and exceptional locus $\mathbb V(t)$. Substituting
\begin{equation} \label{eqn:mod cA6 subs}
x = u \beta - a_2
\end{equation}
into $h$, we find that $Z$ is isomorphic to a hypersurface $Z_5 \subseteq \mathbb P(1, 1, 1, 1, 2)$ with variables $u, y, z, \beta, \alpha$. The substitution~\labelcref{eqn:mod cA6 subs} does not lift onto~$Y_2$. Instead, on the patch $Z_\beta$, we can substitute $u = (a_2 + x)/\beta$. This substitution lifts to $(Y_2)_\beta$. By \cref{def:mod generality conds}, $P_\beta \in Z$ is a $cA_3$ singularity and the hypersurface $Z_\beta$ is given by a weight 4 polynomial. By \cref{thm:wei cAn wt correct implies Kawakita blup}, $(Y_2)_\beta \to Z_\beta$ is a $(3, 1, 1, 1)$-Kawakita blowup.

Note that $Z$ has an ordinary double point at $[u_0, 0, 0, 1, 2]$ for some $u_0 \in \mathbb C$.
\end{proof}

\subsection{\texorpdfstring
{$cA_7$ family 7.1 model}
{cA7 family 7.1 model}} \label{subsec:mod cA7-1}

\begin{proposition} \label{thm:mod cA7-1}
A Mori fibre space sextic double solid with a $cA_7$ singularity in family~7.1 satisfying \cref{def:mod generality conds} has a Sarkisov link to $Z_{3, 4} \subseteq \mathbb P(1, 1, 1, 1, 2, 2)$ with an ordinary double point, starting with a $(4, 4, 1, 1)$-blowup of the $cA_7$ point, then two $(4, 1, 1, -2, -1; 2)$-flips, and finally a blowdown (with standard weights $(1, 1, 1, 1)$) to an ordinary double point. Under further generality conditions, $Z$ has exactly five singular points, namely two $1/2(1, 1, 1)$ singularities and three ordinary double points.
\end{proposition}

\begin{proof}
We exhibit the diagram below.
\begin{equation*}
\begin{tikzcd}[column sep = scriptsize]
& \arrow[ld, "{(4, 4, 1, 1)}"'] Y_0 \arrow[r, "\sim"] & Y_1 \arrow[rd, ""] \arrow[rr, dashed, "{2 \times (4, 1, 1, -2, -1; 2)}"] & & Y_2 \arrow[ld, ""] \arrow[rd, "{(1, 1, 1, 1)}"] \\
cA_7 \in X \subseteq \mathbb P(1^4, 3) \hspace{-3em} & & & W_0 & & \operatorname{ODP} \in Z_{3, 4} \subseteq \mathbb P(1^4, 2^2)
\end{tikzcd}
\end{equation*}

We construct $X$ and a $(4, 4, 1, 1)$-Kawakita blowup $Y_0 \to X$. We can write a sextic double solid $X$ with an isolated $cA_7$ singularity in family~7.1 by
\[
X\colon \mathbb V(f, \beta - x t - r_2, \gamma - x \beta - s_3) \subseteq \mathbb P(1, 1, 1, 1, 2, 3, 3)
\]
with variables $x, y, z, t, \beta, \gamma, w$, where
\[
\begin{aligned}
f & = -w^2 + \gamma^2 - 2 t \gamma e_2 + 2 \beta^2 e_2 + 2 t \beta c_3 + 4 t \gamma b_2 - 2 \beta^2 b_2 - 2 t \beta^2 b_1 + 4 x t^2 \beta b_1\\
  & + 2 x^2 t^4 b_0 - 16 t \gamma a_1^2 + 16 \beta^2 a_1^2 + 4 \beta \gamma a_1 - 8 \beta^3 a_0 + 12 x t \beta^2 a_0 + x t^3 C_2 + t^2 D_4,
\end{aligned}
\]
where $C_i$, $D_i \in \mathbb C[y, z, t]$ are homogeneous of degree~$i$. Define $T_0$ by
\[
\begin{array}{ccc|ccccccc}
            &  u & x & y & z & w & \gamma & \beta & t &\\
\Ta{T_0\colon} &  0 & 1 & 1 & 1 & 3 & 3 & 2 & 1 & \Tb{.}\\
            & -1 & 0 & 1 & 1 & 4 & 4 & 3 & 2 &
\end{array}
\]
Define $Y_0$ by
\[
Y_0\colon \mathbb V(I_Y) \subseteq T_0 \ \text{ where } \ I_Y = (\Phi^*f / u^8,\, u \beta - r_2 - x t,\, u \gamma - s_3 - x \beta).
\]
The ample model of $\mathbb V(x) \subseteq Y_0$ is a $(4, 4, 1, 1)$-Kawakita blowup $Y_0 \to X$.

We show that the diagram $Y_0 \to W_0 \gets Y_1$ induces an isomorphism $Y_0 \to Y_1$. Acting by the matrix $\smat{4 & -3\\-1 & 1}$, we find
\[
\begin{array}{ccc|ccccccc}
               &  u &  x & y & z & w & \gamma &  \beta &  t &\\
\Ta{T_0 \cong} &  3 &  4 & 1 & 1 & 0 & 0 & -1 & -2 & \Tb{.}\\
               & -1 & -1 & 0 & 0 & 1 & 1 &  1 &  1 &
\end{array}
\]
Define $T_1$, respectively $T_2$, with the same action as $T_0$ but with irrelevant ideal $(u, x, y, z) \cap (w, \gamma, \beta, t)$, respectively $(u, x, y, z, w, \gamma) \cap (\beta, t)$. Define $Y_1 \subseteq T_1$ and $Y_2 \subseteq T_2$ by the same ideal~$I_Y$. The base of the flop $T_0 \to \mathcal W_0 \gets T_1$ restricts to $\mathbb V(r_2, s_3) \subseteq \mathbb P^1 \subseteq W_0$, which is empty. Therefore, $Y_0 \to W_0$ and $W_0 \gets Y_1$ are isomorphisms.

We show that the next diagram $Y_1 \to W_1 \gets Y_2$ in the 2-ray game of $Y_0$ is locally analytically two $(4, 1, 1, -2, -1; 2)$-flips. The only monomials in $\Phi^*f / u^8$ that are not in $(u, x, y, z, \beta, t)$ are $-w^2$ and~$\gamma^2$. Therefore, the base of the flip is two points, $[1, 1]$ and $[-1, 1] \in \mathbb P^1$ with variables $w$ and $\gamma$ inside $W_1$. We make a change of coordinates $w' = w - \gamma$, respectively $w' = w + \gamma$, for the flip above $[1, 1]$, respectively~$[-1, 1]$. On the patch where $\gamma$ is non-zero, we can substitute $u = s_3 + x \beta$
in $\Phi^*f / u^8$, and express $w'$ locally analytically and equivariantly above $[1, 1]$, respectively $[-1, 1]$, in terms of $x, y, z, \beta, t$. After projecting away the variables $u$ and~$w'$, we are left with the principal ideal $(\beta s_3 - r_2 + x \beta^2 - x t)$.
Since it contains both $r_2$ and $xt$, by case~(1) in \cite[Theorem~8]{Bro99}, it is a terminal $(4, 1, 1, -2, -1; 2)$-flip above both $[1, 1]$ and $[-1, 1]$. The flip contracts two curves, both containing a $1/4(1, 1, 3)$ singularity, and extracts two curves, both containing a $1/2(1, 1, 1)$ singularity and a $cA_1$ singularity. The $cA_1$ points are both ordinary double points if $r_2$ is not a square of a linear form, and are both 3-fold $A_2$ singularities (given by $x_1^2 + x_2^2 + x_3^2 + x_4^3$) otherwise. On~$Y_2$, the $cA_1$ singularities are at $[0, 0, 0, 0, 1, 1, 1, 1]$ and $[0, 0, 0, 0, -1, 1, 1, 1]$.

We show that the last map in the link for $X$ is a divisorial contraction $Y_2 \to Z'$. Acting by the matrix $\smat{3 & -2\\2 & -1}$ on the initial action-matrix of $T_0$, we see that
\[
\begin{array}{ccccccc|ccc}
               & u & x & y & z & w & \gamma & \beta &  t &\\
\Ta{T_2 \cong} & 2 & 3 & 1 & 1 & 1 & 1 & 0 & -1 & \Tb{.}\\
               & 1 & 2 & 1 & 1 & 2 & 2 & 1 &  0 &
\end{array}
\]
Define $Z' \subseteq \mathbb P(1, 1, 1, 1, 2, 2, 2)$ with variables $u, y, z, \beta, w, \gamma, x$ by the ideal~$I_{Z'}$, where $I_{Z'}$ is the image of the ideal $I_Y$ under the homomorphism $t \mapsto 1$. Let $Y_2 \to Z'$ be the ample model of $\mathbb V(\beta)$. On the affine patch $Z'_\beta$, we can express $u$ and $x$ locally analytically and equivariantly in terms of $y, z, w, \gamma, \beta, t$. This coordinate change lifts to~$Y_2$. By \cref{def:mod generality conds}, we can compute that $P_\beta \in Z'$ is an ordinary double point, and $Y_2 \to Z'$ is locally analytically the (usual) blowup with centre~$P_\beta$.

The variety $Z'$ is isomorphic to a complete intersection $Z_{3, 4} \subseteq \mathbb P(1^4, 2^2)$, by projecting away from~$x$. The variety $Z$ is given by
\[
Z_{3, 4}\colon \mathbb V(-s_3 + \beta r_2 + u \gamma - u \beta^2,\, h) \subseteq \mathbb P(1, 1, 1, 1, 2, 2)
\]
with variables $u, y, z, \beta, w, \gamma$, where
\[
\begin{aligned}
h & = -w^2 + \gamma^2 + 2 b_0 r_2^2 - 4 \beta b_1 r_2 - 4 u \beta b_0 r_2 - 12 \beta^2 a_0 r_2 - 2 \gamma e_2 + 2 \beta^2 e_2 + 2 \beta c_3 + 4 \gamma b_2\\
  & - 2 \beta^2 b_2 + 2 u \beta^2 b_1 + 2 u^2 \beta^2 b_0 - 16 \gamma a_1^2 + 16 \beta^2 a_1^2 + 4 \beta \gamma a_1 + 4 u \beta^3 a_0 + (u \beta - r_2) C_Z + D_Z,
\end{aligned}
\]
where $C_Z = C_2(y, z, u)$ and $D_Z = D_4(y, z, u)$. The variety $Z$ has two $cA_1$ singularities at $[0, 0, 0, 1, 1, 1]$ and $[0, 0, 0, 1, -1, 1]$.
\end{proof}

\begin{remark}
We explain how we found the variety~$X$. Using $p_2 = r_2$ and $p_3 = s_3$, we can write a sextic double solid with an isolated $cA_7$ in family~7.1 by $\bar X\colon \mathbb V(\bar f, x^2 t + x r_2 + s_3 - \bar{\gamma})$
inside $\mathbb P(1, 1, 1, 1, 3, 3)$ with variables $x, y, z, t, w, \bar{\gamma}$, where $\bar f$ is given as in \cref{mai:con}. The $(1, 1, 4, 4, 2)$-blowup $\bar Y_0 \to \bar X$ for variables $y, z, w, \bar{\gamma}, t$ is a $(4, 4, 1, 1)$-Kawakita blowup, but the 2-ray game of $\bar Y_0$ does not follow the ambient toric variety~$\bar T_0$. Namely, the toric anti-flip $\bar T_0 \to \bar{\mathcal W}_0 \gets \bar T_1$ restricts to $\bar Y_0 \to \bar{W}_0 \gets \bar Y_1$, where $\bar Y_0 \to \bar{W}_0$ is an isomorphism and $\bar{W}_0 \gets \bar Y_1$ extracts $\mathbb P^2$, a divisor on~$\bar Y_1$. The reason why $\bar Y_0$ was not the correct variety is that one of the generators of the ideal of $\bar Y_0$ is $\bar g_1 = x^2 t + x r_2 + u s_3 - u \bar{\gamma}$,
which is inside the irrelevant ideal $(u, x)$. We find the correct variety $Y_0$ by ``unprojecting'' $\bar g_1 = 0$ with respect to~$u, x$. By ``unprojection'', we mean the coordinate change $\bar Y_0 \to Y_0$, an isomorphism. See \cite[Section~2]{Rei00} or \cite[Section~2.3]{PR04} for more details on this type of unprojection. This coordinate change induces the coordinate change $\bar X \to X$, where $X$ is given as in the proof of \cref{thm:mod cA7-1}.
\end{remark}

\subsection{\texorpdfstring
{$cA_7$ family 7.2 model}
{cA7 family 7.2 model}} \label{subsec:mod cA7-2}

\begin{proposition} \label{thm:mod cA7-2}
A Mori fibre space sextic double solid with a $cA_7$ singularity in family~7.2 satisfying \cref{def:mod generality conds} has a Sarkisov link to a complete intersection $Z_{2, 4} \subseteq \mathbb P(1, 1, 1, 1, 1, 2)$ with a $cA_2$ singularity, starting with a $(4, 4, 1, 1)$-blowup of the $cA_7$ point, followed by one Atiyah flop, then two $(4, 1, -1, -3)$-flips, and finally a $(3, 3, 2, 1)$-blowdown to a $cA_2$ point. Under further generality conditions, the variety $Z$ is smooth outside the $cA_2$ point.
\end{proposition}

\begin{proof}
We exhibit the diagram below.
\begin{equation*}
\begin{tikzcd}[column sep = small]
& \arrow[ld, "{(4, 4, 1, 1)}"'] Y_0 \arrow[rd, ""'] \arrow[rr, dashed, "{(1, 1, -1, -1)}"] & & Y_1 \arrow[ld, ""'] \arrow[rd, ""] \arrow[rr, dashed, "{2 \times (-4, -1, 1, 3)}"] & & Y_2 \arrow[ld, ""'] \arrow[r, "\sim"] & Y_3 \arrow[rd, "{(3, 3, 2, 1)}"] \\
cA_7 \in X \subseteq \mathbb P(1^4, 3) \hspace{-3em} & & W_0 & & W_1 & & & cA_2 \in Z_{2, 4} \subseteq \mathbb P(1^5, 2)
\end{tikzcd}
\end{equation*}

We describe the sextic double solid~$X$. Define $X \subseteq \mathbb P(1, 1, 1, 1, 2, 3, 3, 3)$ with variables $x$, $y$, $z$, $t$, $\beta$, $w$, $\gamma$, $\xi$ by the ideal
\begin{equation} \label{eqn:cA7-2 X}
I_X = (f - 2 e_3 \xi,\, \beta - q_1 r_1 - x t,\, \gamma - q_1 s_2 - x \beta,\, -\xi + t s_2 - \beta r_1),
\end{equation}
where
\[
\begin{aligned}
f & = -w^2 + \gamma^2 + 2 t \beta c_3 + 4 t \gamma b_2 - 2 \beta^2 b_2 - 2 t \beta^2 b_1 + 4 x t^2 \beta b_1 + 2 x^2 t^4 b_0\\
  & - 16 t \gamma a_1^2 + 16 \beta^2 a_1^2 + 4 \beta \gamma a_1 - 8 \beta^3 a_0 + 12 x t \beta^2 a_0 + x t^3 C_2 + t^2 D_4
\end{aligned}
\]
where $C_i$, $D_i \in \mathbb C[y, z, t]$ are homogeneous of degree~$i$.

We describe the weighted blowup $Y_0 \to X$, restriction of $\Phi\colon T_0 \to \mathbb P(1, 1, 1, 1, 2, 3, 3, 3)$. Define $T_0$ by
\[
\begin{array}{ccc|cccccccc}
            &  u & x & y & z & w & \gamma & \beta & \xi & t &\\
\Ta{T_0\colon} &  0 & 1 & 1 & 1 & 3 & 3 & 2 & 3 & 1 & \Tb{.}\\
            & -1 & 0 & 1 & 1 & 4 & 4 & 3 & 5 & 2 &
\end{array}
\]
Define $Y_0 \subseteq T_0$ by the ideal $I_Y$ with the 6 generators
\[
\begin{aligned}
g - 2 e_3 \xi, && u \beta - q_1 r_1 - x t, && u \gamma - q_1 s_2 - x \beta,\\
-u \xi + t s_2 - \beta r_1, && -x \xi + \beta s_2 - \gamma r_1, && -q_1 \xi + t \gamma - \beta^2,
\end{aligned}
\]
where $g = \Phi^*f / u^8$. On the affine patch~$X_x$, we can express $\beta, t$ and $\xi$ in terms of $w, \gamma, y, z$, to get a hypersurface in $\mathbb C^4$ given by~$f_{\operatorname{hyp}} \in \mathbb C[w, \gamma, y, z]$. Note that these coordinate changes lift to $(Y_0)_x$. Since $f_{\operatorname{hyp}}$ has weight~$8$, by \cref{thm:wei cAn wt correct implies Kawakita blup}, $Y_0 \to X$ is a $(4, 4, 1, 1)$-Kawakita blowup.

We show that the first diagram $Y_0 \to W_0 \gets Y_1$ in the 2-ray game of $Y_0$ is an Atiyah flop, provided that $r_1$ and $q_1$ are coprime in~$\mathbb C[y, z]$. Acting by the matrix $\smat{4 & -3\\-1 & 1}$ on the action-matrix of~$T_0$, define $T_1$ by
\[
\begin{array}{ccccc|cccccc}
            &  u &  x & y & z & w & \gamma &  \beta &  \xi &  t &\\
\Ta{T_1\colon} &  3 &  4 & 1 & 1 & 0 & 0 & -1 & -3 & -2 & \Tb{.}\\
            & -1 & -1 & 0 & 0 & 1 & 1 &  1 &  2 &  1 &
\end{array}
\]
Define $Y_1 \subseteq T_1$ by the ideal $I_Y$. The base of the flop is $\mathbb V(q_1) \subseteq \mathbb P^1$ with variables $y, z$, which is one point. Perform a suitable invertible linear coordinate change on $y, z$ such that $q_1 = z$ and $r_1 = y$. Since $u \beta - q_1 r_1 - x t$
is in $I_Y$, we can substitute $z = u \beta - x t$
on the patch where $y$ is non-zero. The coefficients of $\beta$ in $-u \xi + t s_2 - \beta y \in I_Y$
and $\gamma$ in $-x \xi + \beta s_2 - \gamma y \in I_Y$
are non-zero on the patch where $y$ is non-zero. Therefore, we can locally analytically equivariantly express $\beta$ and $\gamma$ in terms of $u, x, w, t$. After substituting $z, \beta, \gamma$, we find that the diagram $Y_0 \to W_0 \gets Y_1$ is locally analytically the Atiyah flop.

The next diagram in the 2-ray game of~$Y_0$ is the flip $Y_1 \to W_1 \gets Y_2$. The base of the flip is $\mathbb V(\gamma^2 - w^2) \subseteq \mathbb P^1$ with variables~$w, \gamma$, which is two points $[1, 1]$ and $[-1, 1]$. We consider the point $P = [1, 1]$, the flip for the other point is similar. Perform a coordinate change $w' = w - \gamma$. On the patch where $\gamma$ is non-zero, we find $u = q_1 s_2 + x \beta$
and $t = q_1 \xi + \beta^2$.
Writing $q_1 = z$ and $r_1 = y$ as before, we find $y = -x \xi + \beta s_2$.
We are left with the principal ideal in $\mathbb C[x, z, w', \beta, \xi]$ generated by $-w'(2 + w') + {}$terms not involving~$w'$. So, we can locally analytically equivariantly express $w'$ in terms of $x, z, \beta, \xi$. So, the diagram $Y_1 \to W_1 \gets Y_2$ is locally analytically two $(-4, -1, 1, 3)$-flips.

The next diagram in the toric 2-ray game $T_2 \to \mathcal W_2 \gets T_3$ restricts to isomorphisms $Y_2 \to W_2 \gets Y_3$. The reason is that the base of the toric flip $P_\beta$ restricts to an empty set in~$W_2$, since $I_Y$ contains the polynomial $t \gamma - \beta^2 - q \xi$.

We show that the last diagram in the 2-ray game of $Y_0$ is a divisorial contraction $Y_3 \to Z$. Multiplying the action-matrix of $T_1$ by $\smat{2 & 3\\1 & 2}$, we see that $T_3$ is given by
\[
\begin{array}{cccccccc|ccc}
            & u & x & y & z & w & \gamma & \beta & \xi &  t &\\
\Ta{T_3\colon} & 3 & 5 & 2 & 2 & 3 & 3 & 1 & 0 & -1 & \Tb{.}\\
            & 1 & 2 & 1 & 1 & 2 & 2 & 1 & 1 &  0 &
\end{array}
\]
Consider the variety $Z \subseteq \mathbb P(1, 1, 1, 1, 1, 2, 2, 2)$ with variables $\xi, u, y, z, \beta, x, w, \gamma$ where $Y_3 \to Z$ is the ample model of $\mathbb V(\xi)$. On the patch $Z_\xi$, we can substitute
$u = s_2 - \beta r_1$, $x = \beta s_2 - \gamma r_1$ and $z = \gamma - \beta^2 $, and compute that $Z_\xi$ is a hypersurface given by a weight $6$ polynomial, with a $cA_2$ singularity at $P_\xi \in Z_\xi$, of type at least~$2$ (see \cref{def:wei type}). These substitutions lift to $(Y_3)_\xi$, showing that $Y_3 \to Z$ is a $(3, 3, 2, 1)$-Kawakita blowup with centre~$P_\xi$. If the coefficients are general, namely when
\[
-2 e_\beta + 8 \beta^4 a_0 r_\beta - 2 \beta^2 b_\beta + 12 \beta^2 a_\beta^2
\in \mathbb C[y, \beta]
\]
is not a full square, where $r_\beta = r_1(y, -\beta^2)$, $e_\beta = e_3(y, -\beta^2)$, $a_\beta = a_1(y, -\beta^2)$ and $b_\beta = b_2(y, -\beta^2)$, then the point $P_\xi$ is exactly of type~$2$.

The variety $Z$ is isomorphic to a complete intersection $Z_{2, 4} \subseteq \mathbb P(1, 1, 1, 1, 1, 2)$ with variables $u, y, z, \beta, \xi, w$. We see this by substituting $x = u \beta - q_1 r_1$
and $\gamma = q_1 \xi + \beta^2$.
We find that $Z$ is isomorphic to $Z_{2, 4}\colon \mathbb V(-u \xi + s_2 - \beta r_1,\, h)$,
where
\[
\begin{aligned}
h & = -w^2 + \xi^2 q_1^2 - 2 e_3 \xi + \beta^4 + 2 b_0 q_1^2 r_1^2 - 4 \beta b_1 q_1 r_1 - 4 u \beta b_0 q_1 r_1 - 12 \beta^2 a_0 q_1 r_1 + 4 \xi b_2 q_1\\
  & - 16 \xi a_1^2 q_1 + 4 \beta \xi a_1 q_1 + 2 \beta^2 \xi q_1 + 2 \beta c_3 + 2 \beta^2 b_2 + 2 u \beta^2 b_1 + 2 u^2 \beta^2 b_0 + 4 \beta^3 a_1 + 4 u \beta^3 a_0\\
  & + (u \beta - q_1 r_1) C_Z + D_Z,
\end{aligned}
\]
where $C_Z = C_2(y, z, u)$ and $D_Z = D_4(y, z, u)$.
\end{proof}

\begin{remark}
We explain below how we found the embedding of~$X$. Using \cref{mai:con} and the coordinate change in $cA_7$ family~7.1, we can write a sextic double solid $\bar X$ with an isolated $cA_7$ in family~7.2 by
\[
\bar X\colon \mathbb V(f - 2 e_3 (t s_2 - \beta r_1),\, \beta - x t - q_1 r_1,\, \gamma - x \beta - q_1 s_2) \subseteq \mathbb P(1, 1, 1, 1, 2, 3, 3)
\]
with variables $x, y, z, t, \beta, \gamma, w$.

We construct a $(4, 4, 1, 1)$-Kawakita blowup $\bar Y_0 \to \bar X$. Define $\bar T_0$ by
\[
\begin{array}{ccc|ccccccc}
                 &  u & x & y & z & w & \gamma & \beta & t &\\
\Ta{\bar T_0\colon} &  0 & 1 & 1 & 1 & 3 & 3 & 2 & 1 & \Tb{.}\\
                 & -1 & 0 & 1 & 1 & 4 & 4 & 3 & 2 &
\end{array}
\]
Let $T_0 \to \mathbb P(1, 1, 1, 1, 2, 3, 3)$ be the ample model of $\mathbb V(x)$ and $Y_0 \subseteq T_0$ the strict transform of~$X$. Then $\bar Y_0$ is given by the ideal $I_{\bar Y} = (\bar g_1, \ldots, \bar g_5)$, where
\[
\begin{aligned}
\bar g_1 & = u g + 2 e_3 (\beta r_1 - t s_2), & \bar g_2 & = u \beta - q_1 r_1 - x t, & \bar g_3 & = u \gamma - q_1 s_2 - x \beta,\\
\bar g_4 & = x g + 2 e_3 (\gamma r_1 - \beta s_2), & \bar g_5 & = q_1 g + 2 e_3 (\beta^2 - t \gamma)).
\end{aligned}
\]
The morphism $\bar Y_0 \to \bar X$ is a $(4, 4, 1, 1)$-Kawakita blowup, as can be checked on the patch $(\bar Y_0)_x \to \bar X_x$.

Note that we do not prove that $I_{\bar Y}$ is saturated with respect to~$u$. In fact, the saturation will not be $I_Y$ if we do not use assume some generality conditions, similarly to $cA_6$ and $cA_7$ family~7.1. As a heuristic argument to see why $I_{\bar Y}$ might be saturated in the general case (``general'' meaning a Zariski open dense set of the parameter space), we can use computer algebra software, pretend that $a_i$, $b_i$, $c_i$, $d_i$, $q_1$, $r_1$, $s_2$, $e_3$ are algebraically independent variables of a polynomial ring over $\mathbb Q$ or $\mathbb Z_p$ for a large prime~$p$, and calculate that the saturation in that case indeed equals the ideal~$I_{\bar Y}$.

Similarly to the diagram $Y_0 \to W_0 \gets Y_1$ in the proof of \cref{thm:mod cA7-2}, the diagram $\bar Y_0 \to \bar{W}_0 \gets \bar Y_1$ is an Atiyah flop, provided $r_1$ and $q_1$ are coprime.

We show that $I_{\bar Y}$ does not 2-ray follow $\bar T_0$, namely that the diagram $\bar Y_1 \to \bar{W}_1 \gets Y_2$ contracts a curve and extracts a divisor. Acting by the matrix $\smat{4 & -3\\-1 & 1}$ on the action-matrix of $\bar T_0$, define $\bar T_1$ by
\[
\begin{array}{ccccc|ccccc}
                 &  u &  x & y & z & w & \gamma &  \beta &  t &\\
\Ta{\bar T_1\colon} &  3 &  4 & 1 & 1 & 0 & 0 & -1 & -2 & \Tb{,}\\
                 & -1 & -1 & 0 & 0 & 1 & 1 &  1 &  1 &
\end{array}
\]
and define $\bar Y_1 \subseteq \bar T_1$ by the zeros of $I_{\bar Y}$. We consider the toric flip $\bar T_1 \to \bar \bar{\mathcal W}_1 \gets \bar T_2$ and restrict it to $\bar Y_1 \to \bar{W}_1 \gets \bar Y_2$. Since $I_{\bar Y}$ is the zero ideal when restricting to $\mathbb V(u, x, y, z, \beta, t)$, the base $\mathbb P^1 \subseteq \bar \bar{\mathcal W}_1$ of the toric flip restricts to $\mathbb P^1 \subseteq \bar{W}_1$ with variables $w, \gamma$. The morphism $\bar Y_1 \to \bar{W}_1$ contracts a curve $\mathbb P^1$ to both of the points $[1, 1]$ and $[1, -1]$ in the base $\mathbb P^1 \subseteq \bar{W}_1$ and is an isomorphism elsewhere. The morphism $\bar{W}_1 \gets \bar Y_2$ extracts a curve $\mathbb P^1$ for every point in the base $\mathbb P^1 \subseteq \bar{W}_1$, so extracts a divisor on~$\bar Y_2$. The diagram $\bar Y_1 \to \bar{W}_1 \gets \bar Y_2$ is not a step in the 2-ray game of~$\bar Y_0$, so $I_{\bar Y}$ does not 2-ray follow $\bar T_0$. The reason for this was that the ideal $I_{\bar Y}$ is contained in $(u, x, y, z)$, so the surface $\mathbb V(u, x, y, z) \subseteq \bar T_2$ exists on $\bar Y_2$, but does not exist on $\bar T_1$.

We ``unproject'' $\bar g_1 = \bar g_4 = \bar g_5 = 0$ with respect to $u, x, y, z$ in $\bar Y_1 \subseteq \bar T_1$, to find a variety $Y_1 \subseteq T_1$. We explain below what we mean by this. We can write the system of equations $\bar g_1 = \bar g_4 = \bar g_5 = 0$ in the matrix form
\[
\pmat{
g & 0 & 0 & \beta r_1 - t s_2\\
0 & g & 0 & \gamma r_1 - \beta s_2\\
0 & 0 & g & \beta^2 - t \gamma
}
\pmat{u\\x\\q_1\\2 e_3} = \bm 0.
\]
If the above equations hold, then we have
\[
\frac{\abs{\pmat{
0 & 0 & \beta r_1 - t s_2\\
g & 0 & \gamma r_1 - \beta s_2\\
0 & g & \beta^2 - t \gamma
}}}{u}
= \frac{\abs{\pmat{
g & 0 & \beta r_1 - t s_2\\
0 & 0 & \gamma r_1 - \beta s_2\\
0 & g & \beta^2 - t \gamma
}}}{-x}
= \frac{\abs{\pmat{
g & 0 & \beta r_1 - t s_2\\
0 & g & \gamma r_1 - \beta s_2\\
0 & 0 & \beta^2 - t \gamma
}}}{q_1}
= \frac{\abs{\pmat{
g & 0 & 0\\
0 & g & 0\\
0 & 0 & g
}}}{-2 e_3}.
\]
Calculating the determinants and dividing by~$-g^2$, we find the equalities
\begin{equation} \label[eqs]{eqn:mod cA7-2 unproj}
\frac{t s_2 - \beta r_1}{u} = \frac{\beta s_2 - \gamma r_1}{x} = \frac{t \gamma - \beta^2}{q_1} = \frac{g}{2 e_3},
\end{equation}
between elements of the field of fractions of $\mathbb C[u, x, y, z, w, \gamma, \beta, t] / I_{\bar Y}$. Using the \cref{eqn:mod cA7-2 unproj}, we see that the morphism $\bar Y_1 \to Y_1$ given by
\[
[u, x, y, z, w, \gamma, \beta, t] \mapsto [u, x, y, z, w, \gamma, \beta, \frac{t s_2 - \beta r_1}{u}, t]
\]
is an isomorphism, where $Y_1$ is described in the proof of \cref{thm:mod cA7-2}.

The coordinate change $\bar Y_1 \to Y_1$ induces an isomorphism $\bar X \to X$, giving the variety~$X$.
\end{remark}

\subsection{\texorpdfstring
{$cA_7$ family 7.3 model}
{cA7 family 7.3 model}} \label{subsec:mod cA7-3}

\begin{proposition} \label{thm:mod cA7-3}
A Mori fibre space sextic double solid with a $cA_7$ singularity in family~7.3 satisfying \cref{def:mod generality conds} has a Sarkisov link to a degree $2$ del Pezzo fibration, starting with a $(4, 4, 1, 1)$-blowup of the $cA_7$ point and followed by two Atiyah flops.
\end{proposition}

\begin{proof}
We exhibit the diagram below.
\begin{equation*}
\begin{tikzcd}[column sep = small]
& \arrow[ld, "{(4, 4, 1, 1)}"'] Y_0 \arrow[rd, ""'] \arrow[rr, dashed, "{2 \times (1, 1, -1, -1)}"] & & Y_1 \arrow[ld, ""'] \arrow[r, "\sim"] & Y_2 \arrow[rd, "{\text{$\operatorname{dP_2}$-fibration}}"] \\
cA_7 \in X \subseteq \mathbb P(1^4, 3) \hspace{-3em} & & W_0 & & & \mathbb P^1
\end{tikzcd}
\end{equation*}

First, we define $X$ and a $(4, 4, 1, 1)$-Kawakita blowup $Y_0 \to X$. Any sextic double solid with an isolated $cA_7$ family~7.3 can be given by a bidegree $(6, 2)$ complete intersection
\[
X\colon \mathbb V(f, -\xi + t s_1 - q_2 - x t) \subseteq \mathbb P(1, 1, 1, 1, 2, 3)
\]
with variables $x, y, z, t, \xi, w$, where
\[
\begin{aligned}
f & = -w^2 + x^2 \xi^2 - 2 \xi e_4 + \xi^2 (s_1^2 + 4 a_1 s_1 + 2 x s_1 - 2 b_2 + 16 a_1^2 + 4 x a_1 + 8 \xi a_0)\\
  & + t (t s_1^4 + 4 t a_1 s_1^3 - 8 t^2 a_0 s_1^3 - 2 \xi s_1^3 + 2 t b_2 s_1^2 - 2 t^2 b_1 s_1^2 - 8 \xi a_1 s_1^2 + 24 t \xi a_0 s_1^2\\
  & + 12 x t^2 a_0 s_1^2 - 2 x \xi s_1^2 + 2 t c_3 s_1 + 4 t \xi b_1 s_1 + 4 x t^2 b_1 s_1 - 16 \xi a_1^2 s_1 - 4 x \xi a_1 s_1\\
  & - 24 \xi^2 a_0 s_1 - 24 x t \xi a_0 s_1 - 2 \xi c_3 - 4 x \xi b_2 - 2 \xi^2 b_1 - 4 x t \xi b_1 + 2 x^2 t^3 b_0 + 16 x \xi a_1^2\\
  & + 12 x \xi^2 a_0 + x t^2 C_2 + t D_4),
\end{aligned}
\]
where $C_i$, $D_i \in \mathbb C[y, z, t]$ are homogeneous of degree~$i$. Define
\[
\begin{array}{ccc|cccccccc}
            &  u & x & y & z & w & \xi & t &\\
\Ta{T_0\colon} &  0 & 1 & 1 & 1 & 3 & 2 & 1 & \Tb{.}\\
            & -1 & 0 & 1 & 1 & 4 & 4 & 2 &
\end{array}
\]
Define $\Phi\colon T_0 \to \mathbb P(1, 1, 1, 1, 2, 3)$ by the ample model of $\mathbb V(x)$, and define $Y_0$ as the strict transform of~$X$. Then, $Y_0$ is given by
\[
Y_0\colon \mathbb V(I_Y) \subseteq T_0 \ \text{ where } \ I_Y = (\Phi^*f / u^8,\, -u^2 \xi + u t s_1 - q_2 - x t),
\]
Using \cref{thm:wei cAn wt correct implies Kawakita blup}, we see that $Y_0 \to X$ is a $(4, 4, 1, 1)$-Kawakita blowup.

We describe the flop $Y_0 \to W_0 \gets Y_1$. Multiplying the action-matrix of $T_0$ by $\smat{1 & -1\\0 & 1}$, we find
\[
\begin{array}{ccc|cccccccc}
               &  u & x & y & z &  w &  \xi &  t &\\
\Ta{T_0 \cong} &  1 & 1 & 0 & 0 & -1 & -2 & -1 & \Tb{.}\\
               & -1 & 0 & 1 & 1 &  4 &  4 &  2 &
\end{array}
\]
The base of the flop is given by $\mathbb V(q_2) \subseteq \mathbb P^1 \subseteq W_0$. After a suitable coordinate change on $y, z$, we find $q_2 = yz$. Consider the flop over $\mathbb V(y)$, the flop over the other point is similar. Since $q_2$ and $e_4$ have no common divisor, on the patch where $z$ is non-zero, we can express $y$ and $\xi$ locally analytically equivariantly in terms of $u, x, t, w$. So, $Y_0 \to W_0 \gets Y_1$ is locally analytically two Atiyah flops.

The morphisms $Y_1 \to W_1 \gets Y_2$ are isomorphisms, since $w^2$ has a non-zero coefficient in~$\Phi^*f / u^8$.

We show that $Y_2$ is a degree $2$ del Pezzo fibration. Multiplying the original action-matrix of $T_0$ by the matrix $\smat{1 & 0\\2 & -1}$ with determinant~$-1$, we find
\[
\begin{array}{cccccc|ccccc}
            & u & x & y & z & w & \xi & t &\\
\Ta{T_2\colon} & 0 & 1 & 1 & 1 & 3 & 2 & 1 & \Tb{.}\\
            & 1 & 2 & 1 & 1 & 2 & 0 & 0 &
\end{array}
\]
The ample model of $\mathbb V(t)$ is
\[
\begin{aligned}
Y_2 & \to \mathbb P(2, 1)\\
[u, x, y, z, w, \xi, t] & \mapsto [\xi, t].
\end{aligned}
\]
Since $\mathbb P(2, 1)$ is isomorphic to $\mathbb P^1$, we see that $Y_2$ is a fibration onto~$\mathbb P^1$. On the patch $(Y_2)_t$, we can substitute $x = u s_1 - q_2 - u^2 \xi$,
to find that the general fibre is a weighted degree $4$ hypersurface in $\mathbb P(1, 1, 1, 2)$, so a degree $2$ del Pezzo surface.
\end{proof}

\subsection{\texorpdfstring
{$cA_8$ model}
{cA8 model}} \label{subsec:mod cA8}

\begin{proposition} \label{thm:mod cA8}
A Mori fibre space sextic double solid with a $cA_8$ singularity satisfying \cref{def:mod generality conds} has a Sarkisov link to a complete intersection $Z_{3, 3} \subseteq \mathbb P(1, 1, 1, 1, 1, 2)$ with a $cD_4$ singularity, starting with a $(5, 4, 1, 1)$-blowup of the $cA_8$ point, followed by a $(4, 1, 1, -1, -2; 2)$-flip, and finally a $(3, 2, 2, 1, 5)$-blowdown to the $cD_4$ singularity. Under further generality conditions, the singular locus of $Z$ consists of 3 points, namely the $cD_4$ point, the $1/2(1, 1, 1)$ singularity and an ordinary double point.
\end{proposition}

\begin{proof}
We exhibit the diagram below.
\begin{equation*}
\begin{tikzcd}[column sep = small]
& \arrow[ld, "{(1, 1, 4, 5)}"'] Y_0 \arrow[r, "\sim"] & Y_1 \arrow[rd, ""] \arrow[rr, dashed, "{(4, 1, 1, -1, -2; 2)}"] & & Y_2 \arrow[ld, ""] \arrow[r, "\sim"] & Y_3 \arrow[rd, "{(3, 2, 2, 1, 5)}"] \\
cA_8 \in X_6 \subseteq \mathbb P(1^4, 3) \hspace{-4em} & & & W_1 & & & \hspace{-1em} cD_4 \in Z_{3, 3} \subseteq \mathbb P(1^5, 2)
\end{tikzcd}
\end{equation*}

First, we describe~$X$ and the weighted blowup $Y_0 \to X$. A sextic double solid with a $cA_8$ singularity can be given by a multidegree $(6, 2, 3)$ complete intersection
\[
X\colon \mathbb V(f,\, \beta - x t - r_2,\, \gamma - x \beta - s_3) \subseteq \mathbb P(1, 1, 1, 1, 2, 3, 3),
\]
with variables $x, y, z, t, \beta, \gamma, \xi$ where
\[
\begin{aligned}
f & = 8 \beta^3 (A_0 - a_0) + \xi (-\xi + 2 \gamma - 8 t A_0 r_2 + 2 t b_2 - 4 t a_1^2 + 4 \beta a_1)\\
  & + t (-16 t \beta A_0^2 r_2 + 2 t \beta c_2 + 4 t \gamma b_1 - 2 \beta^2 b_1 - 2 t \beta^2 b_0 + 4 x t^2 \beta b_0 - 8 t \gamma a_0 a_1 + 8 \beta^2 a_0 a_1\\
  & + 12 \beta \gamma a_0 - 2 t \gamma B_1 + 2 \beta^2 B_1 + 16 t \beta^2 A_0^2 - 16 x t^2 \beta A_0^2 - 8 \beta \gamma A_0 + x t^3 C_1 + t^2 D_3)
\end{aligned}
\]
where $C_i$, $D_i \in \mathbb C[y, z, t]$ are homogeneous of degree~$i$. Note that $B_1 \in \mathbb C[y, z]$. Define
\[
\begin{array}{ccc|cccccccc}
            &  u & x & y & z & \gamma & \beta & \xi & t &\\
\Ta{T_0\colon} &  0 & 1 & 1 & 1 & 3 & 2 & 3 & 1 & \Tb{.}\\
            & -1 & 0 & 1 & 1 & 4 & 3 & 5 & 2 &
\end{array}
\]
Let $\Phi\colon T_0 \to \mathbb P(1, 1, 1, 1, 2, 3, 3)$ be the ample model of $\mathbb V(x)$ and let $Y_0 \subseteq T_0$ be the strict transform of~$X$. Then $Y_0$ is given by
\[
Y_0\colon \mathbb V(I_Y) \subseteq T_0 \ \text{ where } \ I_Y = \mleft( \Phi^*f / u^9,\, u \beta - x t - r_2,\, u \gamma - x \beta - s_3 \mright),
\]
and $Y_0 \to X$ is a $(5, 4, 1, 1)$-Kawakita blowup.

The first diagram in the 2-ray game of $T_0$ restricts to a isomorphisms $Y_0 \to W_0 \gets Y_1$, since $r_2$ and $s_3$ are coprime.

The second diagram in the 2-ray game of $T_0$ restricts to a $(4, 1, 1, -1, -2; 2)$-flip $Y_1 \to W_1 \gets Y_2$. Define the toric variety $T_1$ by multiplying the action matrix of $T_0$ by the matrix $\smat{4 & -3\\3 & -2}$,
\[
\begin{array}{ccccc|ccccc}
            & u & x & y & z & \gamma &  \beta &  \xi &  t &\\
\Ta{T_1\colon} & 3 & 4 & 1 & 1 & 0 & -1 & -3 & -2 & \Tb{.}\\
            & 2 & 3 & 1 & 1 & 1 &  0 & -1 & -1 &
\end{array}
\]
On the patch where $\gamma$ is non-zero, we have $u = x \beta + s_3$
and we can write $\xi$ locally analytically equivariantly in terms of $x, y, z, \beta, t$. We are left with the hypersurface given by $x \beta^2 + \beta s_3 - x t - r_2$
in $\mathbb C^5$ with variables $x, y, z, \beta, t$ with weights $(4, 1, 1, -1, -2)$. The polynomial contains $xt$ and $r_2$, so this corresponds to case~(1) in \cite[Theorem~8]{Bro99}, a $(4, 1, 1, -1, -2; 2)$-flip. Similarly to \cref{thm:mod cA7-1}, the flip contracts a curve containing a $1/4(1, 1, 3)$ singularity, and extracts a curve containing a $1/2(1, 1, 1)$ singularity and a $cA_1$ singularity, which is an ordinary double point if $r_2$ is not a square and is a 3-fold $A_2$ singularity otherwise. The $cA_1$ singularity on $Y_2$ is at $[0, 0, 0, 0, 1, 1, -2 a_0, 1]$.

The third diagram in the 2-ray game of $T_0$ restricts to isomorphisms $Y_2 \to W_2 \gets Y_3$, under \cref{def:mod generality conds}, namely that $a_0 \neq A_0$. On the patch where $\beta$ is non-zero, the base of the toric flip restricts to $\mathbb V(A_0 - a_0, u, x, y, z, \gamma, \xi, t) \subseteq W_2$.

We describe the weighted blowdown $Y_3 \to Z$. Multiplying the action matrix of $T_0$ by the matrix $\smat{5 & -3\\2 & -1}$, the toric variety $T_3$ is given by
\[
\begin{array}{ccccccc|ccc}
            & u & x & y & z & \gamma & \beta & \xi &  t &\\
\Ta{T_3\colon} & 3 & 5 & 2 & 2 & 3 & 1 & 0 & -1 & \Tb{.}\\
            & 1 & 2 & 1 & 1 & 2 & 1 & 1 &  0 &
\end{array}
\]
The ample model of $\mathbb V(\xi)$ is $Y_3 \to Z$ where $Z$ is the tridegree $(3, 2, 3)$ complete intersection
\[
Z\colon \mathbb V(h,\, u \beta - x - r_2,\, u \gamma - x \beta - s_3) \subseteq \mathbb P(1, 1, 1, 1, 1, 2, 2)
\]
with variables $u, y, z, \beta, \xi, x, \gamma$, where
\[
\begin{aligned}
h & = 8 \beta^3 (A_0 - a_0) + \xi (-u \xi + 2 \gamma - 8 A_0 r_2 + 2 b_2 - 4 a_1^2 + 4 \beta a_1)\\
  & - 16 \beta A_0^2 r_2 + 2 \beta c_2 + 4 \gamma b_1 - 2 \beta^2 b_1 - 2 u \beta^2 b_0 + 4 x \beta b_0 - 8 \gamma a_0 a_1 + 8 \beta^2 a_0 a_1\\
  & + 12 \beta \gamma a_0 - 2 \gamma B_1 + 2 \beta^2 B_1 + 16 u \beta^2 A_0^2 - 16 x \beta A_0^2 - 8 \beta \gamma A_0 + x C_Z + D_Z
\end{aligned}
\]
where $C_Z = C_1(y, z, u)$ and $D_Z = D_3(y, z, u)$. Substituting $x = u \beta - r_2$,
we see that $Z$ is isomorphic to a complete intersection of bidegree $(3, 3)$ in $\mathbb P(1^5, 2)$ with variables $u, y, z, \beta, \xi, \gamma$. The variety $Z$ has a $cA_1$ singularity at $[0, 0, 0, 1, -2 a_0, 1]$. We can compute that the point $P_\xi \in Z$ is a $cD_4$ point, by showing the complex analytic space germ $(Z, P_\xi)$ is isomorphic to $(\mathbb V(u^2 + 2 \beta r_2 - s_3 + \mathrm{h.o.t}), \bm 0) \subseteq (\mathbb C^4, \bm 0)$
with variables $u, \beta, y, z$, where $\mathrm{h.o.t}$ are higher order terms in $y, z, \beta$. We can compute that $Y_3 \to Z$ is the divisorial contraction to a $cD_4$ point described in \cite[Theorem~2.3]{Yam18}.
\end{proof}

\paragraph{Acknowledgment}
The author would like to thank Jihun Park for discussions on factoriality.

\newcommand{\etalchar}[1]{$^{#1}$}
\providecommand{\bysame}{\leavevmode\hbox to3em{\hrulefill}\thinspace}
\providecommand{\MR}{\relax\ifhmode\unskip\space\fi MR }
\providecommand{\MRhref}[2]{%
  \href{http://www.ams.org/mathscinet-getitem?mr=#1}{#2}
}
\providecommand{\href}[2]{#2}

\vspace{0.5\baselineskip}
\ShowAffiliations{\\[1\baselineskip]}%


\begin{thebibliography}{BCHM10}

\bibitem[Ahm17]{Ahm17}
Hamid Ahmadinezhad, \emph{On pliability of del {P}ezzo fibrations and {C}ox
  rings}, J. Reine Angew. Math. \textbf{723} (2017), 101--125.

\bibitem[AK16]{AK16}
Hamid Ahmadinezhad and Anne-Sophie Kaloghiros, \emph{Non-rigid quartic
  3-folds}, Compos. Math. \textbf{152} (2016), no.~5, 955--983.

\bibitem[AO18]{AO18}
Hamid Ahmadinezhad and Takuzo Okada, \emph{Birationally rigid {P}faffian {F}ano
  3-folds}, Algebr. Geom. \textbf{5} (2018), no.~2, 160--199.

\bibitem[AZ16]{AZ16}
Hamid Ahmadinezhad and Francesco Zucconi, \emph{Mori dream spaces and
  birational rigidity of {F}ano 3-folds}, Adv. Math. \textbf{292} (2016),
  410--445.

\bibitem[Bar96]{Bar96}
W.~Barth, \emph{Two projective surfaces with many nodes, admitting the
  symmetries of the icosahedron}, J. Algebraic Geom. \textbf{5} (1996), no.~1,
  173--186.

\bibitem[Bas06]{Bas06}
A.~B. Basset, \emph{The maximum number of double points on a surface}, Nature
  \textbf{73} (1906), no.~1889, 246--246.

\bibitem[BB13]{BB13}
Gavin Brown and Jaros\l~aw Buczy\'{n}ski, \emph{Maps of toric varieties in
  {C}ox coordinates}, Fund. Math. \textbf{222} (2013), no.~3, 213--267.

\bibitem[BCHM10]{BCHM10}
Caucher Birkar, Paolo Cascini, Christopher~D. Hacon, and James McKernan,
  \emph{Existence of minimal models for varieties of log general type}, J.
  Amer. Math. Soc. \textbf{23} (2010), no.~2, 405--468.

\bibitem[Bro99]{Bro99}
Gavin Brown, \emph{Flips arising as quotients of hypersurfaces}, Math. Proc.
  Cambridge Philos. Soc. \textbf{127} (1999), no.~1, 13--31.

\bibitem[BZ10]{BZ10}
Gavin Brown and Francesco Zucconi, \emph{Graded rings of rank 2 {S}arkisov
  links}, Nagoya Mathematical Journal \textbf{197} (2010), 1–44.

\bibitem[CC82]{CC82}
F.~Catanese and G.~Ceresa, \emph{Constructing sextic surfaces with a given
  number {$d$} of nodes}, J. Pure Appl. Algebra \textbf{23} (1982), no.~1,
  1--12.

\bibitem[CDG{\etalchar{+}}94]{CDG+94}
F.~Campana, G.~Dethloff, H.~Grauert, Th. Petemell, and R.~Remmert,
  \emph{Several complex variables. {VII}}, Encyclopaedia of Mathematical
  Sciences, vol.~74, Springer-Verlag, Berlin, 1994, Sheaf-theoretical methods
  in complex analysis, A reprint of {\textit{Current problems in mathematics.
  Fundamental directions. Vol. 74}} (Russian), Vseross. Inst. Nauchn. i Tekhn.
  Inform. (VINITI), Moscow.

\bibitem[CG17]{CG17}
Ivan Cheltsov and Mikhail Grinenko, \emph{Birational rigidity is not an open
  property}, Bull. Korean Math. Soc. \textbf{54} (2017), no.~5, 1485--1526.

\bibitem[CLS11]{CLS11}
David~A. Cox, John~B. Little, and Henry~K. Schenck, \emph{Toric varieties},
  Graduate Studies in Mathematics, vol. 124, American Mathematical Society,
  Providence, RI, 2011.

\bibitem[CM04]{CM04}
Alessio Corti and Massimiliano Mella, \emph{Birational geometry of terminal
  quartic 3-folds. {I}}, Amer. J. Math. \textbf{126} (2004), no.~4, 739--761.

\bibitem[Cor95]{Cor95}
Alessio Corti, \emph{Factoring birational maps of threefolds after {S}arkisov},
  J. Algebraic Geom. \textbf{4} (1995), no.~2, 223--254.

\bibitem[Cox95]{Cox95}
David~A. Cox, \emph{The homogeneous coordinate ring of a toric variety}, J.
  Algebraic Geom. \textbf{4} (1995), no.~1, 17--50.

\bibitem[CP10]{CP10}
Ivan Cheltsov and Jihun Park, \emph{Sextic double solids}, Cohomological and
  Geometric Approaches to Rationality Problems: New Perspectives (Fedor
  Bogomolov and Yuri Tschinkel, eds.), Birkh{\"a}user Boston, Boston, 2010,
  pp.~75--132.

\bibitem[CP17]{CP17}
\bysame, \emph{Birationally rigid {F}ano threefold hypersurfaces}, Mem. Amer.
  Math. Soc. \textbf{246} (2017), no.~1167, v+117.

\bibitem[CPR00]{CPR00}
Alessio Corti, Aleksandr Pukhlikov, and Miles Reid, \emph{Fano {$3$}-fold
  hypersurfaces}, Explicit birational geometry of 3-folds, London Math. Soc.
  Lecture Note Ser., vol. 281, Cambridge Univ. Press, Cambridge, 2000,
  pp.~175--258.

\bibitem[CS19]{CS19}
Ivan Cheltsov and Constantin Shramov, \emph{Finite collineation groups and
  birational rigidity}, Selecta Math. (N.S.) \textbf{25} (2019), no.~5, Paper
  No. 71.

\bibitem[dF17]{dF17}
Tommaso de~Fernex, \emph{Birational rigidity of singular {F}ano hypersurfaces},
  Ann. Sc. Norm. Super. Pisa Cl. Sci. (5) \textbf{17} (2017), no.~3, 911--929.

\bibitem[Eis95]{Eis95}
David Eisenbud, \emph{Commutative algebra: With a view toward algebraic
  geometry}, Graduate Texts in Mathematics, Springer, 1995.

\bibitem[EP18]{EP18}
Thomas Eckl and Aleksandr Pukhlikov, \emph{Effective birational rigidity of
  {F}ano double hypersurfaces}, Arnold Math. J. \textbf{4} (2018), no.~3-4,
  505--521.

\bibitem[GLS07]{GLS07}
G.-M. Greuel, C.~Lossen, and E.~Shustin, \emph{Introduction to singularities
  and deformations}, Springer Monographs in Mathematics, Springer, Berlin,
  2007.

\bibitem[HK00]{HK00}
Yi~Hu and Sean Keel, \emph{Mori dream spaces and {GIT}.}, Michigan Math. J.
  \textbf{48} (2000), no.~1, 331--348.

\bibitem[HM13]{HM13}
Christopher~D. Hacon and James McKernan, \emph{The {S}arkisov program}, J.
  Algebraic Geom. \textbf{22} (2013), no.~2, 389--405.

\bibitem[IM71]{IM71}
V.~A. Iskovskih and Ju.~I. Manin, \emph{Three-dimensional quartics and
  counterexamples to the {L}üroth problem}, Mat. Sb. (N.S.) \textbf{86(128)}
  (1971), 140--166.

\bibitem[IP99]{IP99}
V.A. Iskovskikh and Yu.G. Prokhorov, \emph{Algebraic geometry {V}: {F}ano
  varieties}, Algebraic geometry, Springer, 1999.

\bibitem[Isk80]{Isk80}
V.~A. Iskovskikh, \emph{Birational automorphisms of three-dimensional algebraic
  varieties}, Journal of Soviet Mathematics \textbf{13} (1980), no.~6,
  815--868.

\bibitem[JR97]{JR97}
David~B. Jaffe and Daniel Ruberman, \emph{A sextic surface cannot have {$66$}
  nodes}, J. Algebraic Geom. \textbf{6} (1997), no.~1, 151--168.

\bibitem[Kaw88]{Kaw88}
Yujiro Kawamata, \emph{Crepant blowing-up of {$3$}-dimensional canonical
  singularities and its application to degenerations of surfaces}, Ann. of
  Math. (2) \textbf{127} (1988), no.~1, 93--163.

\bibitem[Kaw96]{Kaw96}
\bysame, \emph{Divisorial contractions to {$3$}-dimensional terminal quotient
  singularities}, Higher-dimensional complex varieties ({T}rento, 1994), de
  Gruyter, Berlin, 1996, pp.~241--246.

\bibitem[Kaw03]{Kaw03}
Masayuki Kawakita, \emph{General elephants of three-fold divisorial
  contractions}, J. Amer. Math. Soc. \textbf{16} (2003), no.~2, 331--362, see
  \url{https://www.kurims.kyoto-u.ac.jp/~masayuki/Website/Documents/erratum-03.pdf}
  for erratum.

\bibitem[KM92]{KM92}
J\'{a}nos Koll\'{a}r and Shigefumi Mori, \emph{Classification of
  three-dimensional flips}, J. Amer. Math. Soc. \textbf{5} (1992), no.~3,
  533--703.

\bibitem[KM98]{KM98}
\bysame, \emph{Birational geometry of algebraic varieties}, Cambridge Tracts in
  Mathematics, vol. 134, Cambridge University Press, Cambridge, 1998, With the
  collaboration of C. H. Clemens and A. Corti, Translated from the 1998
  Japanese original.

\bibitem[Kol98]{Kol98}
János Kollár, \emph{Real algebraic threefolds. {I}. {T}erminal
  singularities}, Collect. Math. \textbf{49} (1998), no.~2-3, 335--360,
  Dedicated to the memory of Fernando Serrano.

\bibitem[Kry18]{Kry18}
Igor Krylov, \emph{Birational geometry of del {P}ezzo fibrations with terminal
  quotient singularities}, J. Lond. Math. Soc. (2) \textbf{97} (2018), no.~2,
  222--246.

\bibitem[Liu06]{Liu06}
Qing Liu, \emph{Algebraic geometry and arithmetic curves. {Transl}. by {Reinie}
  {Ern{\'e}}}, Oxf. Grad. Texts Math., vol.~6, Oxford: Oxford University Press,
  2006 (English).

\bibitem[Mav99]{Mav99}
Anvar~R. Mavlyutov, \emph{Cohomology of complete intersections in toric
  varieties}, Pacific J. Math. \textbf{191} (1999), no.~1, 133--144.

\bibitem[MY82]{MY82}
John~N. Mather and Stephen S.~T. Yau, \emph{Classification of isolated
  hypersurface singularities by their moduli algebras}, Inventiones
  mathematicae \textbf{69} (1982), no.~2, 243--251.

\bibitem[Nam97]{Nam97}
Yoshinori Namikawa, \emph{Smoothing {F}ano {$3$}-folds}, J. Algebraic Geom.
  \textbf{6} (1997), no.~2, 307--324.

\bibitem[Oka14]{Oka14}
Takuzo Okada, \emph{Birational {M}ori fiber structures of {$\mathbb{Q}$}-{F}ano
  3-fold weighted complete intersections}, Proc. Lond. Math. Soc. (3)
  \textbf{109} (2014), no.~6, 1549--1600.

\bibitem[PR04]{PR04}
Stavros~Argyrios Papadakis and Miles Reid, \emph{Kustin-{M}iller unprojection
  without complexes}, J. Algebraic Geom. \textbf{13} (2004), no.~3, 563--577.

\bibitem[Pro18]{Pro18}
Yu.~G. Prokhorov, \emph{The rationality problem for conic bundles}, Uspekhi
  Mat. Nauk \textbf{73} (2018), no.~3(441), 3--88.

\bibitem[Puk97]{Puk97}
A.~V. Pukhlikov, \emph{Birational automorphisms of double spaces with
  singularities}, J. Math. Sci. (New York) \textbf{85} (1997), no.~4,
  2128--2141, Algebraic geometry, 2.

\bibitem[Rei80]{Rei80}
Miles Reid, \emph{Canonical {$3$}-folds}, Journ\'ees de {G}\'eometrie
  {A}lg\'ebrique d'{A}ngers, {J}uillet 1979/{A}lgebraic {G}eometry, {A}ngers,
  1979, Sijthoff \& Noordhoff, Alphen aan den Rijn---Germantown, Md., 1980,
  pp.~273--310.

\bibitem[Rei83]{Rei83}
\bysame, \emph{Minimal models of canonical {$3$}-folds}, Algebraic varieties
  and analytic varieties ({T}okyo, 1981), Adv. Stud. Pure Math., vol.~1,
  North-Holland, Amsterdam, 1983, pp.~131--180.

\bibitem[Rei91]{Rei91}
\bysame, \emph{Birational geometry of 3-folds according to {S}arkisov},
  preprint (1991).

\bibitem[Rei92]{Rei92}
\bysame, \emph{What is a flip}, unpublished manuscript of Utah seminar (1992).

\bibitem[Rei00]{Rei00}
\bysame, \emph{Graded rings and birational geometry}, Proc. of algebraic
  geometry symposium (2000), 1--72.

\bibitem[Sar89]{Sar89}
V.~G. Sarkisov, \emph{Birational maps of standard {$\mathbb{Q}$}-{F}ano
  fiberings}, I. V. Kurchatov Institute Atomic Energy preprint (1989).

\bibitem[Tho72]{Tho72}
René Thom, \emph{Stabilité structurelle et morphogénèse}, W. A. Benjamin,
  Inc., Reading, Mass., 1972, Essai d'une théorie générale des modèles,
  Mathematical Physics Monograph Series.

\bibitem[Wah98]{Wah98}
Jonathan Wahl, \emph{Nodes on sextic hypersurfaces in {$\mathbb P^3$}}, J.
  Differential Geom. \textbf{48} (1998), no.~3, 439--444.

\bibitem[Yam18]{Yam18}
Yuki Yamamoto, \emph{Divisorial contractions to {$cDV$} points with discrepancy
  greater than 1}, Kyoto J. Math. \textbf{58} (2018), no.~3, 529--567.

\bibitem[Zha06]{Zha06}
Qi~Zhang, \emph{Rational connectedness of log {${\bf Q}$}-{F}ano varieties}, J.
  Reine Angew. Math. \textbf{590} (2006), 131--142.

\end{thebibliography}
\end{document}